\documentclass[11pt]{amsart}
\usepackage[T1]{fontenc}
\usepackage{txfonts}
\usepackage{hyperref}
\usepackage{times,txfonts}
\usepackage{amssymb,verbatim}
\usepackage{mathrsfs}
\usepackage{wrapfig}
\usepackage{subfig,color}
\usepackage{a4wide}
\usepackage{tikz}
\usetikzlibrary{arrows}

\newcommand\QQQ{\mathcal Q}
\newcommand\B{{\bf B}}
\newcommand\CCC{\mathcal C}
\newcommand{\R}{\mathbb{R}}
\newcommand{\C}{\mathbb{C}}
\renewcommand\P{\mathbb{P}}
\newcommand\Z{\mathbb{Z}}
\newcommand\Prym{\mathrm{Prym}}
\newcommand{\N}{\mathbb{N}}
\newcommand{\Q}{\mathbb{Q}}
\renewcommand{\H}{\mathcal{H}}
\newcommand{\SL}{{\rm SL}}
\newcommand{\GL}{{\rm GL}}
\newcommand{\s}{\sigma}
\newcommand{\PrymD}{\Omega E_D}
\newcommand{\XX}{{\mathcal X}}

\newcommand{\DS}{\displaystyle}
\newcommand{\sst}{\scriptstyle}
\newcommand{\ie}{{\em i.e }}

\newcommand{\ol}{\overline}

\newcommand{\dist}{\mathbf{d}}
\newcommand{\Ord}{\mathcal{O}}

\newcommand{\eps}{\varepsilon}
\newcommand{\Mod}{\mathcal{M}}

\newcommand{\Id}{\mathrm{Id}}

\newcommand{\Orb}{\mathcal{O}}
\newcommand{\Aa}{\textrm{Area}}

\newcommand{\lbd}{\lambdaup}

\newcommand{\inv}{\tau}

\newcommand{\ra}{\rightarrow}

\newtheorem{Theorem}{Theorem}[section]
\newtheorem{Corollary}[Theorem]{Corollary}
\newtheorem{Lemma}[Theorem]{Lemma}
\newtheorem{Proposition}[Theorem]{Proposition}
\newtheorem{Remark}[Theorem]{Remark}
\newtheorem{Definition}[Theorem]{Definition}

\newtheorem{Claim}{Claim}
\newtheorem*{Question}{Question}

\begin{document}
\title[$\GL^+(2,\R)$-orbits of Prym eigenforms]  
{$\GL^+(2,\R)$-orbits in Prym eigenform loci}

\author{Erwan Lanneau and Duc-Manh Nguyen}

\begin{abstract}
This paper is  devoted to the classification of $\GL^+(2,\R)$-orbit 
closures of surfaces in the  intersection of the  Prym   eigenform  locus 
with various strata of Abelian differentials. 
We show that the following dichotomy holds: an orbit is either closed or dense in a connected component 
of the Prym eigenform locus.

The proof uses several topological properties of Prym eigenforms,  in particular the tools and the proof are independent  of the
recent results of Eskin-Mirzakhani-Mohammadi.

As an application we obtain a finiteness result for the number of closed $\GL^+(2,\R)$-orbits (not necessarily primitive)
in the Prym eigenform locus $\PrymD(2,2)$ for any fixed $D$ that is not a square.
\end{abstract}

\date{\today}

\address{
Institut Fourier, Universit\'e de Grenoble I, BP 74, 38402 Saint-Martin-d'H\`eres, France
}
\email{erwan.lanneau@ujf-grenoble.fr}
\address{
IMB Bordeaux-Universit\'e de Bordeaux, 351, Cours de la Lib\'eration, 33405 Talence Cedex, France
}
\email{duc-manh.nguyen@math.u-bordeaux1.fr}
\keywords{Real multiplication, Prym locus, Translation surface}

\maketitle
\setcounter{tocdepth}{1}
\tableofcontents

\section{Introduction}

For any $g\geq 1$ and any integer partition $\kappa=(\kappa_1,\dots,\kappa_r)$ of 
$2g-2$ we denote by $\H(\kappa)$ a stratum of the moduli space of pairs 
$(X,\omega)$, where $X$ is a Riemann surface of genus $g$  and $\omega$ is a 
holomorphic $1$-form having $r$ zeros with prescribed multiplicities $\kappa_1,\dots,\kappa_r$.
Analogously, one defines the strata of the moduli space of {\em quadratic differentials} 
$\QQQ(\kappa')$ having zeros and simple poles of multiplicities $\kappa'_1,\dots,\kappa'_s$ with $\sum_{i=1}^s \kappa'_s=4g-4$ (simple poles correspond to zeros of multiplicity $-1$).

The $1$-form $\omega$ defines a canonical flat metric on $X$ with conical singularities at 
the zeros of $\omega$. Therefore we will refer to points of $\H(\kappa)$ as flat surfaces or  
{\em translation surfaces}. The strata admit a natural action of the group $\GL^+(2,\R)$ that can 
be viewed as a generalization of the $\GL^+(2,\R)$ action on the space $\GL^+(2,\R)/\SL(2,\Z)$ of flat tori. 
For an introduction to this subject, we refer to  the excellent surveys~\cite{Masur2002,Zorich:survey}. \medskip

It has been discovered that many topological and dynamical properties of a translation surface 
can be revealed by its $\GL^+(2,\R)-$orbit closure.  The most spectacular example of this phenomenon is the case of {\em Veech 
surfaces}, or {\em lattice surfaces}, that is surfaces whose $\GL^+(2,\R)$-orbit is a closed subset in its stratum; 
for such surfaces, the famous {\em Veech dichotomy} holds: 
the linear flow in any direction is either periodic or uniquely ergodic.  

It follows from the foundation results of Masur and Veech that most of $\GL^+(2,\R)$ orbits are dense in their stratum. 
However, in any stratum there always exist surfaces whose orbits are closed, they arise from coverings of the standard 
flat torus and are commonly known as {\em square-tiled surfaces}.

During the past three decades, much effort has been made in order to obtain the list of possible $\GL^+(2,\R)$-orbit 
closures and to understand their structure as subsets of strata. So far, such a list is only known in genus two by the work 
of  McMullen~\cite{Mc07}, but the problem is wide open in higher genus, even though some breakthroughs have been 
achieved recently (see below).

In genus two the complex dimensions of the connected strata $\H(2)$ and $\H(1,1)$ are, respectively, $4$ and $5$.
In this situation, McMullen proved that if a $\GL^+(2,\R)$-orbit is not dense, then it  belongs to a {\em Prym eigenform locus}, 
which is a submanifold of complex dimension $3$. In this case, the orbit is either closed or dense in the whole 
Prym eigenform locus. These (closed) invariant submanifolds, that we denote by $\Omega E_{D}$, where $D$ is a 
{\em discriminant} (that is $D\in \N, \ D \equiv 0,1 \mod 4$),  are characterized by the following properties:
\begin{enumerate}
\item Every surface $(X,\omega) \in \Omega E_{D}$ has a holomorphic 
involution $\inv: X \rightarrow X$, and
\item The {\em Prym variety} $\Prym(X,\inv)=(\Omega^-(X,\inv))^*/H_1(X,\Z)^-$
admits a real multiplication by some quadratic order $\Ord_D:=\Z[x]/(x^2+bx+c), \ b,c \in \Z,  \ b^2-4c=D$.
\end{enumerate}
(where $\Omega^-(X,\inv)=\{\eta \in \Omega(X): \, \inv^*\eta=-\eta\}$). \medskip

Latter, McMullen proved the existence of similar loci is genus up to $5$, and showed that the intersection of such loci with the minimal strata give rise to some infinite families of primitive Veech surfaces (see~\cite{Mc1,Mc6,Lanneau:Manh:H4} for more details). \medskip

Recently, Eskin-Mirzakhani-Mohammadi~\cite{Eskin:preprint,EskMirMoh12}
have announced a proof of the conjecture that any $\GL^+(2,\R)$-orbit closure is an
affine invariant submanifold of $\H(\kappa)$.  This result is of great importance in view of the classification of 
orbit closures as it provides some very important characterizations of such subsets. 
However {\em a priori} this result does not allow us to construct explicitly such invariant submanifolds.

So far, most of $\GL^+(2,\R)$-invariant submanifolds  of a stratum are obtained from coverings of translation surfaces of lower genera. 
The only known examples of invariant submanifolds {\bf not} arising from this construction belong to one of the following families:
\begin{enumerate}
\item Primitive  Teichm\"uller curves (closed orbits), and
\item Prym eigenforms.
\end{enumerate}

This paper is concerned with the classification of $\GL^+(2,\R)-$orbit closures in the space of Prym eigenforms. 
To be more precise, for any non empty stratum $\QQQ(\kappa')$, there is a (local) affine map
$\phi : \QQQ^{(g')}(\kappa') \rightarrow \H^{(g)}(\kappa)$ that is given by the orientating double covering (here, the superscripts $g$ and $g'$ indicate the  genus of the corresponding Riemann surfaces). When 
$g-g'=2$, following McMullen~\cite{Mc6} we call the image of $\phi$ a Prym locus and denote it by $\Prym(\kappa)$.
Those Prym loci contain $\GL^+(2,\R)$-invariant suborbifolds denoted by $\Omega E_{D}(\kappa)$ (see Section~\ref{sec:background} for more precise definitions).
We will investigate the $\GL^+(2,\R)$-orbit closures in $\Omega E_{D}(\kappa)$. The first main theorem of this paper is the following.
\begin{Theorem}
\label{theo:main}
Let $(X,\omega)\in \Omega E_{D}(\kappa)$ be a Prym eigenform, where $\Omega E_{D}(\kappa)$ 
has complex dimension $3$ ({\em i.e.} $\Omega E_{D}(\kappa)$  is contained in one of the 
Prym loci in Table~\ref{tab:strata:list}). We denote by $\mathcal O$ its orbit under $\GL^+(2,\R)$. Then
\begin{enumerate}
\item Either $\mathcal O$ is closed (i.e. $(X,\omega)$ is a Veech surface), or
\item $\overline{\mathcal O}$ is a connected component of $\Omega E_{D}(\kappa)$.
\end{enumerate}
\end{Theorem}
\begin{table}[htb]
$$
\begin{array}{cc}
\begin{array}{|llc|}
\hline
\QQQ(\kappa')    & \Prym(\kappa) & g(X)  \\ 
\hline
\QQQ^{(0)}(-1^6,2)  & \Prym(1,1) \simeq \H(1,1)  &  2 \\
\QQQ^{(1)}(-1^3,1,2)  & \Prym(1,1,2)  &  3 \\
\QQQ^{(1)}(-1^4,4)  & \Prym(2,2)^\mathrm{odd}  &  3\\
\QQQ^{(2)}(-1^2,6)  & \Prym(3,3) \simeq \H(1,1) &  4 \\
\hline
\end{array} & 
\begin{array}{|llc|}
\hline
\QQQ(\kappa')    & \Prym(\kappa) & g(X)  \\ 
\hline
\QQQ^{(2)}(1^2,2)  & \Prym(1^2,2^2) \simeq \H(0^{2},2)  &  4 \\
\QQQ^{(2)}(-1,2,3)  & \Prym(1,1,4)  &  4 \\
\QQQ^{(2)}(-1,1,4)  & \Prym(2,2,2)^\mathrm{even}  &  4 \\
\QQQ^{(3)}(8)  & \Prym(4,4)^\mathrm{even}  &  5 \\
\hline
\end{array}
\end{array}
$$
\caption{
\label{tab:strata:list}
Prym loci for which the corresponding stratum of quadratic differentials has (complex) 
dimension $5$. The Prym eigenform loci $\PrymD(\kappa)$ has complex dimension $3$. Observe that
the stratum $\H(1,1)$ in genus $2$ is a particular case of Prym loci.}
\end{table}
\begin{Remark}\hfill
\begin{itemize}
\item[$\bullet$] The case $\PrymD(1,1)$ is  part of McMullen's classification in genus two, which is obtained via decompositions of translation surfaces of genus two into connected sums of two tori.

\item[$\bullet$] The classification of connected components of $\Omega E_{D}(2,2)$ and 
$\Omega E_{D}(1,1,2)$ will be addressed in a forthcoming paper~\cite{Lanneau:Manh:composantes} 
(see also~\cite{Lanneau:Manh:H4} for related work). The statement is the following:
for any discriminant $D\geq 8$ and $\kappa \in \{(2,2),(1,1,2)\}$, 
the locus $\Omega E_{D}(\kappa)$ is non-empty if and only if $D\equiv 0,1,4 \mod 8$,
and it is connected if $D\equiv 0,4 \mod 8$, and has two connected components otherwise.
\end{itemize}
\end{Remark}
Even though Theorem~\ref{theo:main} is a particular case of the results of Eskin-Mirzakhani-Mohammadi~\cite{Eskin:preprint, EskMirMoh12}, our proof is independent from these work, it is based essentially on a careful investigation of the geometric and topological properties of Prym eigenforms. It is also likely to us that the method introduced here can be generalized to yield Eskin-Mirzakhani-Mohammadi's result in invariant submanifolds which possess the {\em complete periodic} property (see Section~\ref{sec:complete:periodic}), for instance, the intersections of the Prym eigenform loci with other strata with higher dimension. \medskip

We will also prove a finiteness result for Teichm\"uller curves in the locus $\PrymD(2,2)^{\rm odd}$; 
this is our second main result:
\begin{Theorem}
\label{theo:main:finiteness}
If $D$ is not a square then there exist only finitely many closed $\GL^+(2,\R)$-orbits in $\PrymD(2,2)^{\rm odd}$.
\end{Theorem}
We end with a few remarks.
\begin{Remark}[On Theorem~\ref{theo:main:finiteness}]\hfill
\begin{itemize}
\item[$\bullet$] To the authors' knowledge, such finiteness results are not direct consequences of the work by Eskin-Mirzakhani-Mohammadi.
\item[$\bullet$] Our techniques allow us to get a similar result for the loci $\PrymD(1,1,2) \subset \Prym(1,1,2)$, but we will not 
include the proof in the present paper.
\item[$\bullet$]  In $\Prym(1,1)$ a stronger statement holds: there exist only finitely many $\GL^+(2,\R)$-closed orbits in $\underset{D \text{ not a square}}{\sqcup}\PrymD(1,1)$ (see~\cite{Mc5,Mc6bis}). We also notice that the same result for $\Prym(1,1,2)$  is proved in a forthcoming paper by the first author and M.~M\"oller (see~\cite{Lanneau:Moeller}). However, this is no longer true 
in $\Prym(2,2)^{\rm odd}$ as we will see in Theorem~\ref{thm:inf:Veech}.
\item[$\bullet$] As  by products of our approach, we obtain some evidences supporting the prediction that those Prym eigenform loci are quasiprojective varieties. 
\item[$\bullet$] Other finiteness results on Teichm\"uller curves have been obtained in other situations by different methods, see for instance \cite{Moller08, Bainbridge:Moeller12,Matheus:Wright13}.
\end{itemize}
\end{Remark}
\subsection*{Outline of the paper}

Here below we give a sketch of our  proofs of Theorem~\ref{theo:main} and 
Theorem~\ref{theo:main:finiteness}.
Before going into the details, we single out the relevant properties of 
$\Omega E_{D}(\kappa)$ for our purpose. In what 
follows $(X,\omega)$ will denote a surface in $\Omega E_{D}(\kappa)$.
%
\begin{enumerate}
\item \label{prop:kernel}
Each locus is preserved by the {\em kernel foliation}, that we will denote by 
$X+v$ for a sufficiently small vector $v\in \R^2$ (see Section~\ref{sec:kernel}).
In particular, up to action of $\textrm{GL}^+(2,\R)$, a neighborhood of $(X,\omega)$ in 
$\Omega E_D(\kappa)$ can be identified with the set
$$
\left\{(X,\omega)+v,\ v\in \B(\varepsilon)\right\}.
$$
\item \label{prop:cp}
Every surface in $\Omega E_D(\kappa)$ is completely periodic in the sense of Calta: the directions
of simple closed geodesics are completely periodic, and thus the surface is decomposed into cylinders in those directions. The number of cylinders is bounded by $g+|\kappa|-1$, where $|\kappa|$ is the number of zeros of $\omega$ (see Section~\ref{sec:background}).
\item  \label{prop:eq}
Assume that $(X,\omega)$ decomposes into cylinders in the horizontal direction, then the moduli of those cylinders are related by some equations with rational coefficients (see Corollary~\ref{cor:parabolic} and Lemma~\ref{lm:cyl:dec:3class:rel}).
\item \label{prop:stable}
The cylinder decomposition in a completely periodic direction is said to be {\em  stable} if there 
is no saddle connection connecting two different zeros in this direction. The stable periodic directions are {\em generic}
for the kernel foliation in the following sense: if the horizontal direction is stable for $(X,\omega)$ then there exists 
$\varepsilon >0$ such that for any $v$ with $v\in \B(\varepsilon)$, 
the horizontal direction is also periodic and stable on $X+v$. \\
If the horizontal direction is {\em unstable} then there exists $\varepsilon >0$ such that
for any $v=(x,y)$ with $v\in \B(\varepsilon)$ and $y\not =0$ 
the horizontal direction  is periodic and stable on $X+v$.
\end{enumerate}

The properties \eqref{prop:kernel}-\eqref{prop:cp}-\eqref{prop:eq} are explained in~\cite{Lanneau:Manh:cp}
(see Section 3.1 and Corollary 3.2, Theorem~1.5, Theorem~7.2, respectively). We will give more details on Property~\eqref{prop:stable}
in Section~\ref{sec:stable}.\medskip

We now give a sketch of the proof of our results. The first part of the paper (Sections~\ref{sec:kernel}-\ref{sec:proof:main:th})
is devoted to the proof of Theorem~\ref{theo:main}, while the second part (Sections~\ref{sec:surgeries}-\ref{sec:proof:finiteness}) 
is concerned with Theorem~\ref{theo:main:finiteness}.

\subsubsection*{Sketch of proof of Theorem~\ref{theo:main}.} Let $(X,\omega)\in \PrymD(\kappa)$ be a Prym eigenform and 
let ${\mathcal O}:=\GL^+(2,\R)\cdot(X,\omega)$ be the corresponding $\GL^+(2,\R)-$orbit. We will show that if $\mathcal{O}$ is not a closed subset in $\PrymD(\kappa)$ then it is dense in a connected component of $\PrymD(\kappa)$.

We first prove a  weaker version of Theorem~\ref{theo:main} (see Section~\ref{sec:weaker}) under the additional condition 
that there exists a completely periodic direction $\theta$ on $(X,\omega)$ that is not parabolic.  We start by applying the horocycle flow in that periodic direction, and use the classical Kronecker's theorem to show that the orbit closure contains the set $(X,\omega)+x\vec{v}$, where $\vec{v}$ is the unit vector in direction $\theta$, and $x \in (-\eps,\eps)$ with $\eps >0$ small enough. Next, we look at another periodic direction transverse to $\theta$, and apply the same argument to the surfaces $(X,\omega)+x\vec{v}$. It follows that $\overline{\mathcal{O}}$ contains a neighborhood of $(X,\omega)$, and hence for any $g\in \GL^+(2,\R)$, $\overline{\mathcal O}$ contains a neighborhood of  $g\cdot(X,\omega)$. Using this fact, we show that for a surface $(Y,\eta)$ in $\overline{\Orb}$ but not in $\Orb$, $\overline{\Orb}$ also contains a neighborhood of $(Y,\eta)$, from which we deduce that $\overline{\mathcal{O}}$ is an open subset of $\PrymD(\kappa)$. Hence $\overline{\Orb}$ must be a connected component of 
$\PrymD(\kappa)$. \medskip

In full generality, (see Section~\ref{sec:proof:main:th}) we show that if the orbit is not closed and all the periodic directions 
are parabolic, then it is also dense in a component of $\PrymD(\kappa)$. For this, we consider a surface $(Y,\eta) \in \overline{\mathcal{O}} \setminus \mathcal{O}$ for which the horizontal direction is periodic. From Property~\eqref{prop:kernel}, we see that there is a sequence $((X_n,\omega_n))_{n \in \N}$ of surfaces in $\Orb$ converging to $(Y,\eta)$ such that we can write $(X_n,\omega_n)=(Y,\eta)+(x_n,y_n)$, where $(x_n,y_n) \longrightarrow (0,0)$. Property~\eqref{prop:stable} then implies that the horizontal direction is also periodic for $(X_n,\omega_n)$, moreover, we can assume that the corresponding cylinder decomposition in $(X_n,\omega_n)$ is stable.

For any  $x \in (-\eps,\eps)$, where $\eps >0$ small enough, we show that, by choosing a suitable time, the orbit of the horocycle flow though $(X_n,\omega_n)$ contains a surface $(X_n,\omega_n)+(x_n,0)$ such that the sequence $(x_n)$ converges to $x$. As a consequence, we see that $\overline{\Orb}$ contains $(Y,\eta)+(x,0)$ for every $x \in (-\eps,\eps)$. We can now conclude that $\overline{\Orb}$ is a component of $\PrymD(\kappa)$ by the weaker version, which is proved previously.

\subsubsection*{Sketch of proof of Theorem~\ref{theo:main:finiteness}.} We first show a finiteness result up 
to the (real) kernel foliation for surfaces in $\PrymD(2,2)^{\rm odd}$ (see Theorem~\ref{thm:prot:unstable:dec}):
If $D$ is not a square then there exists a finite family $\mathbb P_D \subset \PrymD(2,2)^{\rm odd}$ 
such that for any $(X,\omega)\in \PrymD(2,2)^{\rm odd}$ with an unstable cylinder decomposition, 
up to rescaling by $\GL^+(2,\R)$, we have the following
$$
(X,\omega) = (X_k,\omega_k)+ (x,0) \qquad \textrm{for some } (X_k,\omega_k)\in \mathbb P_D.
$$
Compare to~\cite{Mc4,Lanneau:Manh:H4} where a similar result is established.

Now let us assume that there exists an infinite family, say ${\mathcal Y}=\bigcup_{i\in I} \GL^+(2,\R)\cdot(X_i,\omega_i)$,
of closed $\GL^+(2,\R)$-orbits, generated by Veech surfaces $(X_i,\omega_i), \, i \in I$.

By previous finiteness result, up to taking a subsequence, we assume that $(X_i,\omega_i)=(X,\omega)+(x_i,0)$
for some $(X,\omega)\in\mathbb{P}_D$, where $x_i$ belongs to a finite open interval $(a,b)$ which is independent of $i$ (see Theorem~\ref{thm:unstable:dec:collapse}). Up to taking a subsequence, one can assume that the sequence $(x_i)$ converges to some $x\in [a,b]$.
Hence the sequence $(X_i,\omega_{i})=(X,\omega)+(x_{i},0)$ converges to $(Y,\eta):=(X,\omega)+(x,0)$.\smallskip

\noindent If $x \in (a,b)$ then $(Y,\eta)$ belongs to $\PrymD(2,2)^{\rm odd}$, otherwise, that is $x\in \{a,b\}$, $(Y,\eta)$ belongs to one of the following loci $\PrymD(0,0,0), \PrymD(4)$, or $\Omega E_{D'}(2)^*$, with $D' \in \{D,D/4\}$ (see Section~\ref{sec:degenerating:surfaces}). Then by using a by-product of the proof of Theorem~\ref{theo:main}, replacing $\Orb$ by $\mathcal Y$
(see Theorem~\ref{thm:byproduct}) we obtain that $\mathcal Y$ is dense in a component 
of $\PrymD(2,2)^{\rm odd}$. We conclude with Theorem~\ref{thm:exist:open:noVeech} which asserts that  
the set of closed $\GL^+(2,\R)-$orbits  is not dense in any component of $\PrymD(2,2)^{\rm odd}$ when $D$ is not a square. \medskip

\subsection*{Acknowledgments}

We would like to thank Corentin Boissy, Pascal Hubert, John Smillie, and Barak Weiss for useful discussions.
We would also like to thank the Universit\'e de Bordeaux and Institut Fourier in Grenoble for the hospitality during 
the preparation of this work. Some of the research visits which made this collaboration possible were supported by 
the ANR Project GeoDyM. The authors are partially supported by the ANR Project GeoDyM.

\section{Background}\label{sec:background}

For an introduction to translation surfaces, and a nice survey on this
topic, see {\em e.g.}~\cite{Zorich:survey, Masur2002}. In this section we recall necessary background 
and relevant properties of $\Omega E_{D}(\kappa)$ for our purpose. For a general
reference on Prym eigenforms, see~\cite{Mc6} (the main properties are reminded below).
We will also review the kernel foliation, and complete periodicity.  \medskip

\noindent We will use the following notations along the paper: \\
$\B(\eps)=\{v \in \R^2, \: ||v||<\eps \}$, and \\
$\omega(\gamma):=\int_\gamma \omega$, for any $\gamma\in H_1(X,\Z)$.

\subsection{Prym loci and Prym eigenforms}
Let $X$ be a compact Riemann surface, and $\inv: X \rightarrow X$ be a holomorphic involution of $X$. 
We define  the {\em Prym variety} of $X$:
$$
\Prym(X,\inv)=(\Omega^-(X,\inv))^*/H_1(X,\Z)^-,
$$
\noindent where $\Omega^-(X,\inv)=\{\eta \in \Omega(X): \, \inv^*\eta=-\eta\}$. It is a sub-Abelian variety of the 
Jacobian variety $\mathrm{Jac}(X):=\Omega(X)^*/H_1(X,\Z)$. \medskip

For any integer vector $\kappa=(k_1,\dots,k_n)$ with nonnegative entries, we denote by 
$\Prym(\kappa) \subset \H(\kappa)$ the subset of pairs 
$(X,\omega)$ such that there exists an involution $\inv:  X \ra X$ satisfying  $\inv^*\omega=-\omega$, and 
$\dim_\C\Omega^-(X,\inv)=2$. Following McMullen~\cite{Mc6}, we will call an element of $\Prym(\kappa)$ a {\em Prym form}. For instance, in genus two, one has $\Prym(2) \simeq \H(2)$ and $\Prym(1,1) \simeq \H(1,1)$ (the Prym involution being the hyperelliptic involution). \medskip

Let $Y$ be the quotient of $X$ by the Prym involution (here $g(Y)=g(X)-2$) and $\pi$ the corresponding (possibly ramified) double covering from 
$X$ to $Y$. By push forward, there exists a meromorphic quadratic differential $q$ on 
$Y$ (with at most simple poles) so that $\pi^*q=\omega^2$. Let $\kappa'$ be 
the integer vector that records the orders of the zeros and poles of $q$. Then there is a  $\GL^+(2,\R)$-equivariant bijection
between $\QQQ(\kappa')$ and $\Prym(\kappa)$~\cite[p.~6]{Lanneau04}. \medskip

All the strata of quadratic differentials of dimension $5$ are recorded in Table~\ref{tab:strata:list}. 
It turns out that  the corresponding Prym varieties have complex  dimension two (\ie if $(X,\omega)$ 
is the orientating double covering of $(Y,q)$ then $g(X)-g(Y)=2$).  \medskip

We now give the definition of Prym eigenforms. Recall that a quadratic order 
is a ring isomorphic to $\Ord_D = \Z[X]/(X^2+bX+c)$, where $D = b^2-4c>0$
(quadratic orders being classified by their discriminant $D$).
\begin{Definition}[Real multiplication]
\label{def:Real:mult}
Let $A$ be an Abelian variety of dimension $2$. We say that $A$ admits a real 
multiplication by $\Ord_D$ if there exists an injective homomorphism 
$\mathfrak{i}: \Ord_D \ra \mathrm{End}(A)$, such that $\mathfrak{i}(\Ord_D)$ 
is a self-adjoint, proper subring of $\mathrm{End}(A)$ ({\em i.e.} for any $f\in \mathrm{End}(A)$, if there exists 
$n\in \Z\backslash \{0\}$ such that $nf\in \mathfrak{i}(\Ord_D)$ then $f \in \mathfrak{i}(\Ord_D)$).
\end{Definition}

\begin{Definition}[Prym eigenform]
\label{def:Prym:eig:form}
For any quadratic discriminant $D>0$,  we denote by $\PrymD(\kappa)$ the set of 
$(X,\omega) \in \Prym(\kappa)$ such that  $\dim_{\C}\Prym(X,\inv)=2$, $\Prym(X,\inv)$ admits a multiplication by $\Ord_D$, 
and $\omega$  is an eigenvector of $\Ord_D$. Surfaces in $\PrymD(\kappa)$ are called {\em Prym eigenforms}.
\end{Definition}
Prym eigenforms do exist in each Prym locus described in Table~\ref{tab:strata:list}, as  real multiplications arise 
naturally with pseudo-Anosov homeomorphisms commuting  with  $\inv$ 
(see~\cite{Mc6}).  

\subsection{Periodic directions and Cylinder decompositions}
\label{sec:cylinders}

We collect here several results concerning surfaces having a decomposition into 
periodic cylinders.

Let $(X,\omega)$ be a translation surface. A {\em cylinder} is a
topological annulus embedded in $X$, isometric to a flat cylinder
$\R / w \Z \times (0,h)$. In what follows all cylinders are supposed to be {\em maximal}, that is, they are not properly contained in a larger one. If $g\geq 2$, the boundary of a maximal cylinder is a finite union of saddle connections. If $\CCC$ is a cylinder, we will denote by $w(\CCC), h(\CCC),t(\CCC),\mu(\CCC)$ the width, height, twist, and modulus of $\CCC$ respectively.
 
A direction $\theta$ is {\em completely periodic} or simply {\em periodic} on $X$ if all regular
geodesics in this direction  are closed. This means that $X$
is the closure of a finite number of cylinders in
direction $\theta$, we will say that $X$ admits a {\em cylinder decomposition} in this direction. 

We can associate to any cylinder decomposition a separatrix diagram which encodes the way the cylinders are glued together, see~\cite{Kontsevich2003}). Given such a diagram, one can reconstruct the surface $(X,\omega)$ (up to a rotation) from the widths, heights, and twists of the cylinders (see Section~\ref{sec:stable}). 

\subsection{Complete periodicity}
\label{sec:complete:periodic}
A translation surface $(X,\omega)$ is said to be {\em completely periodic} if  it satisfies the following property: let $\theta\in \R\P^1$ be a direction, if 
the linear flow $\mathcal F_\theta$ in direction $\theta$ has a regular closed orbit on $X$, then $\theta$  is a periodic direction.  Flat tori and their ramified coverings are completely periodic, as well as  Veech surfaces. 
\smallskip

Completely periodicity is a very particular property. Indeed, when the genus is at least two, the Lebesgue measure of the set of surfaces having this property is zero, this is because complete periodicity is locally expressed via proportionality of a non-empty set of relative periods, and thus is defined by some quadratic equations in the period coordinates. This property has been initiated by Calta~\cite{Calta04} (see also~\cite{CaSm}) where she proved that any surface in $\PrymD(2)$ and $\PrymD(1,1)$ is completely periodic. Latter the authors extended this property to any Prym eigenform given by Table~\ref{tab:strata:list}. This property is also proved by  A.~Wright~\cite{Wright2013} by a different argument.

\begin{Theorem}[\cite{Lanneau:Manh:cp},\cite{Wright2013}]
\label{thm:cp}
Any Prym eigenform in the loci $\PrymD(\kappa) \subset \Prym(\kappa)$ given by the cases $(4)-(5)-(6)-(7)-(8)$ of 
Table~\ref{tab:strata:list} is completely  periodic.
\end{Theorem}

\section{Kernel foliation on Prym loci}
\label{sec:kernel}

The notion of kernel foliation already appeared in several papers (see~\cite{EMZ,Masur:Zorich,Calta04,Lanneau:Manh:cp}). For a proper overview on
the properties of the kernel foliation, we refer to \cite{Zorich:survey}, Section 9.6. Here below, we recall the (local) construction of this foliation which will be used throughout the paper.  In all of this section, we fix a translation surface $(X,\omega)$ with {\em several distinct} zeros. \medskip

We take some $\eps>0$ small enough so that, for every zero $P$ of $\omega$, the set
$D(P,\eps)=\{x\in X, \dist(P,x) < \eps\}$ is an embedded disc in $X$. For any direction $\theta$, 
it is a classical result that $D(P,\eps)$ can be constructed from $2(k+1)$ half-discs 
(where $k$ is the multiplicity of the zero $P$) all glued together in such a way that their centers are 
identified with $P$~\cite[Figure~3]{EMZ}. \medskip

The kernel foliation is a local action of $\C$ defined as follows: pick a complex number $w \in \C$ with $0< |w| < \eps$.
We then cut $D(P,\eps)$ into several half-discs in the direction of $w$. We will modify the flat metric of the 
polydisc $D(P,\eps)$ without changing the metric outside: on the diameter of each half-disc, there is a unique point $P'$ such that $\overrightarrow{PP'}=w$, we can glue the half-discs in such a way that all the points $P'$ are identified. Let us denote by $D'$ the  domain obtained from this gluing. We can glue $D'$ to $X\setminus D(P,\eps)$ along $\partial D'=\partial D(P,\eps)$, what we  get is a new translation surface $(X',\omega')$ which has the same absolute periods as $(X,\omega)$, and given any path $c$ in $X$  joining $P$ to another zero of $\omega$, and $c'$ the corresponding  path in $X'$, we have  $\omega(c)=\omega(c')+w$. We will say that  $(X',\omega')$ lies in the {\em kernel foliation leaf} through $(X,\omega)$. \smallskip

Remark that the Prym forms in the  Prym loci in Table~\ref{tab:strata:list} have two or three zeros. If such a Prym form has two zeros, then   the zeros are permuted by the Prym involution, if it has three zeros, then two of them are permuted, and the third one is fixed.  We also have a kernel foliation in Prym loci in Table~\ref{tab:strata:list} as follows: let $P_1,P_2$ be the pair of zeros of $\omega$ which are permuted by the Prym involution $\inv$,  given $\eps$ and $w$ as above, to get a surface $(X',\omega')$ in the same Prym locus, it suffices to move $P_1$ by $w/2$ and move $P_2$ by $-w/2$. Indeed, by assumptions, the Prym involution exchanges $D(P_1,\eps)$ and $D(P_2,\eps)$. Let $D'_1$ and $D'_2$ denote the new domains we obtain from $D(P_1,\eps)$ and $D(P_2,\eps)$ after modifying the metric. One can check that $D'_1$ and $D'_2$ are symmetric, thus the involution in $X\setminus(D(P_1,\eps)\sqcup D(P_2,\eps))$ can be extended to $D'_1 \sqcup D'_2$. Therefore we have an involution $\inv'$  on $X'$ 
such that ${\inv'}^*\omega'=-\omega'$, which implies that $(X',\omega')$ also belongs to the same Prym locus as $(X,\omega)$. We will write $(X',\omega') = (X,\omega) + w$, or simply by $X'=X+w$. \smallskip

It is worth noticing that we do not have a global action of $\C$ on each leaf of the kernel foliation, \ie even $(X,\omega)+w_1$ and $(X,\omega)+w_2$ exist, $(X,\omega)+w_1+w_2$ may not be  well defined. Nevertheless,  there still exists a local action of $\C$, namely, in a neighborhood of $(X,\omega)$ on which a local chart (by period mappings) can be defined. This is because in such a neighborhood there exists a {\em unique} surface that has the same absolute periods as $(X,\omega)$, and the relative periods different from the ones of $(X,\omega)$ by a small complex number. Therefore, if $|w_1|$ and $|w_2|$ are small enough then $(X,\omega)+(w_1+w_2)=((X,\omega)+w_1)+w_2=((X,\omega)+w_2)+w_1$. \smallskip

\noindent {\bf Convention :} Throughout this paper, we only consider the intersection of kernel foliation leaves with a neighborhood of $(X,\omega)$ on which this local action of $\C$ is well-defined, and by $(X,\omega)+w$ we will mean the surface obtained from $(X,\omega)$ by the construction described above. \smallskip

The next lemma follows from the above construction (see 
Figure~\ref{fig:decomposition:5:cylinders} for an example in $\Prym(1,1,2)$).

\begin{Lemma}
Let $c$ be a path on $X$ joining two zeros of $\omega$, and $c'$ be the corresponding path on 
$X'$. Then
\begin{enumerate}
\item If the two endpoints of $c$ are exchanged by $\inv$ then $\omega'(c')-\omega(c)=\pm w$.
\item If one endpoint of $c$ is fixed by $\inv$, but the other is not, then $\omega'(c')-\omega(c)=\pm w/2$.
\end{enumerate}
The sign of the difference is determined by the orientation of $c$.
\end{Lemma}

\begin{figure}[htbp]

\begin{minipage}[t]{0.4\linewidth}
\centering
\begin{tikzpicture}[scale=1]
\def\s{0.50}
\def\t{0.50}
\fill[fill=blue!20] (0,0) -- (1,0) -- (1,1) -- (0,1);
\fill[fill=gray!20] (0,1) -- (3,1) -- (3.25,2) -- (0.25,2);
\fill[fill=gray!20] (1,0) -- (-2,0) -- (-2.25,-1) -- (0.75,-1) ;
\fill[fill=red!10] (1.75,2) -- (3.25,2) -- (3.75,2.5) -- (2.25,2.5);
\fill[fill=red!10] (-2.25,-1) -- (-0.75,-1) -- (-1.25,-1.5) -- (-2.75,-1.5);

\draw (0,0) -- (1,0) -- (1,1) -- (3,1) -- (3.25,2) -- (3.75,2.5) -- (2.25,2.5) -- (1.75,2) -- (0.25,2) --
(0,1) -- (0,0);

\draw (0,0) -- (-2,0) -- (-2.25,-1) -- (-2.75,-1.5) -- (-1.25,-1.5) -- (-0.75,-1) -- (0.75,-1) --
(1,0);

\draw (0,1) -- (1,1);
\draw (1.75,2) -- (3.25,2);
\draw (-2.25,-1) -- (-0.75,-1);

\foreach \x in {(0,0),(0,1),(1,0),(1,1),(3,1),(-2,0)} \filldraw[fill=white] \x circle(1.5pt);
\foreach \x in {(0.25,2),(1.75,2),(3.25,2),(-2.75,-1.5),(-1.25,-1.5)} \filldraw[fill=black] \x circle(1.5pt);
\foreach \x in {(-2.25,-1),(-0.75,-1),(0.75,-1),(2.25,2.5),(3.75,2.5)} \filldraw[fill=red] \x circle(1.5pt);
\draw (0.5,0.5) node {$\scriptstyle \mathcal C_2$};
\draw (1.75,1.5) node {$\scriptstyle\mathcal C_3$};
\draw (-0.75,-0.5) node {$\scriptstyle\inv(\mathcal C_3)$};
\draw (2.75,2.25) node {$\scriptstyle\mathcal C_1$};
\draw (-1.75,-1.25) node {$\scriptstyle\inv(\mathcal C_1)$};
\draw (2,1) node[below] {$\scriptstyle C$};
\draw (-1,0) node[above] {$\scriptstyle C$};
\draw (3,2.5) node[above] {$\scriptstyle B$};
\draw (0,-1) node[below] {$\scriptstyle B$};
\draw (1,2) node[above] {$\scriptstyle A$};
\draw (-2,-1.5) node[below] {$\scriptstyle A$};
\draw (0,-3) node[below] {$(X,\omega)$};
\end{tikzpicture}
\end{minipage}
\begin{minipage}[t]{0.4\linewidth}
\centering
\begin{tikzpicture}[scale=1]
\def\s{0.50}
\def\t{0.50}

\fill[fill=blue!20] (0,0) -- (1,0) -- (1,1) -- (0,1);
\fill[fill=gray!20] (0,1) -- (3,1) -- (3.25-\s,2-\t) -- (0.25-\s,2-\t);
\fill[fill=gray!20] (1,0) -- (-2,0) -- (-2.25+\s,-1+\t) -- (0.75+\s,-1+\t) ;
\fill[fill=red!10] (1.75-\s,2-\t) -- (3.25-\s,2-\t) -- (3.75+\s,2.5+\t) -- (2.25+\s,2.5+\t);
\fill[fill=red!10] (-2.25+\s,-1+\t) -- (-0.75+\s,-1+\t) -- (-1.25-\s,-1.5-\t) -- (-2.75-\s,-1.5-\t);

\draw (0,0) -- (1,0) -- (1,1) -- (3,1) -- (3.25-\s,2-\t) -- (3.75+\s,2.5+\t) -- (2.25+\s,2.5+\t) -- (1.75-\s,2-\t) -- (0.25-\s,2-\t) --
(0,1) -- (0,0);

\draw (0,0) -- (-2,0) -- (-2.25+\s,-1+\t) -- (-2.75-\s,-1.5-\t) -- (-1.25-\s,-1.5-\t) -- (-0.75+\s,-1+\t) -- (0.75+\s,-1+\t) --
(1,0);

\draw (0,1) -- (1,1);
\draw (1.75-\s,2-\t) -- (3.25-\s,2-\t);
\draw (-2.25+\s,-1+\t) -- (-0.75+\s,-1+\t);

\foreach \x in {(0,0),(0,1),(1,0),(1,1),(3,1),(-2,0)} \filldraw[fill=white] \x circle(1.5pt);
\foreach \x in {(0.25-\s,2-\t),(1.75-\s,2-\t),(3.25-\s,2-\t),(-2.75-\s,-1.5-\t),(-1.25-\s,-1.5-\t)} \filldraw[fill=black] \x circle(1.5pt);
\foreach \x in {(-2.25+\s,-1+\t),(-0.75+\s,-1+\t),(0.75+\s,-1+\t),(2.25+\s,2.5+\t),(3.75+\s,2.5+\t)} \filldraw[fill=red] \x circle(1.5pt);
\draw (0.5,0.5) node {$\scriptstyle\mathcal C_2$};
\draw (1.75-\s/2,1.5-\t/2) node {$\scriptstyle\mathcal C_3$};
\draw (-0.75+\s/2,-0.5+\t/2) node {$\scriptstyle\inv(\mathcal C_3)$};
\draw (2.75,2.25) node {$\scriptstyle\mathcal C_1$};
\draw (-1.75,-1.25) node {$\scriptstyle\inv(\mathcal C_1)$};
\draw (2,1) node[below] {$\scriptstyle C$};
\draw (-1,0) node[above] {$\scriptstyle C$};
\draw (3+\s,2.5+\t) node[above] {$\scriptstyle B$};
\draw (0+\s,-1+\t) node[below] {$\scriptstyle B$};
\draw (1-\s,2-\t) node[above] {$\scriptstyle A$};
\draw (-2-\s,-1.5-\t) node[below] {$\scriptstyle A$};

\draw (0,-3) node[below] {$(X,\omega) + (s,t)$};
\end{tikzpicture}
\end{minipage}

\caption{Decomposition of a surface $(X,\omega)\in \Prym(1,1,2)$. 
The cylinders $\mathcal C_2$ is fixed by the Prym involution $\inv$, while 
the cylinders $\mathcal C_i$ and $\inv(\mathcal C_i)$ are exchanged for $i=1,3$. 
Along a kernel foliation leaf $(X,\omega) + (s,t)$ the twists and heights change as
follows: $t_1(s) = t_1-s$, $t_2(s) = t_2 $, $t_3(s) = t_3+s/2$ and 
$h_1(t) = h_1-t$, $h_2(t) = h_2$, $h_3(t) = h_3+t/2$.
\label{fig:decomposition:5:cylinders}
}
\end{figure}
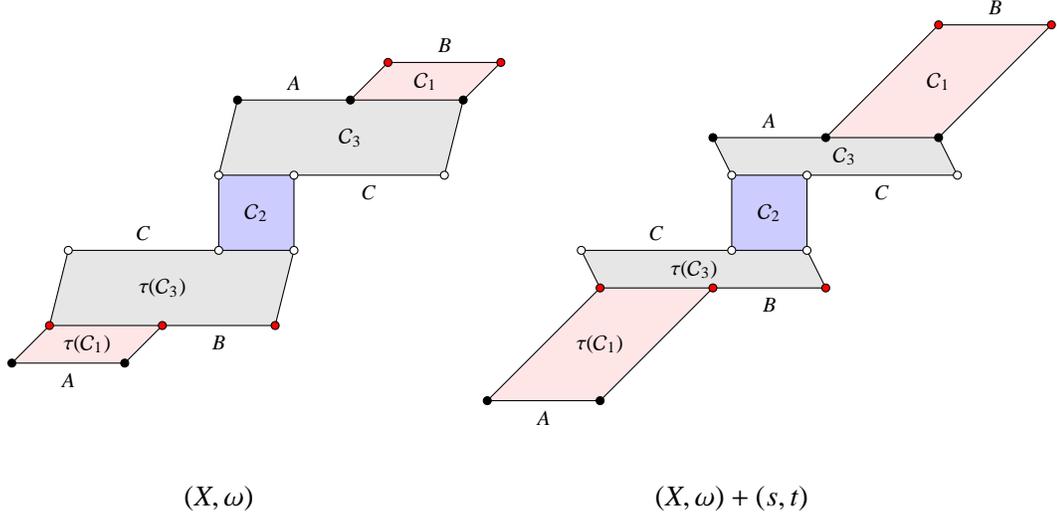


We have seen that the kernel foliation preserves the Prym locus; moreover it also preserves 
the real multiplication locus as it is shown in the next proposition.

\begin{Proposition}
\label{prop:preserves}
For any $(X,\omega)\in \Omega E_D(\kappa)$,  if $(X',\omega')=(X,\omega)+w$ is a Prym form in the same 
Prym locus as $(X,\omega)$ then $(X',\omega')\in \Omega E_D(\kappa)$.
\end{Proposition}

\begin{proof}[Sketch of the proof]
The proof is classical and details are left to the reader (see~\cite{Lanneau:Manh:cp}).
By construction, $(X',\omega')$ and $(X,\omega)$ share the same absolute periods. Let $T$ be a generator of the quadratic order $\mathcal{O}_D$ in $\mathrm{End}(\Prym(X,\inv))$. Let $T'$ be the $\R-$linear endomorphism of $H_1(X',\Z)^-$
corresponding to $T$. Since $\Prym(X',\inv')$ has complex dimension $2$, $T'$ is $\C-$linear~\cite{Mc6}. Hence $T' \in \mathrm{End}(\Prym (X',\inv'))$,
and since $\omega'$ is an eigenform of $T'$, one has $(X',\omega')\in \Omega E_D(\kappa)$.
\end{proof}
We end this section by giving a description of a neighborhood of a Prym eigenform: 
up to the action of $\GL^+(2,\R)$ a neighborhood of a point $X$ in $\PrymD(\kappa)$ can be identified with 
the ball $\left\{X+w, |w| < \varepsilon \right\}$.

\begin{Proposition}[\cite{Lanneau:Manh:cp}]
\label{prop:local}
For any $(X,\omega)\in \Omega E_D(\kappa)$,  if $(X',\omega')$ is a Prym eigenform in $\Omega E_D(\kappa)$ close enough to $(X,\omega)$, then there exists a unique pair $(g,w)$, where $g\in \GL^+(2,\R)$ close to $\Id$, and $w\in \R^2$ with $|w|$ small, such that $(X',\omega')=g\cdot(X,\omega)+w$.
\end{Proposition}

\begin{proof}
For completeness we include the proof here (see~\cite[Section~3.2]{Lanneau:Manh:cp}).\\
Let $(Y,\eta)=(X,\omega)+w$, with $|w|$ small, be a surface in the leaf of the kernel foliation through $(X,\omega)$. 
We denote by $[\omega]$ and $[\eta]$ the classes of $\omega$ and $\eta$ in $H^1(X,\Sigma;\C)^-$. Then we have 
$$
[\eta]-[\omega] \in \ker\rho,
$$
where $\rho: H^1(X,\Sigma;\C)^- \rightarrow H^1(X,\C)^-$ is the natural surjective linear map.
On the other hand, the action of $g\in \GL^+(2,\R)$ on $H^1(X,\Sigma;\C)^-$ satisfies 
$$
\rho(g\cdot[\omega])=g\cdot\rho([\omega]).
$$
\noindent Therefore the leaves of the kernel foliation and the orbits of $\GL^+(2,\R)$ are transversal. Since their dimensions are complementary, the proposition follows.
\end{proof}

\section{Stable an Unstable cylinder decompositions}
\label{sec:stable}

\subsection{Definitions}
We call a geodesic ray emanating from a zero of $\omega$ a {\em separatrix}. It is a well-known fact that a direction is periodic if and only if all the separatrices in this direction are saddle connections. The following definition will be useful for us.

\begin{Definition}\label{def:stable:cyl:dec}
 A cylinder decomposition of $(X,\omega)$ is said to be {\em stable} if every separatrix joins a zero of $\omega$ to itself. The decomposition is said to be {\em unstable} otherwise.
\end{Definition}

Obviously,  a stable cylinder decomposition only makes sense when $\omega$ has more than one zero. In $\H(1,1)$, a cylinder decompositions may have one, two, or three cylinders, and stable decompositions are the ones with three cylinders. 

\begin{Lemma}\label{lm:max:nb:cyl}
If the genus of $X$ is $g$ then any direction $\theta$ that decomposes $(X,\omega)\in \H(\kappa)$ into 
$g+|\kappa|-1$ cylinders is stable ($|\kappa|$ is the number of zeros of $\omega$).
\end{Lemma}

\begin{proof}
We begin by observing that any periodic direction decomposes the surface $X$ into {\it at most}
$g+|\kappa|-1$ cylinders. Now if the direction $\theta$ is not stable then there
exists necessarily a saddle connection between two different zeros that we can
collapse to a point without destroying any cylinder. But in this way we get a surface $(X',\omega')\in \H(\kappa')$
of genus $g$ where $|\kappa'| < |\kappa|$, and having $g+|\kappa|-1$ cylinders. This is a contradiction.
\end{proof}

The proof of the following lemma is elementary and left to the reader.

\begin{Lemma}\label{lm:nbr:cyl:max3}
Let $(X,\omega)\in\Prym(\kappa)$ be a surface in the strata given by Table~\ref{tab:strata:list}. Let us assume that 
$\kappa \not =(1,1,2,2)$. If the horizontal direction is periodic for $(X,\omega)$, with $n$ horizontal cylinders 
counted up to the Prym involution, then $n\leq 3$. Moreover, if $n=3$ then the cylinder decomposition is stable. 
\end{Lemma}
\begin{proof}
One can easily check that, in all cases, if $n=3$ then the number  $k$ of horizontal cylinders satisfies $k=|\kappa|+g-1$.
\end{proof}

\medskip

\begin{Remark}
Let $\H(0,0,2)$ be the space of quadruplets $(X,\omega,P_1,P_2)$ where $(X,\omega) \in \H(2)$ and 
$P_1,P_2$ are two regular points of $X$ that are exchanged by the hyperelliptic involution. The above lemma is false for the stratum $\Prym(1^2,2^2)$. However, using the identification $\Prym(1^2,2^2) \simeq \H(0,0,2)$, Lemma~\ref{lm:nbr:cyl:max3} becomes true with the convention that a cylinder decomposition of $(X,\omega)\in\Prym(1^2,2^2)$ is stable/unstable, if the decomposition of the corresponding surface in $\H(0,0,2)$ is.
\end{Remark}

\begin{Remark}
For $\Prym(1,1) \simeq \H(1,1)$, $|\kappa|+g-1=3$, and all stable cylinder decompositions have $3$  cylinders.  However, in the other Prym loci, there exist stable decompositions with less than $|\kappa|+g-1$ cylinders. 
\end{Remark}

\subsection{Combinatorial data}
Given a surface $(X,\omega)$ for which the horizontal direction is periodic, since each saddle connection is contained in the upper (resp. lower) boundary of a unique cylinder, we can associate to the cylinder decomposition the following data

\begin{itemize}
 \item[$\bullet$] two partitions of the set of saddle connections into $k$ subsets, where $k$ is the number of cylinders, each subset in these partitions is equipped with a cyclic ordering, and
 \item[$\bullet$] a pairing of subsets in these two partitions.
\end{itemize}

We will call these data the {\em combinatorial data} or {\em topological model} of the cylinder decomposition. Note that while there exists only one topological model for cylinder decompositions with maximal number of cylinders in $\Prym(1,1)$, in general, there are several topological models for such decompositions in other Prym loci in Table~\ref{tab:strata:list}.

\subsection{Kernel foliation and stable decomposition}
The next two propositions will play an important role in the sequel.
\begin{Proposition}
\label{prop:stable:dec}
Let $(X,\omega)\in \PrymD(\kappa)$, where $\kappa$ is one of the strata in Table~\ref{tab:strata:list}. If $(X,\omega)$ 
admits a stable cylinder decomposition then there exists $\eps>0$ such that for every $v\in \R^2$, with $|v| < \eps$, 
$(X,\omega)+v$ admits a stable cylinder decomposition (in the same direction) with the same combinatorial 
data and the same widths of cylinders.
\end{Proposition}

\begin{proof} 
We only give the proof for $\kappa=(2,2)^{\rm odd}$ since the arguments for the other cases are completely similar.
As usual $\theta$ is assumed to be the horizontal direction. We begin with the following observation:
since the horizontal direction is stable, the horizontal kernel foliation is well defined for all time. Thus the proposition 
is clear if $v$ is horizontal. Hence we need to prove the proposition for {\em vertical} vectors only.
We recall here the vertical kernel foliation (in the direction of $v$) in more details (see~\cite{EMZ}).

We consider two embedded discs $D_P$ and $D_Q$, centered at the zeroes $P$ and $Q$, of radius $\eps$
that misses all other zeros. In a more concrete way, each disc $D_P$ and $D_Q$ is constructed from the union 
of $3$ pairs of Euclidian half-discs, $(D^{-}_i,D^{+}_i)$ 
where $D^-_i=\{z\in {\bf B}(\eps),\ -\eps\leq \mathrm{Re}(z) \leq 0 \}$ and 
$D^+_i=\{z \in {\bf B}(\eps),\ 0\leq \mathrm{Re}(z) \leq \eps\}$ whose the boundaries are isometrically glued together in a ``circular fashion''. More specifically, to get $D_P$, we glue $D^\pm_i, \, i=1,2,3$, as follows
\begin{itemize}
 \item $D^+_i$ is glued to $D^-_{i}$ along the segment $\{\mathrm{Re}(z)=0, \, 0 \leq \mathrm{Im}(z) < \eps\}$, and
 \item $D^-_i$ is glued to $D^+_{i+1}$ along the segment $\{\mathrm{Re}(z)=0, \, -\eps < \mathrm{Im}(z) \leq 0\}$.
\end{itemize}
To get $D_Q$, we glue $D^\pm_i, \, i=4,5,6$, as follows
\begin{itemize}
\item $D^-_i$ is glued to $D^+_i$ along the segment $\{{\rm Re}(z)=0, -\eps < {\rm Im}(z) \leq 0\}$, and
\item $D^+_i$ is glued to $D^-_{i+1}$ along the segment $\{{\rm Re}(z)=0, \, 0 \leq {\rm Im}(z) <  \eps\}$.
\end{itemize}
(with the dummy conventions $D^{\pm}_4=D^{\pm}_1$ and $D^{\pm}_7=D^{\pm}_4$). Note that the gluings for $D_P$ and $D_Q$ are different, and one can assume that $\inv(D^+_i)=D^-_{i+3}, \ \inv(D^-_i)=D^+_{i+3}, \, i=1,2,3$.  \medskip

We now make a local surgery of the flat structure of $(X,\omega)$, {\em i.e.} we do not change the flat structure 
outside the union of the discs $D_P$ and $D_Q$. This is carried out as follows (see Figure~\ref{fig:unstable:ker:fol}): we fix some $0\leq h\leq \eps$, we then replace $D_P$ and $D_Q$ by discs $D'_P$ and $D'_Q$ that are constructed from the same pairs of half discs 
$(D^-_i, D^+_i)_{i=1,\dots,6}$ but with different gluings. Namely, for $D'_P$ ($i=1,2,3$):
\begin{itemize}
\item $D^+_i$ is glued to $D^-_{i}$ along the segment $\{\mathrm{Re}(z)=0, \, -h \leq \mathrm{Im}(z) < \eps\}$, and
\item $D^-_i$ is glued to $D^+_{i+1}$ along the segment $\{\mathrm{Re}(z)=0, \, -\eps < \mathrm{Im}(z) \leq -h\}$,
\end{itemize}
and similarly for $D'_Q$ ($i=4,5,6$):
\begin{itemize}
\item $D^-_i$  is glued to $D^+_i$ along the segment $\{{\rm Re}(z)=0, -\eps < {\rm Im}(z) \leq h \}$,
\item $D^+_i$ is glued to $D^-_{i+1}$ along the segment $\{{\rm Re}(z)=0, \, h \leq {\rm Im}(z) < \eps\}$,
\end{itemize}
By construction the new surface we get is $(X',\omega')=(X,\omega)+(0,2h)$. \medskip

Now if $\{\gamma_j\}_{j=1,\dots,k}$ denote the core curves of the horizontal cylinders in $X$
(whose distances to the two boundary components are equal), we can always choose 
$\eps>0$  small enough so that the embedded discs $D_P$ and $D_Q$ are also disjoint from 
$\DS{ \cup_{j=1}^k\gamma_j}$. 

For each $i=1,2,3$, let $a^+_i$ (respectively, $a^-_i$) be the intersection of $D^+_i$ (respectively, of $D^-_i$) 
with the horizontal saddle connections emanating from $P$. Since any saddle connections emanating from $P$ terminates at $P$, there is a permutation $\piup$ of the set $\{1,2,3,4,5,6\}$, which preserves the subsets $\{1,2,3\}$ and $\{4,5,6\}$, such that  $a^+_i$ and $a^-_{\piup(i)}$ belong to the same saddle connection. 

We perform the same construction for the surface $(X',\omega')$: the corresponding segments are 
$b^+_i = \{z \in D^+_i, \, {\rm Im}(z)=-h\}$ and $b^-_i = \{z \in D^-_i, \, {\rm Im}(z)=-h\}$. 
By construction, $b^+_i$ and $b^-_{\piup(i)}$ also belong to the same 
horizontal saddle connection in $X'$, which therefore joins the zero $P'$ of $\omega'$, corresponding to $P$, to itself.
Since the same surgery applies for the disc $D_Q$, one concludes that the horizontal direction on 
$(X',\omega')$ is completely periodic and stable.

It remains to show that the combinatorial data are the same. First notice that the curves $\gamma_j$ are 
core curves of the cylinders in $X'$ (since they are preserved along the surgery). Thus 
the number of cylinders  and the widths of the cylinders are the same. Since the gluings are the 
same along the surgery, the combinatorics of the gluings are also preserved as well.
The  proposition is then proved.
\end{proof}
 
\begin{Proposition}
\label{prop:cyl:dec:unstable}
Let $(X,\omega)\in \PrymD(\kappa)$, where $\kappa$ is one of the strata in Table~\ref{tab:strata:list}. If $(X,\omega)$ 
admits an unstable cylinder decomposition in the horizontal direction then there exists $\eps>0$ such that for every 
$v=(x,y) \in \R^2$, with $|v| < \eps$ and $y\neq 0$, $(X,\omega)+v$ admits a stable cylinder decomposition in the 
horizontal direction. Moreover, the combinatorial data of the decomposition and the widths of the cylinders  depend only on the sign of $y$.
\end{Proposition}

\begin{proof}
Again, we only  give the proof for the case $\kappa=(2,2)^{\rm odd}$. We keep the same conventions 
as in the proof of the preceding proposition. Clearly, we only need to consider the case $v=(0,2h), h \neq 0$. Let us assume that $h>0$. For each of the half-discs $D^\pm_i, \, i=1,\dots,6$, we define
$$
\begin{array}{l}
a^\pm_i=\{z\in D^\pm_i, \, {\rm Im}(z)=0\}, \\
b^\pm_i=\{z\in D^\pm_i, \, {\rm Im}(z)=-h\},\ \mathrm{and}\\
c^\pm_i= \{z\in D^\pm_i, \, {\rm Im}(z)=h\}.
\end{array}
$$


Since all  the separatrices in the horizontal direction are saddle connections (the horizontal direction is periodic) 
there is a permutation $\piup$ of the set $\{1,\dots,6\}$ such that $a^+_i$ and $a^-_{\piup(i)}$  belong to the same 
saddle connection. Hence for each $i$,  $b^+_i$ and $b^-_{\piup(i)}$ (respectively,  $c^+_i$ and $c^-_{\piup(i)}$) belong to 
the same horizontal leaf. Moreover, from the kernel foliation construction, and since $h>0$, one has (see Figure~\ref{fig:unstable:ker:fol})
\begin{itemize}
\item[$\bullet$] $a^-_i$ and $a^+_i$ belong to the same horizontal leaf  for $i=1,\dots,6$.
\item[$\bullet$] $b^-_i$ and $b^+_i$ belong to the same horizontal leaf  for $i=4,5,6$,  
\item[$\bullet$] $c^-_i$ and $c^+_i$ belong to the same horizontal leaf  for $i=1,2,3$. 
\end{itemize}
The assumption that the decomposition of $(X,\omega)$ is not stable means that $\piup(\{1,2,3\}) \neq \{1,2,3\}$. For every $i\in \{1,2,3\}$, there exists a unique sequence $(i=i_0,i_1,\dots,i_k)$, where $i_{j+1}=\piup(i_j)$, $i_j\in \{4,5,6\}$, for $j=1,\dots,k-1$, and $i_k\in\{1,2,3\}$. Remark that such a sequence corresponds to a saddle connection joining $P'$ to itself, $P'$ is the zero of $\omega'$ corresponding to $P$, this saddle connection contains the segments $b^+_{i_0},b^-_{i_1},b^+_{i_1},\dots,b^-_{i_{k-1}}, b^+_{i_{k-1}}, b^-_{i_k}$. Similarly, for every $i\in \{4,5,6\}$, there exists a unique sequence $(i=i_0,i_1,\dots,i_k)$, where $i_{j+1}=\piup(i_j)$, $i_j\in \{1,2,3\}$, for $j=1,\dots,k-1$, and $i_k\in \{4,5,6\}$. Such a sequence corresponds a saddle connection joining $Q'$, the zero of $\omega'$ corresponding to $Q$, to itself, this saddle connection contains the segments $c^+_{i_0},c^-_{i_1}, c^+_{i_1}, \dots, c^-_{i_{k-1}}, c^+_{i_{k-1}}, c^-_{i_k}$. It follows that $(X',\omega')$ also 
admits a cylinder decomposition in the horizontal direction, and this decomposition is stable. 

By construction, $a^+_i$ and $a^-_{\piup(i)}$ are contained in the same horizontal leaf of $(X',\omega')$, it follows that each cycle of $\piup$ corresponds to a simple closed geodesic in $X'$.  Let  $\hat{\gamma}_j, \, j=1,\dots,m$, denote the simple closed geodesics corresponding to the cycles of $\piup$, and  $\hat{C}_j$ denote the cylinder associated to $\hat{\gamma}_j$.

Since the curves $\gamma_j, \, j=1,\dots,k,$ are disjoint from $D_P$ and $D_Q$, they are closed geodesics in $(X',\omega')$. Let $C'_j$ denote the cylinder associated to $\gamma_j$.
It is clear that the combinatorial data of the cylinder decomposition of $(X',\omega')$, which consists of $C'_j,\; j=1,\dots,k$, and $\hat{C}_j, \; j=1,\dots,m$,  are determined by $\piup$ and stay unchanged as long as $h>0$. 



\begin{figure}[htb]
\centering
\subfloat{
\begin{tikzpicture}[scale=0.4]
 \draw (-8,5) -- (-8,2) -- (-6,2) -- (-6,0) -- (-2,0) -- (-2,2) -- (-4,2) -- (-4,5) -- (-2,5) -- (-2,7) -- (-6,7) -- (-6,5) -- cycle; \draw (-6,5) -- (-4,5) (-6,2) -- (-4,2);
 \foreach \x in {(-8,5), (-8,2), (-4,7), (-4,5), (-4,2), (-4,0)} \filldraw[fill=black] \x circle (3pt);
 \foreach \x in {(-6,7), (-6,5), (-6,2), (-6,0),  (-2,7), (-2,5), (-2,2), (-2,0)} \filldraw[fill=white] \x circle (3pt);
 
 \draw (-4,6) node {$\sst C_1$} (-6,3.5) node {$\sst C_2$}  (-4,1) node {$\sst C_3$};

 \fill[yellow!50!black!30!blue!20] (4,7) -- (4,6.5) -- (6,6.5) -- cycle;
 \fill[yellow!50!black!30!blue!20] (6,6.5) -- (8,6.5) -- (8,7) -- cycle;
 \fill[yellow!50!black!30!blue!20] (2,4.5) -- (6,4.5) -- (8,5) -- (4,5) -- cycle;
 \fill[yellow!50!black!30!blue!20] (2,2) -- (2,1.5) -- (4,2) -- cycle;
 \fill[yellow!50!black!30!blue!20] (4,2) -- (4,1.5) -- (6,1.5) -- (6,2) -- cycle;
 \fill[yellow!50!black!30!blue!20] (6,1.5) -- (8,1.5) -- (8,2) -- cycle;
 \fill[yellow!50!black!30!blue!20] (4,0) -- (6,-0.5) -- (8,0);

 \draw (2,4.5) -- (2,1.5) -- (4,2) -- (4,0) -- (6,-0.5) -- (8,0) -- (8,2) -- (6,1.5) -- (6,4.5) -- (8,5) -- (8,7) -- (6,6.5) -- (4,7) -- (4,5) -- cycle; \draw (4,6.5) -- (8,6.5) (4,5) -- (8,5) (2,4.5) -- (6,4.5) (2,2) -- (6,2) (4,1.5) -- (8,1.5) (4,0) --(8,0);
 
 \foreach \x in {(2,4.5), (2,1.5), (6,6.5), (6,4.5), (6,1.5), (6,-0.5)} \filldraw[fill=black] \x circle (3pt);
 \foreach \x in {(4,7), (4,5),(4,2), (4,0), (8,7),(8,5),(8,2),(8,0)} \filldraw[fill=white] \x circle (3pt);
 
 \draw (6,5.5) node {$\sst C'_1$} (4,3) node[above] {$\sst C'_2$} (6,0.5) node {$\sst C'_3$}; 
\end{tikzpicture}

}

\smallskip

\subfloat{
\begin{tikzpicture}[scale=0.5]
\foreach \x in {(-12.5,0),(-7.5,0),(-2.5,0), (2.5,0), (7.5,0), (12.5,0)} {\draw[red] \x +(-0.1,0) -- +(-2.1,0) \x +(0.1,0) -- +(2.1,0); \draw[red] \x ++(-0.1,0) +(0,1) -- +(150:2) +(0,-1) -- +(210:2); \draw[red] \x ++(0.1,0) +(0,1) -- +(30:2) +(0,-1) -- +(-30:2); }

\foreach \x in {(-12.5,0), (-7.5,0), (-2.5,0)} {\draw[red] \x +(-0.1,1) -- +(0.1,1) +(-0.1,0) -- +(0.1,0); \draw \x +(-0.1,2) -- +(0.1,2) \x +(0,-1) -- +(0,2) \x +(-0.1,-1) -- +(-0.1,-2) \x +(0.1,-1) -- +(0.1,-2); \filldraw[fill=black] \x +(0,-1) circle (3pt);  }

\foreach \x in {(2.5,0),(7.5,0), (12.5,0)} {\draw[red] \x +(-0.1,-1) -- +(0.1,-1) +(-0.1,0) -- +(0.1,0); \draw \x +(-0.1,1) -- +(-0.1,2) \x +(0.1,1) -- +(0.1,2); \draw \x +(0,1) -- +(0,-2); \filldraw[fill=white] \x +(0,1) circle (3pt); }

\draw (-13.5,1) node[fill=white] {$\scriptstyle c^-_1$}  (-13.5,-1) node[fill=white] {$\scriptstyle b^-_1$} (-11.5,1) node[fill=white] {$\scriptstyle c^+_1$} (-11.5,-1) node[fill=white] {$\scriptstyle b^+_1$} (-13.5,0) node[fill=white] {$\scriptstyle a^-_1$} (-11.5,0) node[fill=white] {$\scriptstyle a^+_1$} ;

\draw (-8.5,1) node[fill=white] {$\scriptstyle c^-_2$} (-8.5,-1) node[fill=white] {$\scriptstyle b^-_2$} (-6.5,1) node[fill=white] {$\scriptstyle c^+_2$} (-6.5,-1) node[fill=white] {$\scriptstyle b^+_2$} (-8.5,0) node[fill=white] {$\scriptstyle a^-_2$} (-6.5,0)  node[fill=white] {$\scriptstyle a^+_2$};

\draw (-3.5,1) node[fill=white] {$\scriptstyle c^-_3$}  (-3.5,-1) node[fill=white] {$\scriptstyle b^-_3$} (-1.5,1) node[fill=white] {$\scriptstyle c^+_3$} (-1.5,-1) node[fill=white] {$\scriptstyle b^+_3$} (-3.5,0) node[fill=white] {$\scriptstyle a^-_3$} (-1.5,0) node[fill=white] {$\scriptstyle a^+_3$};

\draw (1.5,0) node[fill=white] {$\scriptstyle a^-_4$} (3.5,0) node[fill=white] {$\scriptstyle a^+_4$} (1.5,1) node[fill=white] {$\scriptstyle c^-_4$}  (1.5,-1) node[fill=white] {$\scriptstyle b^-_4$} (3.5,1) node[fill=white] {$\scriptstyle c^+_4$} (3.5,-1) node[fill=white] {$\scriptstyle b^+_4$};

\draw (6.5,1) node[fill=white] {$\scriptstyle c^-_5$}  (6.5,-1) node[fill=white] {$\scriptstyle b^-_5$} (8.5,1) node[fill=white] {$\scriptstyle c^+_5$} (8.5,-1) node[fill=white] {$\scriptstyle b^+_5$} (6.5,0) node[fill=white] {$\scriptstyle a^-_5$} (8.5,0) node[fill=white] {$\scriptstyle a^+_5$};

\draw  (11.5,1) node[fill=white] {$\scriptstyle c^-_6$}  (11.5,-1) node[fill=white] {$\scriptstyle b^-_6$} (13.5,1) node[fill=white] {$\scriptstyle c^+_6$} (13.5,-1) node[fill=white] {$\scriptstyle b^+_6$} (11.5,0) node[fill=white] {$\scriptstyle a^-_6$} (13.5,0) node[fill=white] {$\scriptstyle a^+_6$};

\foreach \x in {(-12.5,0), (-7.5,0), (-2.5,0)} {\draw \x +(0.1,-2) arc (-90:90:2) \x +(-0.1,2) arc (90:270:2); \draw \x +(-0.1,2) -- +(0.1,2);}
\foreach \x in {(2.5,0), (7.5,0), (12.5,0)} {\draw \x +(0.1,-2) arc (-90:90:2) \x +(-0.1,2) arc (90:270:2); \draw \x +(-0.1,-2) -- +(0.1,-2);}

\end{tikzpicture}
}
\caption{An example of kernel foliation near an unstable decomposition, in this case all the horizontal rays starting from $P$ terminate at $Q$, and $(X,\omega)$ has $3$ horizontal cylinders, $\piup=(1,4,3,5,2,6)$, $h>0$, $m=1$, thus $(X',\omega')$ has $4$ cylinders, the new cylinder is colored. The saddle connections emanating from $P'$ correspond to the sequences: $\{1,4,3\}, \{2,6,1\}, \{3,5,2\}$, and those starting from $Q'$ correspond to $\{4,3,5\}, \{5,2,6\}, \{6,1,4\}$.}
\label{fig:unstable:ker:fol}
\end{figure}
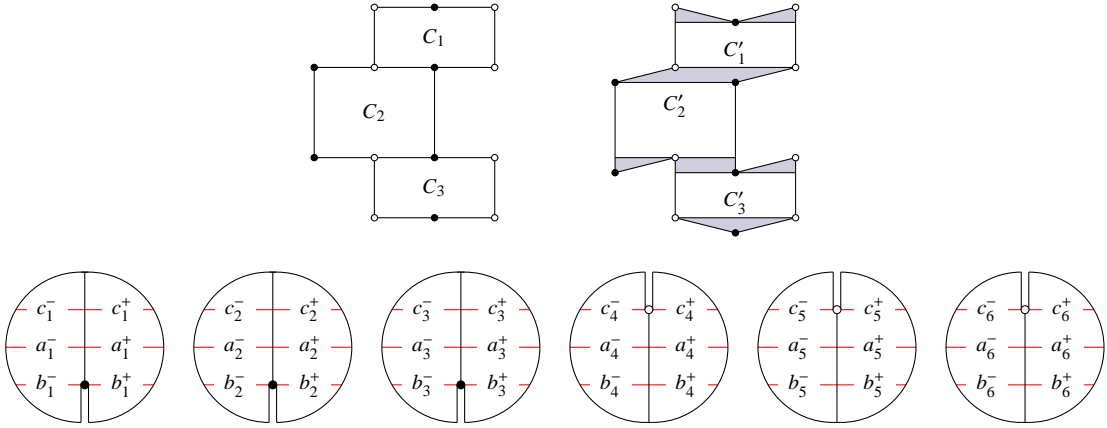

It is also clear from the construction that $C'_j$ and $C_j$ have the same width, while the width of $\hat{C}_j$ is determined by the lengths of the horizontal saddle connections of $(X,\omega)$ and the permutation $\piup$. Thus the widths  of the cylinders in $(X',\omega')$ only depends on the sign of $h$. The proof of the proposition is now complete.
\end{proof}


\subsection{Action of the kernel foliation on cylinders}
\label{sec:kernel:periodic:surface}

\subsubsection{Horizontal kernel foliation}
Let $(X,\omega)\in \PrymD(\kappa)$ be a Prym eigenform with a stable cylinder decomposition in the horizontal
direction. For any $s \in \R$, the kernel foliation $(X,\omega)+(s,0)$ is well defined, and also admits a cylinder 
decomposition in the horizontal direction with the same topological properties as the decomposition of 
$(X,\omega)$. Let $C_i(s,0)$ denote the horizontal cylinder in $(X,\omega)+(s,0)$ corresponding to 
$C_i, \, i=1,\dots,k$. Let $w(C_i(s,0)), h(C_i(s,0)), t(C_i(s,0))$ denote the width, height, and twist of $C_i(s,0)$. Since the cylinder decomposition is stable, the upper (resp. lower) boundary of $C_i$ contains only one zero of $\omega$. Thus, the twist of $C_i$ is well defined up to an absolute period of $\omega$. \smallskip

By construction, we have
$$
\begin{array}{l}
w(C_i(s,0))=w(C_i)=w_i,\\
h(C_i(s,0))=h(C_i)=h_i,
\end{array}
$$
for any $s$. However, in general $t(C_i(s,0))$ is a non-constant function of $s$. 

\begin{Lemma}\label{lm:ker:hor:twist}
We have $t(C_i(s,0))=t_i +\alpha_i s$, where
$$
\alpha_i=\left\{
\begin{array}{ll}
 0 & \hbox{ if the zeros in the upper and lower boundaries of $C_i$ are the same,}\\
 \pm 1 & \hbox{ if the zeros are exchanged by the Prym involution,}\\
 \DS{\pm 1/2} & \hbox{ if one zero is fixed, the other is mapped to the third one by the Prym involution.}\\
\end{array}
\right.
$$
\end{Lemma}

\subsubsection{Vertical kernel foliation}
If $v=(0,t)$, then by Proposition~\ref{prop:stable:dec}, $(X,\omega)+(0,t)$ is well defined whenever $|t| < \min\{h_i, \; i=1,\dots,k\}$. Let $C_i(0,t)$ denote the cylinder in $(X,\omega)+(0,t)$ that corresponds to $C_i$.  The widths (as they are absolute periods) and the twists of the cylinders $C_i(0,t)$ are unchanged, only their heights vary. Namely,
\begin{Lemma}\label{lm:ker:ver:height}
We have $h(C_i(0,t))=h_i+\alpha_it$, where
$$
\alpha_i =\left\{
\begin{array}{ll}
0 & \textrm{ if the zeros in the upper and lower boundaries are the same,}\\
\pm 1 & \textrm{ if the zeros are exchanged by the Prym involution,}  \\
\DS{\pm 1/2} & \textrm{ if one zero is fixed, the other is mapped to the third one by the Prym involution.}\\
\end{array}
\right.
$$
\end{Lemma}

The proofs of Lemma~\ref{lm:ker:hor:twist} and Lemma~\ref{lm:ker:ver:height} are elementary and left to the reader.

\subsection{Action of the horizontal horocycle flow on cylinders}
\label{sec:horo:action}

The (horizontal) horocycle flow is defined as the action of the one parameter subgroup $U=\{u_s,\quad s \in \R\}$ of 
$\GL^+(2,\R)$, where $u_s = \left( \begin{array}{ll} 1 & s \\ 0 & 1 \end{array} \right).$ If the horizontal direction on $(X,\omega)$ is completely periodic, then obviously the action of $u_s$ on $(X,\omega)$  preserves the cylinder decomposition topologically. Moreover each 
cylinder $C_i$ with parameters $(w_i,h_i,t_i~\mod w_i)$ is mapped to a cylinder 
$C_i(s):= u_s(C_i)$ of $u_s\cdot (X,\omega)$ with the same width and height, while the twist is given by
\begin{equation}\label{eq:horo:twist}
t(C_i(s))= t_i + sh_i \mod w_i.
\end{equation}

\subsection{Cylinders decomposition: relation of moduli}
\label{sec:moduli:equation}
We first recall the following result.
\begin{Theorem}[McMullen~\cite{Mc2}]
\label{theo:Cal:eq}
Let $K\subset \R$ be a real quadratic field and let $(X,\omega)\in \Omega E_D(\kappa)$ be a Prym eigenform such that
all the absolute periods of $\omega$ belong to $K(\imath)$. We assume 
that the horizontal direction is completely  periodic with $k$ cylinders. If we cannot normalize  by $\GL^+(2,K)$ so
that all the absolute periods of $\omega$ belong to $\Q(\imath)$ then the following equation holds
$$
\sum_{i=1}^{k}w'_ih_i=0,
$$
\noindent where $w_i,h_i$ are respectively the width and the height of the $i$-th cylinder, and 
$w'_i$ is the Galois conjugate of $w_i$ in $K$.
\end{Theorem}
\begin{proof}[Sketch of proof]
A  remarkable property of Prym eigenform  is that the {\em complex flux} vanishes. Namely (see~\cite[Theorem 9.7]{Mc2})
$$
\int_X \omega\wedge\omega'=\int_X\omega\wedge \ol{\omega}'=0.
$$
\noindent Here $\ol{\omega}$ and $\omega'$ are respectively the complex conjugate and the Galois conjugate of $\omega$.  The argument is as follows:
let $T$ be a generator of the order $\Ord_D$, we have a pair of $2$-dimensional 
eigenspaces $S\oplus S' = H^{1}(X,\mathbb R)^-$ on which $T$ acts by multiplication by a scalar, where 
$S$ is spanned by $\mathrm{Re}(\omega)$ and $\mathrm{Im}(\omega)$, and $S'$ is spanned by ${\rm Re}(\omega')$ and ${\rm Im}(\omega')$. Since $T$ is self-adjoint, $S$ and $S'$ are
orthogonal with respect to the cup product. This shows the equalities above. Now 
since 
$$
\int_{\CCC_i} {\rm Im}(\omega)\wedge \mathrm{Re}(\omega')= w'_ih_i,
$$
where $\CCC_1,\dots,\CCC_k$ are the horizontal cylinders in $X$, and since the surface $X$ is covered by those cylinders:
$$
\sum_{i=1}^k w_i'h_i = \sum_{i=1}^k  \int_{\CCC_i}{\rm Im}(\omega)\wedge {\rm Re}(\omega') =\int_X \mathrm{Im}(\omega) \wedge \mathrm{Re} (\omega') =
\frac{1}{4\imath}\int_X (\omega - \ol{\omega})  \wedge (\omega' + \ol{\omega}') = 0. 
$$
Theorem~\ref{theo:Cal:eq} is proved.
\end{proof}

\begin{Corollary}
\label{cor:parabolic}
Let $(X,\omega)$ be a Prym eigenform in some locus $\PrymD(\kappa)$, where $D$ is not a square, and $K=\Q(\sqrt{D})$. Assume that $(X,\omega)$ is periodic in the horizontal direction.
Let $n$ be the number of horizontal cylinders up to Prym involution, then the following equation holds:
\begin{equation}
\label{eq:CAP:mod:rel}
\sum_{i=1}^{n}\beta_i\mu_iN(w_i)=0,
\end{equation}
\noindent  where $N(w_i)=w_iw'_i\in \Q$, $\mu_i$ is the modulus of $\CCC_i$, and $\beta_i=1$ if $\mathcal C_i$ is preserved by the Prym involution, and $\beta_i=2$ otherwise.\\
In particular, in the case  $n\leq 2$, Equation~\eqref{eq:CAP:mod:rel} implies that all the cylinders are 
commensurable, {\em i.e.} the horizontal direction is parabolic.
\end{Corollary}

\medskip

Corollary~\ref{cor:parabolic} implies that when $D$ is not a square, there is always a rational relation between the moduli of the cylinders (in a cylinder decomposition). We will now prove the same statement for the case $D$ is a square, that is $\Q(\sqrt{D})=\Q$. In what follows $(X,\omega)$ will be a Prym eigenform in one of the loci in Table~\ref{tab:strata:list}, and $D$ will be the discriminant of the Prym eigenform locus that contains $(X,\omega)$. We also assume that $(X,\omega)$ decomposes into $k$ cylinders, denoted by $\CCC_1,\dots,\CCC_k$, in the horizontal direction. The width, height, and modulus of $\CCC_i$ are denoted by  $w_i,h_i$, and $\mu_i$ respectively. If the corresponding cylinder decomposition is stable, then the coefficient associated to $\CCC_i$ (see Lemma~\ref{lm:ker:hor:twist} and Lemma~\ref{lm:ker:ver:height}) will be denoted by $\alpha_i$.    Let us start by

\begin{Lemma}\label{lm:cyl:h:abs:per}
For every $i\in \{1,\dots,k\}$ either $h_i$ is an absolute period, or there exists $j\neq i$ and some integers $x_i,x_j\in \{1,2\}$ such that $x_ih_i+x_jh_j$ is an absolute period. Moreover, if the cylinder decomposition is stable, and $\alpha_i,\alpha_j$ are the coefficients associated to $\CCC_i$ and $\CCC_j$ respectively, then  $x_i\alpha_i+x_j\alpha_j=0$. 
\end{Lemma}
\begin{proof}
If there is a zero of $\omega$ that is contained in both top and bottom border of $\CCC_i$, then $h_i$ is an absolute period. Let us suppose that this does not occur. We have two cases:

\begin{itemize}
  \item[(a)]{\bf Case 1:} $\omega$ has two zeros $P_1,P_2$. Note that in this case $P_1$ and $P_2$ are exchanged by the Prym involution $\inv$. We can assume that the bottom border of $\CCC_i$ contains $P_1$, and its top border contains $P_2$. By connectedness of $X$, there must exist a cylinder $\CCC_j$ whose bottom border contains $P_2$ and top border contains $P_1$. Remark that we must have $i\neq j$ otherwise $P_1$ is contained in both top and bottom borders of $\CCC_i$. Let $\s_i$ and $\s_j$ be  respectively  some saddle connections in $\CCC_i$ and $\CCC_j$ which join $P_1$ to $P_2$. Then $c=\s_i\cup\s_j$ is a simple closed curve in $X$, and we have $h_1+h_2={\rm Im}\omega(c)$. 
 

 \item[(b)]{\bf Case 2:} $\omega$ has $3$ zeros. In this case two zeros are permuted by $\inv$, we denote them by $P_1,P_2$, the third one is fixed by $\inv$, let us denote this one by $Q$. We can always assume that $P_1$ is contained in the bottom border of $\CCC_i$, but not in the top border of $\CCC_i$.  
 
 Assume that the top border of $\CCC_i$ contains $P_2$, and let $\s_i$ be saddle connection in $\CCC_i$ which joins $P_1$ to $P_2$. If there exists another cylinder whose bottom border contains $P_2$ and top border contains $P_1$ then we are done. Otherwise, there must exists a cylinder $\CCC_j$ whose bottom border contains $P_2$ and top border contains $Q$. Let $\CCC_{j'}$ be the cylinder which is permuted with $\CCC_{j}$ by $\inv$, then the top border of $\CCC_{j'}$ contains $P_1$ and the bottom border of $\CCC_{j'}$ contains $Q$. In particular, we have $\CCC_{j'}\neq \CCC_i$.

If $\CCC_{j'}=\CCC_j$, then the top border of $\CCC_j$ contains $P_1$ contradicting our hypothesis. Thus we have $\CCC_{j'} \neq \CCC_j$. Let $\s_j$ be a saddle connection in $\CCC_j$ which joins $P_2$ to $Q$, then $\inv(\s_j)$ is a saddle connection in $\CCC_{j'}$ that joins $Q$ to $P_1$. Consequently, $c:=\inv(\s_j)\cup\s_j\cup \s_i$ is a simple closed curve in $X$, and ${\rm Im}\omega(c)=h_i+h_j+h_{j'}=h_i+2h_j$.
 
We are left with the case where the top border of $\CCC_i$ contains $Q$. Let $\CCC_{i'}$ be the cylinder which is permuted with $\CCC_i$ by $\inv$, then the top border of $\CCC_{i'}$ contains $P_2$ and the bottom border contains $Q$. By assumption, we have $\CCC_{i'} \neq \CCC_i$. By connectedness of $X$, there exists a cylinder $\CCC_j \neq \CCC_i$ which contains $P_1$ in the top border, and $P_2$ or $Q$ in the bottom border. If $P_2$ is contained in the bottom border of $\CCC_j$ then $h_j+h_i+h_{i'}=h_j+2h_i$ is an absolute period. If $Q$ is an contained in the bottom border of $\CCC_j$ then $h_i+h_j$ is an absolute period. 
\end{itemize}

Since $x_ih_i+x_jh_j$ is an absolute period, it is unchanged by the kernel foliation, Lemma~\ref{lm:ker:ver:height} then implies that $x_i\alpha_i+x_j\alpha_j=0$. 
\end{proof}

\begin{Lemma}\label{lm:cyl:dec:3class:rel}
 Assume that $\CCC_1,\CCC_2,\CCC_3$ are distinct up to permutation by the Prym involution $\inv$. Then there exists $(r_1,r_2,r_3) \in \Q^3 \setminus \{(0,0,0)\}$ such that  
 \begin{equation}\label{eq:rat:rel:mod}
  r_1\mu_1+r_2\mu_2+r_3\mu_3=0
 \end{equation}
\noindent and
\begin{equation}\label{eq:rat:rel:coeff}
 r_1\frac{\alpha_1}{w_1} +r_2\frac{\alpha_2}{w_2} + r_3\frac{\alpha_3}{w_3}=0.
\end{equation} 
\end{Lemma}

\begin{proof}
 By Lemma~\ref{lm:nbr:cyl:max3}, we know that the cylinder decomposition is stable. Thus we can associate to each cylinder $\CCC_i$ a coefficient $\alpha_i\in \{0,\pm 1/2,\pm 1\}$. We first observe that moving in the leaves of the kernel foliation does not change the area of the surface, therefore 
 $$
 \Aa(X,\omega)=\Aa((X,\omega)+(0,s)) \quad \Rightarrow \quad \sum_{i=1}^k w_ih_i = \sum_{i=1}^k w_i(h_i+\alpha_is)
 $$ 
 which implies 
 \begin{equation}\label{eq:w:relation}
 \sum_{i=1}^k \alpha_iw_i=\sum_{i=1}^3\alpha_i\beta_iw_i=0  
 \end{equation}
where $\beta_i=1$ if $\CCC_i$ is fixed by $\inv$, and $\beta_i=2$ otherwise. We have two cases:

\begin{itemize}
\item[(a)] $D$ is a square. In this case we can normalize, using $\GL^+(2,\R)$, so that all the absolute periods of $\omega$ belong to $\Q(\imath)$. By Lemma~\ref{lm:cyl:dec:3class:rel}, there exist $j \in \{1,\dots,k\}$ and $a\in \{1,2\}, b\in \{0,1,2\}$ such that $ah_1+bh_j$ is an absolute period. Since $\CCC_j$ is permuted with one of the cylinders $\CCC_1,\CCC_2,\CCC_3$, we can assume that $ah_1+bh_3$ is an absolute period. Similarly, there exist $j\in \{1,3\}$ and $c,d\in \N, \, c \neq 0$ such that $ch_2+dh_j$ is an absolute period.  Let us assume that $j=3$. Since all the absolute periods are in $\Q$, there exists $\lbd\in \Q, \lbd >0$, such that $ah_1+bh_3=\lbd(ch_2+dh_3)$. Thus we have
$$
aw_1\mu_1-\lbd c w_2\mu_2+(b-\lbd d)w_3\mu_3=0.
$$
\noindent Set $r_1=aw_1,r_2=-\lbd c w_2, r_3=(b-\lbd d)w_3$. We have $(r_1,r_2,r_3)\in \Q^3$ and $r_1r_2\neq 0$. Since $(X,\omega)$ and $(X,\omega)+(0,s)$ have the same absolute periods, we have $a(h_1+\alpha_1s)+b(h_3+\alpha_3s)=\lbd(c(h_2+\alpha_2s)+d(h_3+\alpha_3s))$ which implies $a\alpha_1+b\alpha_3=\lbd(c\alpha_2+d\alpha_3)$. Consequently 
$$
r_1\frac{\alpha_1}{w_1}+r_2\frac{\alpha_2}{w_2} +r_3\frac{\alpha_3}{w_3}= a\alpha_1-\lbd c \alpha_2+(b-\lbd d)\alpha_3=0.
$$
\item[(b)] $D$ is not a square. In this case $K=\Q(\sqrt{D})$ is a quadratic field. It follows from Corollary~\ref{cor:parabolic} that we have 
$$
\sum_{i=1}^3\beta_i N(w_i)\mu_i=0
$$
\noindent where $N(w_i)=w_iw'_i$, and $w'_i$ is the Galois conjugate of $w_i$ in $K$. Set $r_i=\beta_iN(w_i)=\beta_iw_iw'_i\in \Q$. Clearly, $r_i\neq 0, \, i=1,2,3$. We have 
$$
\sum_{i=1}^3r_i\frac{\alpha_i}{w_i}=\sum_{i=1}^3\beta_iw'_i\alpha_i.
$$ 
Since $\alpha_i \in \Q$ and $\beta_i\in \Q$, it follows
$$
\sum_{i=1}^3\beta_iw'_i\alpha_i=\left( \sum_{i=1}^3 \alpha_i\beta_iw_i\right)'=0
$$
\noindent where the last equality follows from \eqref{eq:w:relation}. The lemma is then proved.
\end{itemize}
\end{proof}

\medskip

By Corollary~\ref{cor:parabolic}, we know that, when $D$ is not a square, if the cylinder decomposition is unstable, then the direction is parabolic. Let us now prove the same statement for the case $D$ is a square.

\begin{Lemma}\label{lm:Dsq:2class:unstable:par}
Suppose that $D$ is a square. Then if the cylinder decomposition is unstable, then the horizontal direction is parabolic.
\end{Lemma}
\begin{proof}
If there are $3$ distinct cylinders up to permutation by the Prym involution then the decomposition is stable. Therefore, we can assume that $\CCC_1$ and $\CCC_2$ are not permuted by $\inv$, and any other cylinder is permuted with either $\CCC_1$ or $\CCC_2$. We can normalize so that all the absolute periods of $\omega$ are in $\Q(\imath)$. 
 
If both $h_1,h_2$ are absolute periods then we are done, because all the moduli are rational numbers. Thus, without loss of generality, let us assume that $h_1$ is not an absolute period. By Lemma~\ref{lm:cyl:h:abs:per}, there exists $x_1,x_2 \in \N$ such that $x_1h_1+x_2h_2 $ is an absolute period. In particular, $x_1h_1+x_2h_2\in \Q$. By assumption, both $x_1,x_2$ are none-zero. We have two cases: 
\begin{itemize}
\item[(a)]{\bf Case 1:} $\omega$ has two zeros $P_1,P_2$. We can assume that $P_1$ is contained in the bottom border of $\CCC_1$ and $P_2$ is contained in the top border of $\CCC_1$. Let $\s$ be a saddle connection in $\CCC_1$ which joins $P_1$ to $P_2$. Since the cylinder decomposition is unstable, there exists a horizontal saddle connections $\gamma$ from $P_2$ to $P_1$. Thus $c:=\gamma\cup \s$ is a simple closed curve in $X$ and $h_1={\rm Im}\omega(c)$. Thus $h_1\in \Q$, which implies that $h_2\in \Q$, and the horizontal direction is parabolic.
  
\item[(b)]{\bf Case 2:} $\omega$ has $3$ zeros. Let $P_1,P_2$ denote the zeros which are permuted, and $Q$ be the zero fixed by $\inv$. We first observe that there exists a path from $P_1$ and $P_2$ which is a union of horizontal saddle connection. Indeed, by assumption there exists a horizontal saddle connection $\gamma$ which joins two different zeros. If $\gamma$ joins $P_1$ to $P_2$ then we are done. Otherwise, $\gamma$ joins $Q$ to either $P_1$ or $P_2$. In both case cases, the union of $\gamma$ and $\inv(\gamma)$ is the desired path. Let us denote this path by $\eta$.
  
Without loss of generality, let us assume that $P_1$ is contained in the bottom border of $\CCC_1$. If the top border of $\CCC_1$ contains $P_2$, then the union of $\eta$ and a saddle connection in $\CCC_1$ joining $P_1$ to $P_2$ is a simple closed curve $c$ such that ${\rm Im}\omega(c)=h_1$. Therefore $h_1\in \Q$, and the lemma follows.
 
If the top border of $\CCC_1$ contains $Q$, then let $\CCC_3$ be the cylinder which is permuted  with $\CCC_1$ by $\inv$. Note that the bottom border of $\CCC_3$ contains $Q$, and the top border of $\CCC_3$ contains $P_2$ (in particular $\CCC_3\neq \CCC_1$, by assumption). Let $\s_1$ be a saddle connection in $\CCC_1$ joining $P_1$ to $Q$, and $\s_3$ be the image of $\s_1$ by $\inv$. The union $c:=\eta\cup \s_3\cup \s_1$ is then a closed curve such that ${\rm Im}\omega(c)=2h_1\in \Q$. Hence the lemma follows from the same argument.
\end{itemize}
\end{proof}

\section{Proof of a weaker version of Theorem~\ref{theo:main}}
\label{sec:weaker}

In this section, we prove a weaker version of Theorem~\ref{theo:main}. We say that 
$(X,\omega)$ is not a Veech surface (or the orbit is not closed) for ``{\em the most obvious reason}'' if
there exists a completely periodic direction on $(X,\omega)$ that is not parabolic
(it is a theorem of Veech~\cite{Veech1989} that if the orbit is closed then any completely periodic direction is parabolic). \medskip

We will prove a weaker version of Theorem~\ref{theo:main} under this additional assumption:

\begin{Theorem}
\label{theo:weaker}
Let $(X,\omega)\in \Omega E_{D}(\kappa)$ and let us denote by $\mathcal O$ its 
$\textrm{GL}^+(2,\R)$-orbit. If $\mathcal O$ is not closed for the most obvious reason
then $\overline{\mathcal O}$ is a connected component of $\Omega E_{D}(\kappa)$.
\end{Theorem}
We begin with the following key lemma. The proof is classical, but is 
included here for completeness.
\begin{Lemma}
\label{lm:key}
Let $(X,\omega) \in \Omega E_D(\kappa)$ be a Prym eigenform. We assume 
that the horizontal direction is completely periodic but not parabolic. 
Then for all $s \in \R$, the surface $(X,\omega)+(s,0)$ is well defined, and one has:
$$
(X,\omega) +(s,0) \in \overline{U\cdot (X,\omega)}.
$$
\end{Lemma}
Before proving the lemma, let us state the following corollary:
\begin{Corollary}
\label{cor:1}
Let $(X,\omega) \in \Omega E_D(\kappa)$ be a Prym eigenform. We assume 
that there exists $(Y,\eta)\in\overline{\GL^+(2,\R)\cdot (X,\omega)}$ and $\eps >0$ such that 
$(Y,\eta) +(s,0) \in \overline{\GL^+(2,\R)\cdot (X,\omega)}$
for all $s \in \R$ with $|s| < \eps$. Then there exists $\eps' >0$ such that
$$
(Y,\eta) +v \in \overline{\GL^+(2,\R)\cdot (X,\omega)}
$$
for any  $v \in \R^2$ and  $v\in \B(\varepsilon')$.
\end{Corollary}

\begin{proof}[Proof of Lemma~\ref{lm:key}]
Let $\CCC_1,\dots,\CCC_k$ denote the horizontal cylinders  of $X$. Let $n$ be the number of equivalence classes 
of cylinders that are permuted by  the Prym involution $\inv$. For all the cases in Table~\ref{tab:strata:list}, 
we have $n \leq 3$. 

Let us consider the case $n=3$. Lemma~\ref{lm:nbr:cyl:max3} implies in particular that the cylinder decomposition is stable. 
Hence the horizontal kernel foliation is well defined for all time $s$.

The surface is encoded by the topological gluings of the cylinders $\CCC_i$, and 
the width, height, and twist of $\CCC_i$  (which will be denoted by $w_i,h_i,t_i$ respectively). We choose the 
numbering so that $\CCC_1,\CCC_2,\CCC_3$ are distinct up to Prym involution. The set of surfaces admitting a cylinder 
decomposition in the horizontal direction with the same topological gluings, and the same widths and heights of the cylinders, 
are parameterized by the three dimensional torus 
$$
\XX=N(\R)\times N(\R)\times N(\R)/N(w_1\Z)\times N(w_2\Z)\times N(w_3\Z),
$$ 
where $N(A)=\{u_s;\  s \in A \}$.

The horocycle flow $u_s$ acts on $(X,\omega)$ by preserving the topological decomposition as
well as all the parameters, but the twists $t_i$: the new twists $\widetilde{t_i}$ are given by 
$\widetilde{t_i}=t_i+sh_i \mod w_i$. Hence surfaces in the $U$-orbit of 
$(X,\omega)$ are parameterized by the line $\{(t_1,t_2,t_3)+(h_1,h_2,h_3)s, \ s \in \R\}$. 

By Kronecker's theorem, the orbit closure $\overline{U\cdot (X,\omega)}$ is a subtorus of $\XX$. 
Since the moduli are not commensurable (the horizontal direction is not parabolic) the dimension of 
this subtorus is at least two. More precisely, the orbit closure $\overline{U\cdot (X,\omega)}$ consists of the set of all twists
$(\widetilde{t_1},\widetilde{t_2},\widetilde{t_3})$ such that the normalized twists $\cfrac{\widetilde{t_i}-t_i}{w_i}$ 
verify all non-trivial homogeneous linear relations with rational coefficients that are satisfied by the moduli $\mu_i = h_i/w_i$. 
Let ${\mathbb P}$ be the subspace of $\R^3$ which is defined by all of such rational relations. By assumption, we have $\dim_\R{\mathbb P} \geq 2$.
But we know from Lemma~\ref{lm:cyl:dec:3class:rel} that there exists $(r_1,r_2,r_3)\in \Q^3\setminus\{(0,0,0)\}$ such that 
$ \sum_{i=1}^{n} r_i\mu_i = 0$ (Equation~\eqref{eq:rat:rel:mod}). Therefore, we have $\dim_\R{\mathbb P}=2$ and 
\begin{equation}
\label{eq:calta:twist}
\sum_{i=1}^{3} r_i\left(\cfrac{\widetilde{t_i}-t_i}{w_i}\right)= 0.
\end{equation}

\noindent It follows that $\ol{U\cdot(X,\omega)}$ is the projection to $\XX$ of the plane ${\mathbb P} \subset \R^3$ defined 
by Equation~\eqref{eq:calta:twist}. Hence, all surfaces constructed from the cylinders with the same widths and heights as those of $(X,\omega)$ (by the same gluings),
and with the twists $\widetilde{t_i}$ satisfying Equation~\eqref{eq:calta:twist} above belong to $\overline{U\cdot (X,\omega)}$. \medskip

Recall that in the horizontal kernel foliation leaf, a surface $(X,\omega)+(s,0)$ is still completely periodic 
(for the horizontal direction), and all the data: topological gluings of the cylinders, widths, heights are preserved, except the twists (see Lemma~\ref{lm:ker:hor:twist}). 
To be more precise, if $C_i^s$ is the horizontal cylinder in $(X,\omega)+(s,0)$ corresponding to 
$C_i=C^0_i$, then $t_i(s)=t_i+\alpha_is$ (where the range of $\alpha_i$ is $\{-1,0,1\}$ or 
$\{-1,-1/2,0,1/2,1\}$ depending whether $\omega$ has $2$ or $3$ zeros, respectively). It remains to show that $(t_1+\alpha_1s,t_2+\alpha_2s,t_3+\alpha_3s)=(t_1,t_2,t_3)+(\alpha_1,\alpha_2,\alpha_3)s$ belongs to ${\mathbb P}$. But this is an immediate consequence of Equation~\eqref{eq:rat:rel:coeff}. Thus the lemma is proved for the case $n=3$.

\medskip

Let us now consider the case $n=2$. Note that if $D$ is not a square then the horizontal direction is parabolic in this case (see Corollary~\ref{cor:parabolic}). Therefore,  $D$ must be a square. By Lemma~\ref{lm:Dsq:2class:unstable:par} we know that the cylinder decomposition is stable, which implies that $(X,\omega)+(s,0)$ is defined for all $s$. Without loss of generality, we can assume that $\CCC_1$ and $\CCC_2$ are distinct up to permutation by $\inv$. In this case, the closure of $U\cdot(X,\omega)$ can be identified with the torus
$$
\XX'=N(\R)\times N(\R) / N(w_1\Z)\times N(w_2\Z)
$$
Using this identification, the horizontal kernel foliation leaf through $(X,\omega)$ corresponds to the projection of the affine line $\{(t_1,t_2)+(\alpha_1,\alpha_2)s, \, s \in \R\}$.  
Hence
$$
(X_s,\omega_s) = (X,\omega) + (s,0) \in \overline{U\cdot (X,\omega)},
$$
which concludes the proof of Lemma~\ref{lm:key}.
\end{proof}

\begin{proof}[Proof of Corollary~\ref{cor:1}]
We will apply Lemma~\ref{lm:key} to a transverse direction to $(1:0)$. By Theorem~\ref{thm:cp}, let $\theta$ be a completely periodic direction on $Y$  which is transverse to the horizontal direction. Up to action of $\GL^+(2,\R)$, we can assume that $\theta=(0:1)$.  

By Proposition~\ref{prop:stable:dec} and Proposition~\ref{prop:cyl:dec:unstable}, there exists $\eps>0$ such that the direction $(0:1)$ is  still completely periodic on $(Y,\eta)+(s,0)$ for all $|s| < \varepsilon$, and if $s\neq 0$ the cylinder decomposition of $(Y,\eta)+(s,0)$ in the direction of $(0:1)$ is stable. Moreover, the combinatorial data of this decomposition is unchanged when $s$ varies in the intervals $(-\eps,0)$ and $(0,\eps)$, if the decomposition of $(Y,\eta)$ is stable, then we have the same combinatorial data for all $s \in (-\eps,\eps)$. 

Let $\left\{w_i(s)\right\}_{i=1,\dots,k}$ and $\left\{h_i(s)\right\}_{i=1,\dots,k}$ be the widths and heights of the cylinders 
in the {\em vertical} direction of $(Y,\eta)+(s,0), \; s \neq 0$. Note that the functions  $w_i(s)$ are constant on each of intervals $(-\eps,0)$ and $(0,\eps)$. However, the set of heights $h_i(s)$ define non constant continuous functions of $s$. To be more precise, $h_i(s) = h_i + \alpha_i s$, where $\alpha_i \in \{-1,0,1\}$ or $\alpha_i \in \{-1,-1/2,0,1/2,1\}$ depending on whether $\eta$ has two or three zeros. Obviously, at least two of $\alpha_i$ are different. Hence the set of moduli 
$$
\mu_i(s) = \frac{h_i+s\alpha_i}{w_i}
$$ 
of cylinders (in the vertical direction) define also non constant continuous functions of $s$. In particular for almost every $s$ in $(-\eps,0)$ (resp. $(0,\eps)$), the direction 
$(0:1)$ is completely periodic and not parabolic on $(Y,\eta)+(s,0)$. Hence Lemma~\ref{lm:key} applies in that vertical direction: for any $t \in \R$ one has $(Y,\eta)+(s,t)\in \overline{\GL^+(2,\R)\cdot ((Y,\eta)+(s,0))}$. It follows immediately that we have 
$(Y,\eta)+v\in  \overline{\GL^+(2,\R)\cdot (X,\omega)}$ 
for every $v=(s,t)$ with $|s| < \eps$ and $|t| < c_{\rm min}$, where $c_{\rm min}$ is the length of the smallest vertical saddle connections in $(Y,\eta)$ joining two different zeros. This ends the proof of Corollary~\ref{cor:1}.
\end{proof}

One can now prove the main result of this section.

\begin{proof}[Proof of Theorem~\ref{theo:weaker}]
We will show that any $(Y,\eta) \in \overline{\GL^+(2,\R)\cdot (X,\omega)}=\ol{\Orb}$ has 
an open neighborhood in $\overline{\GL^+(2,\R)\cdot (X,\omega)}$. \medskip

We first show that claim for surfaces in $\GL^+(2,\R)\cdot(X,\omega)$, that is for $(Y,\eta)=g\cdot(X,\omega), \, g \in \GL^+(2,\R)$. By assumption, there exists a periodic direction for $(X,\omega)$ which is not parabolic. Lemma~\ref{lm:key} and Corollary~\ref{cor:1} then imply that there exists $\eps >0$ such that $(X,\omega)+v \in \ol{\Orb}$ for any vector $v \in \R^2$ with $v\in \B(\varepsilon)$. It follows that for all $g \in \GL^+(2,\R)$, $g\cdot(X,\omega)+v \in \ol{\Orb}$ if $||v|| < \varepsilon||g^{-1}||^{-1}$. Thus there exist $\eps_0>0$ and a neighborhood ${\mathcal U}$ of $\Id$ in $\GL^+(2,\R)$  such that $g\cdot(X,\omega)+v \in \ol{\Orb}$, for any $(g,v)\in {\mathcal U}\times \B (\eps_0)$. 
But by Proposition~\ref{prop:local} the set $\{g\cdot(X,\omega) +v, \; (g,v) \in {\mathcal U}\times\B(\eps_0)\}$ is a neighborhood of $(X,\omega)$ in $\PrymD(\kappa)$. The claim is then proved for $(X,\omega)$ and hence for all $(Y,\eta)\in \Orb = \GL^+(2,\R)\cdot(X,\omega)$. \medskip

We now assume that $(Y,\eta)$ is not in the $\GL^+(2,\R)$-orbit of  $(X,\omega)$ and 
we let $(X_n,\omega_n) = g_n \cdot(X,\omega)$ be a sequence converging to $(Y,\eta)$ with $g_n \in\GL^+(2,\R)$.
For $n$ large enough by Proposition~\ref{prop:local} there exists a pair $(a_n,w_n)$, 
where $a_n\in \GL^+(2,\R)$ close to $\Id$, and $w_n\in \R^2$ with $|w_n|$ small, such that 
$(X_n,\omega_n)=a_n(Y,\eta) + w_n$. Hence, up to replacing $g_n$ by $a_n^{-1}g_n$, and up to taking a subsequence, 
we can assume that for $(X_n,\omega_n)=(Y,\eta) + v_n$ where $v_n = a_n^{-1}w_n$  satisfy $v_n \rightarrow 0$ as $n \rightarrow \infty$. Without loss of generality, we also assume that the horizontal direction is completely periodic on $Y$. 

By Propositions~\ref{prop:stable:dec} and~\ref{prop:cyl:dec:unstable}, we can choose $\eps >0$ such that for all 
$v=(s,t) \in \B(\eps)$ the surface $(Y,\eta)+v$ also admits a cylinder decomposition in the horizontal direction.
When $t\neq 0$ this decomposition is stable with combinatorial data  depending only on the sign of $t$. 
We can assume that $v_n\in \B(\eps)$.

Now, since $(X_n,\omega_n) \in \GL^+(2,\R)\cdot(X,\omega)$, we know that $\ol{\Orb}$ contains a neighborhood of $(X_n,\omega_n)$ by the argument above, in particular, for each $n$ there exists $\eps_n>0$ such that $(X_n,\omega_n)+ v \in \overline{\Orb}$  for any 
$v\in \B(\varepsilon_n)$. Note that $(X_n,\omega_n)+v=(Y,\eta)+v_n+v$. For each $n$ choose a $\delta_n \in (0,\eps_n)$ small enough such that
\begin{itemize}
 \item[(a)] $u_n=v_n+(0,\delta_n) \in \B(\eps)$.
 \item[(b)] If $v_n=(s_n,t_n)$ with $t_n \neq 0$, then $t_n+\delta_n$ and $t_n$ have the same sign.
 \item[(c)] The horizontal direction is not parabolic for $(X'_n,\omega'_n)=(X_n,\omega_n)+(0,\delta_n)=(Y,\eta)+u_n$.
 \item[(d)] $\delta_n \rightarrow 0$ as $n \rightarrow \infty$.
\end{itemize}
By definition, we have $(X'_n,\omega'_n) \in \ol{\Orb}$, and $(X'_n,\omega'_n)$ converges to $(Y,\eta)$. Since the horizontal direction is not parabolic for $(X'_n,\omega'_n)$, it follows from Lemma~\ref{lm:key} that $(X'_n,\omega'_n)+ (s,0) \in \ol{\Orb}$ for any $s \in (-\eps,\eps)$. Hence passing to the limit as $n$ tends to infinity, we get that
$$
(Y,\eta)+(s,0) \in \overline{\GL^+(2,\R)\cdot (X,\omega)} \qquad \textrm{ for all }|s| < \varepsilon.
$$ 
Corollary~\ref{cor:1} then implies the theorem.
\end{proof}

\section{Proof of Theorem~\ref{theo:main}}
\label{sec:proof:main:th}

In this section we complete the proof of Theorem~\ref{theo:main} in full generality, 
namely without the assumption that the orbit $\mathcal O := \GL^+(2,\R)\cdot (X,\omega)$ 
is not closed ``for the most obvious reason''. However our proof says nothing about the 
converse of this assumption, {\em i.e.} the following question remains open in our setting:

\begin{Question}
For an orbit $\mathcal O := \GL^+(2,\R)\cdot (X,\omega)$, does the property of being not closed 
is equivalent to be not closed ``for the most obvious reason''?
\end{Question}


\begin{proof}[Proof of Theorem~\ref{theo:main}]
We first begin by fixing some notations and  normalization. As usual, 
let $(X,\omega)\in \Omega E_D(\kappa)$ and let us assume that $\mathcal O := \GL^+(2,\R)\cdot (X,\omega)$ 
is not closed.  Let $(Y,\eta)\in \overline{\mathcal O}\ \backslash\ \mathcal O$ be some translation surface in the orbit closure, but not in the orbit itself.
\begin{Claim}
\label{claim:1}
There exists a sequence $(X_n,\omega_n)$, where $(X_n,\omega_n)= (Y,\eta) + v_n \in \mathcal O$ and $v_n=(x_n,y_n)$, 
that converges to $Y$ so that $y_n\not =0$ for every $n$.
In addition one can always make the assumption that the horizontal direction on $Y$ is 
completely periodic.
\end{Claim}
\begin{proof}[Proof of the claim]
We choose a sequence $(X_{n},\omega_n)\in \mathcal O$ converging to $(Y,\eta)$. As in the proof of Theorem~\ref{theo:weaker} 
we can assume that $(X_{n},\omega_n) = (Y,\eta)+v_{n}$ where $v_n=(x_n,y_n)$ and $v_{n}\in \B(\varepsilon)$.

Again, up to replace $Y$ by $R_\theta \cdot Y$ for some suitable $\theta$, without loss of generality, we will 
also assume that the horizontal direction is completely periodic on $Y$. If $y_n \not = 0$ infinitely 
often then the claim follows by taking a subsequence. Otherwise we assume that $y_n=0$ for every $n>N$.
We choose another (transverse) completely periodic direction on $Y$ (that we can assume to be vertical,
up to the action by some matrix $R_\theta$). Then 
up to replace $(Y,\eta)$ and $(X_n,\omega_n)$ respectively by $R_{\pi/2} \cdot (Y,\eta)$ and $R_{\pi/2} \cdot (X_n,\omega_n)$ the claim
is proved (otherwise $x_n= 0$ for $n$ large enough, thus $(Y,\eta)=(X_n,\omega_n)\in \mathcal O$ that is a contradiction 
to our assumption).
\end{proof}

We choose some $\eps>0$ so that for any $v=(x,y) \in \R^2$, if $v\in \B(\varepsilon)$ then the horizontal direction on $(Y,\eta)+v$ is periodic, and the cylinder decomposition is stable if $y\neq 0$. 
We can assume that  $v_n\in \B(\varepsilon)$ and $y_n>0$ for all $n$, which implies that the combinatorial data of the cylinder decomposition in the horizontal direction of $(X_n,\omega_n)$ are the same for all $n$. Finally we also assume that {\it all} the horizontal directions on $X_n$ are parabolic (otherwise we are done by
Theorem~\ref{theo:weaker}). \medskip

We sketch the idea of the proof. It makes use of the horocycle flow $u_s$ acting on
$X_n$. The key is to show that the actions of the kernel foliation and $u_{s}$ coincide for a subsequence.
\begin{enumerate}
\item Since {\em all} the horizontal directions on $X_n$ are parabolic, we will show that  it is always possible to 
find a ``good time'' $s_{n}$ so that $u_{s_{n}}\cdot X_{n} = X_{n} + (x_n,0)$ for some vector $(x_n,0)\in \R^2$.
\item One can arrange that $(x_{n},0)$ converges to some arbitrary  vector, say $(x,0)\in \R^{2}$, with $|x|$ small.
\end{enumerate}
These two facts correspond, respectively, to Claim~\ref{claim:2} and Claim~\ref{claim:3} below. Once we achieve this, passing to the limit as $n\to \infty$, we get
$$
u_{s_{n}} \cdot (X_n,\omega_n) = (X_{n},\omega_n) + (x_{n},0) \longrightarrow (Y,\eta) + (x,0).
$$
In other words $(Y,\eta) + (x,0) \in \overline{\mathcal O}$ for all  $x\in (-\eps',\eps')$.  Then Corollary~\ref{cor:1} applies and this gives some $\eps'' >0$ so that 
$(Y,\eta) +v \in \overline{\mathcal O}$ for any $v\in \B(\varepsilon'')$ which proves the theorem. \medskip

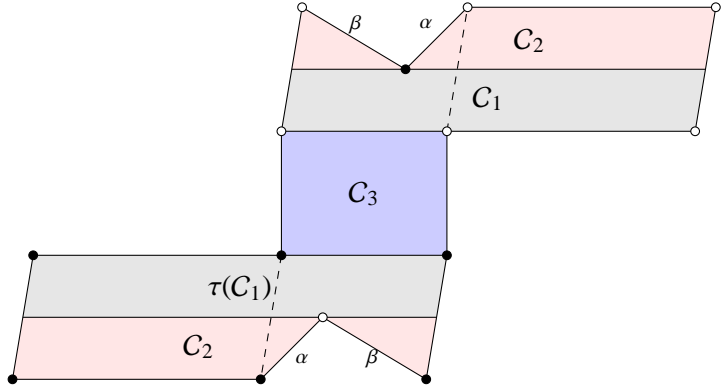
\begin{figure}[htbp]

\begin{minipage}[t]{0.9\linewidth}
\centering
\begin{tikzpicture}[scale=1.1]

\fill[fill=blue!20] (0,0.5) -- (2,0.5) -- (2,2) -- (0,2) -- (0,0);
\fill[fill=gray!20] (0,2) -- (5,2) -- (5.125,2.75) -- (0.125,2.75) -- (0,2);
\fill[fill=gray!20] (2,0.5) -- (-3,0.5) -- (-3.125,-0.25) -- (1.875,-0.25) -- (2,0.5);
\fill[fill=red!10] (0.125,2.75) -- (1.5,2.75) -- (0.25,3.5) -- (0.125,2.75);
\fill[fill=red!10] (1.5,2.75) -- (5.125,2.75) -- (5.25,3.5) -- (2.25,3.5) -- (1.5,2.75);
\fill[fill=red!10] (-3.125,-0.25) -- (0.5,-0.25) -- (-0.25,-1) -- (-3.25,-1) -- (-3.125,-0.25);
\fill[fill=red!10] (1.875,-0.25) -- (0.5,-0.25) -- (1.75,-1) -- (1.875,-0.25);

\draw (0,0.5) -- (2,0.5) -- (2,2) -- (5,2) -- (5.25,3.5) -- (2.25,3.5) -- (1.5,2.75) -- (0.25,3.5) -- (0,2) -- (0,0.5)
-- (-3,0.5) -- (-3.25,-1) -- (-0.25,-1) -- (0.5,-0.25) -- (1.75,-1) -- (2,0.5);

\draw (0.125,2.75) -- (5.125,2.75);
\draw (-3.125,-0.25) -- (1.875,-0.25);
\draw[-, dashed] (2,2) -- (2.25,3.5);
\draw[-, dashed] (0,0.5) -- (-0.25,-1);
\draw (0,2) -- (2,2);
\foreach \x in {(0,2),(2,2),(5,2),(5.25,3.5),(2.25,3.5),(0.25,3.5),(0.5,-0.25)} \filldraw[fill=white] \x circle(1.5pt);
\foreach \x in {(1.5,2.75),(0,0.5),(2,0.5),(-3,0.5),(-3.25,-1),(-0.25,-1),(1.75,-1)} \filldraw[fill=black] \x circle(1.5pt);

\draw (1,1.25) node {$\mathcal C_3$};
\draw (2.50,2.40) node {$\mathcal C_1$};
\draw (-0.5,0.1) node {$\inv(\mathcal C_1)$};
\draw (3,3.1) node {$\mathcal C_2$};
\draw (-1,-0.6) node {$\mathcal C_2$};

\foreach \x in {(1.75,3.3),(0.25,-0.75)} \draw \x node {$\scriptstyle \alpha$};
\foreach \x in {(0.90,3.3),(1.10,-0.75)} \draw \x node {$\scriptstyle \beta$};
\end{tikzpicture}
\end{minipage}

\caption{Complete periodic decomposition into four cylinders of $(X_n,\omega_n)=(Y,\eta)+v_n$ near $(Y,\eta)\in \PrymD(2,2)$
where $v_n=\int_\alpha \omega$.
The cylinders $\mathcal C_2$ and $\mathcal C_3$ are fixed by the Prym involution $\inv$, while 
the cylinders $\mathcal C_1$ and $\inv(\mathcal C_1)$ are exchanged. When $v_n \rightarrow 0$ the 
cylinder $\mathcal C_2$ is destroyed, while $\CCC_3$ is remains in the limit (here 
we have assumed that $h_3 > h_2$).
\label{fig:decomposition}
}
\end{figure}


Remark that a stable cylinder decomposition may have $3$ or $2$ cylinders up to permutation by the Prym involution, where the latter case only occurs when $D$ is a square. In what follows, we will only give the proof for the case where we have $3$ cylinders since the other case can be proved with similar ideas and simpler arguments. \medskip

We now explain how to construct the sequence $(s_n)_{n\in \N}$. As usual, the cylinders on $X_n$ are
denoted by $\CCC^{(n)}_i, i=1,\dots,k$ (the numbering is such that for every $i\in\{1,2,3\}$, 
$\CCC^{(n)}_j = \inv(\CCC^{(n)}_i)$ implies $j=i$ or $j>3$). The width, height, twist, and modulus of $\CCC^{(n)}_i$ are denoted by $w^{(n)}_i,h^{(n)}_i,t^{(n)}_i,\mu^{(n)}_i$ respectively. Recall that by Proposition~\ref{prop:stable:dec} and Proposition~\ref{prop:cyl:dec:unstable}, we have $ w^{(n)}_i$ does not depend on $n$, therefore we can write $w^{(n)}_i=w_i$. Let us define
$$
h_i^\infty=\lim_{n\rightarrow \infty} h^{(n)}_i.
$$
Since the cylinder decomposition of $X_n$ is stable, we can associate to  each family of cylinders $(\CCC^{(n)}_i)_n$ a coefficient $\alpha_i \in \{0,\pm 1/2,\pm 1\}$.
Recall that the kernel foliation action of a vector  $v=(x,y)$ changes the height $h^{(n)}_i$ of $\CCC^{(n)}_i$ to $h^{(n)}_i+\alpha_i y$, hence we can write 
$$
h^{(n)}_i=h^\infty_i+\alpha_i y_n.
$$
Note that the horizontal direction on $Y$ is not necessarily stable, some horizontal cylinders on 
$X_n$ can be destroyed in the limit (as $n$ tends to infinity). Therefore, some of the limits $h^\infty_i$ may be zero. However, there is at least one 
cylinder that remains in the limit, say it is $\mathcal C^{(n)}_3$ (see Figure~\ref{fig:decomposition}
where the cylinder $\mathcal C^{(n)}_2$ is destroyed when performing the kernel foliation). 
Actually, since $(X_n,\omega_n)$ stays in a neighborhood of $(Y,\eta)$, all the cylinders of $(Y,\eta)$ persist in $(X_n,\omega_n)$. Thus, the number of horizontal cylinders of $(X_n,\omega_n)$ is always greater than $(Y,\eta)$. We  denote by $\CCC_3$ the cylinder on $Y$ corresponding to $\CCC^{(n)}_3$
on $X_n$, then the height of $\CCC_3$ is $h^\infty_3$. In particular, we have $h^\infty_3>0$. 

From Lemma~\ref{lm:cyl:dec:3class:rel}, Equation~\eqref{eq:w:relation}, we have
$$
\sum_{i=1}^3 \beta_iw_i\alpha_i=0.
$$
Since all the $\alpha_i$ can not vanish (otherwise for all $i \in \{1,\dots,k\}$ the upper and lower boundaries of $\CCC^{(n)}_i$ contain the same zero, which means that $\omega$ has only one zero), Equation \eqref{eq:w:relation} implies that there exist $i,j$ in $\{1,2,3\}$ such that $\alpha_i$ and $\alpha_j$ are non zero and have opposite signs. In particular, there exists $i \in \{1,2,3\}$ such that $\alpha_i\neq 0$ and $\alpha_i$ has the opposite sign to $\alpha_3$ if $\alpha_3\neq 0$.   In what follows we suppose that $\alpha_1$ satisfies this condition. By a slight abuse of language, we will say that $\alpha_1$ and $\alpha_3$ have opposite signs. In particular, $(t_1^{(n)},h_1^{(n)})$ is a relative coordinate.  For the surface in Figure~\ref{fig:decomposition:5:cylinders}, $\omega$ has three zeros and $(\alpha_1,\alpha_3) =(-1, 1/2)$, and for the one in Figure~\ref{fig:decomposition}, $\omega$ has two zeros and $(\alpha_1,\alpha_3) =(-1, 1)$. \medskip

Recall that, by Lemma~\ref{lm:cyl:dec:3class:rel}, we know that there exists $(r_1,r_2,r_3)\in \Q^3\setminus \{(0,0,0)\}$ such that
$$
r_1\mu^{(n)}_1+r_2\mu^{(n)}_2+r_3\mu^{(n)}_3=0 \quad \text{ and } \quad r_1\frac{\alpha_1}{w_1}+r_2\frac{\alpha_2}{w_2}+r_3\frac{\alpha_3}{w_3}=0.
$$
Obviously, we can assume that $(r_1,r_2,r_3)\in \Z^3$. Note that $(r_1,r_2,r_3)$ does not depend on $n$. Set $\mu^\infty_i=h^\infty_i/w_i$, by continuity we have
$$
r_1\mu^\infty_1+r_2\mu^\infty_2+r_3\mu^\infty_3=0.
$$
\begin{Claim}
 \label{claim:1b}
 We have $r_2 \neq 0$.
\end{Claim}
\begin{proof}
Suppose that $r_2=0$, we have then 
$$
\left\{ \begin{array}{ccl} r_1\mu^{(n)}_1+r_3\mu^{(n)}_3 & = & 0 \\ r_1 \cfrac{\alpha_1}{w_1} + r_3\cfrac{\alpha_3}{w_3} & = & 0 \end{array} \right.
$$
\noindent Since $\mu^{(n)}_i>0, w_i>0$, and $\alpha_1\alpha_3\leq 0$, this system with unknowns $(r_1,r_3)$ has a unique solution $r_1=r_3=0$. Thus we have a contradiction. 
\end{proof}

From now on, we fix an integral vector $(r_1,r_2,r_3) \in \Z^3$ satisfying Equation~\eqref{eq:rat:rel:mod} and Equation~\eqref{eq:rat:rel:coeff}, with $r_2\neq 0$.


\begin{Claim}
\label{claim:2}
Let $(X,\omega)\in \Omega E_D(\kappa)$ be a surface which admits the same cylinder decomposition as $X_n$ in the horizontal 
direction. We denote by $\CCC_i$ the cylinder in $X$ which corresponds to the cylinder $\CCC^{(n)}_i$ of $X_n$. Let $w_i,h_i,t_i,\mu_i$ be the parameters of $\CCC_i$.
With the notations as above, given two integers $k_1,k_3$, if the real numbers $s$ and $x(s)$ satisfy
\begin{equation}
\label{eq:t(s)}
x(s): = \frac1{\alpha_3}(sh_{3}-r_2k_{3}w_{3}) = \frac1{\alpha_1}( sh_{1} - r_2k_{1}w_{1})
\end{equation}
then $ u_{s} \cdot X = X + (x(s),0)$. 
\end{Claim}

\begin{Remark}
 If $\alpha_3=0$, we replace Equation~\eqref{eq:t(s)} by the following system
 $$
 \left\{ \begin{array}{ccl} sh_3 & = & r_2k_3w_3 \\ x(s) & = &\cfrac{sh_1-r_2k_1w_1}{\alpha_1}. \end{array} \right.
 $$
\end{Remark}

\begin{proof}[Proof of the claim]
On one hand, the kernel foliation $X+(x,0)$, for small values of $x$, maps the twist of the cylinder 
$\mathcal C_i$ to $t_i(x)=t_i  + \alpha_ix$. On the other hand, the action of $u_s$ on the cylinder 
$\mathcal C_i$ maps the twist $t_i$ to the twist $\widetilde{t}_i=t_i + sh_i \mod w_i$.  Equation~\eqref{eq:t(s)} implies
$$
sh_1=\alpha_1x(s)+r_2k_1w_1 \quad \text{ and } \quad sh_3=\alpha_3x(s) + r_2k_3w_3
$$
which is equivalent to
$$\left\{ \begin{array}{ccl}
 s\mu_1 & = & \cfrac{\alpha_1}{w_1}x(s) +r_2k_1\\
 s\mu_3 & = & \cfrac{\alpha_3}{w_3}x(s) +r_2k_3
\end{array}
\right.
$$
Hence,  the twist of the first cylinder of $u_s\cdot X$ is $\widetilde{t}_i=t_i + \alpha_ix(s) \mod w_i$, for $i\in\{1,3\}$. It remains to show that $sh_2=\alpha_2x(s) \mod w_2$. Using Equation~\eqref{eq:rat:rel:mod} and Equation~\eqref{eq:rat:rel:coeff}, we have
$$
-r_2s\mu_2 =-r_2\frac{\alpha_2}{w_2}x(s)+r_2(r_1k_1+r_3k_3).
$$
It follows
$$
sh_2=\alpha_2 x(s)-(r_1k_1+r_3k_3)w_2.
$$
Thus we can conclude that $u_s\cdot(X,\omega)=(X,\omega)+(x(s),0)$.
\end{proof}
Equation~\eqref{eq:t(s)} above reads
\begin{equation}\label{eq:define:sn}
s =r_2\frac{w_{1}k_{1}\alpha_3 - w_{3}k_{3}\alpha_1}{h_{1}\alpha_3-h_{3}\alpha_1}. 
\end{equation}
Note that since $\alpha_1$ and $\alpha_3$ have opposite signs, Equation~\eqref{eq:define:sn} always has a solution.
Reporting this last equation into~\eqref{eq:t(s)}, we derive the new relation:
$$
x(s) =\frac{r_2}{\alpha_3} \left( \frac{w_{1}k_{1}\alpha_3 - w_{3}k_{3}\alpha_1}{h_{1}\alpha_3-h_{3}\alpha_1} h_{3}-k_{3}w_{3} \right) 
= ... = \cfrac{r_2h_{3}w_{1}}{h_{1}\alpha_3 - h_{3}\alpha_1} \left( k_{1} - \cfrac{\mu_{1}}{\mu_{3}}\ k_{3} \right).
$$
We now make the additional assumption  that the horizontal direction is parabolic, \ie the moduli $\mu_i$ are 
all commensurable. We thus write the last expression as:
$$
x(s) = \cfrac{r_2h_{3}w_{1}}{h_{1}\alpha_3 - h_{3}\alpha_1} \left( k_{1} - \cfrac{p}{q}\ k_{3} \right),  \textrm{ where }\
\cfrac{p}{q} = \cfrac{\mu_{1}}{\mu_{3}}\in \Q.
$$
We perform this calculation for each surface $X_n$, so that we get a sequence
\begin{equation}
\label{eq:defining:t}
x_{n} = \cfrac{r_2h^{(n)}_{3}w^{(n)}_{1}}{h^{(n)}_{1}\alpha_3 - h^{(n)}_{3}\alpha_1} \left( k^{(n)}_{1} - \cfrac{p^{(n)}}{q^{(n)}}\ k^{(n)}_{3} \right),
\end{equation}
where $(p^{(n)},q^{(n)})\in \Z^2$ and $\gcd(p^{(n)},q^{(n)})=1$. 
We want to choose suitable pair of integers $(k^{(n)}_{1},k^{(n)}_{3})\in \Z^2$ in order 
to make the sequence $(x_n)_n$ converging to some arbitrary $x$. Let $c_{\rm min}$ be the length of the smallest horizontal saddle connection in $(Y,\eta)$

\begin{Claim}
\label{claim:3}
For any $x \in (-c_{\rm min}, c_{\rm min})$, there exists $(k^{(n)}_{1},k^{(n)}_{2})\in \Z^2$ such that if $x_n$ 
is defined by~\eqref{eq:defining:t} then
$$
\left| x_n - x \right|  <  \cfrac{C}{q^{(n)}},
$$
where $C$ is a constant independent of $n$.
\end{Claim}

\begin{proof}[Proof of the claim]
Let $x$ be as in the hypothesis. 
For each $n\in \N$, since $p^{(n)}$ and $q^{(n)}$ are co-prime, we can choose $(k^{(n)}_{1},k^{(n)}_{3})\in \Z^2$ 
such that
\begin{equation}
\label{eq:1}
\left| k^{(n)}_{1} - \cfrac{p^{(n)}}{q^{(n)}}\ k^{(n)}_{3}  - 
\cfrac{h^{(n)}_{1}\alpha_3 - h^{(n)}_{3}\alpha_1}{r_2h^{(n)}_{3}w^{(n)}_{1}}\ x \right| < \cfrac{1}{q^{(n)}}.
\end{equation}
As $n$ tends to infinity, the sequence $(h^{(n)}_{3})_n$ converges to $h^\infty_3$, $w^{(n)}_{1}$ is constant, $h^{(n)}_{1}\alpha_3 - h^{(n)}_{3}\alpha_1$ converges to a non-zero constant (since $\alpha_1$ and $\alpha_3$ have opposite signs), hence there exists some constant $C>0$ such that
\begin{equation}
\label{eq:2}
\cfrac{r_2h^{(n)}_{3}w^{(n)}_{1}}{h^{(n)}_{1}\alpha_3 - h^{(n)}_{3}\alpha_1} < C.
\end{equation}
From~\eqref{eq:1} and~\eqref{eq:2} we draw 
$$
\left| x_n - x \right|  <  \cfrac{C}{q^{(n)}}
$$
that is the desired inequality. The claim is proved.
\end{proof}
In order to conclude the proof of Theorem~\ref{theo:main}, one needs to show that $q^{(n)} \rightarrow \infty$.
Indeed, we then have that $x_{n} \longrightarrow x$ and since $x$ was arbitrary, by Claim~\ref{claim:2} 
this shows
$$
(Y,\eta) + (x,0) \in \overline{\Orb},\ \textrm{ for any } x \in (-c_{\rm min}, c_{\rm min}).
$$
Then Corollary~\ref{cor:1} applies and $Y$ has an open neighborhood in $\overline{\mathcal O}$,
which proves the theorem. \medskip

We now prove that $q^{(n)} \rightarrow \infty$. Recall that
$$
\cfrac{p^{(n)}}{q^{(n)}} = \cfrac{\mu^{(n)}_{1}}{\mu^{(n)}_{3}} = 
\cfrac{w^{(n)}_{3}}{w^{(n)}_{1}} \cdot \cfrac{h^{(n)}_{1}}{h^{(n)}_{3}} = 
\cfrac{w_{3}}{w_{1}} \cdot  \cfrac{h^\infty_1+\alpha_1y_n}{h^\infty_3+\alpha_3y_n}
$$
and $\gcd(p^{(n)},q^{(n)})=1$. Note that since $\alpha_1$ and $\alpha_3$ have opposite signs, $\cfrac{p^{(n)}}{q^{(n)}}$ cannot be a stationary sequence as $y_n$ tends to $0$. 
As $n$ tends to infinity, $p^{(n)}/q^{(n)}$ converges to $p^{\infty}/q^{\infty}=\cfrac{w_3h^\infty_1}{w_1h^\infty_3}$. But as we have seen $\DS{\cfrac{p^{(n)}}{q^{(n)}}}$ cannot be stationary, therefore there are infinitely many $n$ such that $p^{(n)}/q^{(n)} \not = p^{\infty}/q^{\infty}$ which implies that $q^{(n)} \rightarrow \infty$.
\end{proof}

In the remaining of this paper, we will apply Theorem~\ref{theo:main} (more precisely, the techniques used in the proof) 
to show that, for any $D$ which is not a square, there are at most finitely many closed $\GL^+(2,\R)$-orbits in 
$\PrymD(2,2)^{\rm odd}$. Even though, we only prove the result for this case,  
it seems very likely that one can also obtain similar results  for all strata listed in Table~\ref{tab:strata:list}.
In higher ``complexity'' (genus and number singularities) the difficulty comes from the increasing number of 
degenerated surfaces. Along the way, we give a partial proof that the compactification of $\mathbb{P}\PrymD(2,2)^{\rm odd}$ in $\mathbb{P}\Omega \overline{\Mod}_3$ is an algebraic variety. In the case of genus two, this result was proved by McMullen~\cite{Mc5, Mc6} and Bainbridge~\cite{Bai:07,Bai:10}. \medskip

We end this section with a by-product of the proof of Theorem~\ref{theo:main} that will be used in the sequel.

\begin{Theorem}
\label{thm:byproduct}
Let $(Y,\eta) \in \PrymD(\kappa)$ be a Prym eigenform (where $\Omega E_{D}(\kappa)$
has complex dimension $3$) satisfying the following properties:
\begin{enumerate}
\item The horizontal direction is completely periodic,
\item There exists a sequence $(X_n,\omega_n)= (Y,\eta) + (x_n,y_n)$ converging to $(Y,\eta)$ where $y_n\not =0,\ \forall n$,
\item For every $n$, the combinatorial data of the cylinder decomposition in the horizontal direction of $(X_n,\omega_n)$ are the same.
\item \label{ref:ass} The horizontal directions on $X_n$ are parabolic.
\end{enumerate}
Then there exists $\eps>0$ such that $(Y,\eta) + (x,0) \in \overline{\mathcal O}$ for all  $x\in (-\eps,\eps)$, 
where $\Orb =\bigcup_{n} \GL^+(2,\R)\cdot(X_n,\omega_n)$.
\end{Theorem}

Remark that assumption~\eqref{ref:ass} is not necessary.

\section{Preparation of a surgery toolkit}
\label{sec:surgeries}

In this section we will describe several useful surgeries for Prym eigenforms.
More precisely let us fix a surface $(X_0,\omega_0)$ in the following list of strata $\PrymD(\kappa)$:
\begin{itemize}
\item $\PrymD(0,0,0)$ (space a triple tori, Section~\ref{sec:triple:tori}),
\item $\PrymD(4)$ (Section~\ref{sec:Collapsing:prym4}),
\item $\PrymD(2)^\ast$ (set of $(M,\omega) \in \Omega E_{D}(2)$ with a marked Weierstrass point, 
Section~\ref{sec:collapsing:prym2}).
\end{itemize}
For each case, we will construct a continuous locally injective map $\Psi: \mathring{D}(\eps) \rightarrow \PrymD(2,2)^{\rm odd}$, 
where $\mathring{D}(\eps)= \{z\in \C, \, 0< |z|< \eps \}$, such that it induces an embedding of $\mathring{D}(\eps)/(z\sim -z)$ 
into $\PrymD(2,2)^{\rm odd}$. Up to action $\GL^+(2,\R)$, the set $\Psi(\mathring{D}(\eps))$ will be identified to a neighborhood of 
$(X_0,\omega_0)$ in $\PrymD(2,2)^{\rm odd}$. \medskip

\noindent We now describe these surgeries in details (observe that the second one already appears 
in~\cite{Kontsevich2003}  as ``Breaking up a zero'').

\subsection{Space of triple tori}
\label{sec:triple:tori}
\ \smallskip

We say that $(X,\omega)\in\Prym(2,2)^{\rm odd}$ admits a {\em three tori decomposition} 
if there exists a triple of homologous saddle connections $\{\s_0,\s_1,\s_2\}$ on $X$ joining the two 
distinct zeros of $\omega$. It turns out that $(X,\omega)$ can be viewed as a connected sum of three 
tori $(X_j,\omega_j), \, j=0,1,2,$ which are glued together along  the slits corresponding to $\s_j$
(this can be seen by letting the length of saddle connections $\{\s_0,\s_1,\s_2\}$ going to zero in the 
kernel foliation leaf: the limit surface is then a union of three tori which are joint at unique common point $P$).
We will always assume that $X_0$ is preserved and $X_1,X_2$ are exchanged by the Prym involution $\inv$.

Recall that $\H(0)$ is the space of triples $(Y,\eta,P)$ where $Y$ is an elliptic curve, $\eta$ an Abelian differential
on $Y$, and $P$ is a marked point of $Y$. We denote by  $\Prym(0,0,0)$ the space of triples $\{(X_j,\omega_j,P_j), \ j=0,1,2\}$ where $(X_j,\omega_j,P_j) \in \H(0)$ such that $(X_1,\omega_1,P_1)$ and $(X_2,\omega_2,P_2)$ are isometric. The geometric object corresponding to such a triple is the union of the three tori, where we identify 
$P_0,P_1,P_2$ to a unique common point. Note that by construction, there exists an involution $\inv$ on the ``surface'' $X:=\{(X_j,\omega_j,P_j), \ j=0,1,2\}$ which preserves $X_0$ and exchanges $X_1$ and $X_2$, we will call $\inv$ the Prym involution.

We define $\PrymD(0,0,0) \subset \Prym(0,0,0)$ to be the space of all triples 
$\{(X_j,\omega_j,P_j), \, j=0,1,2\}$, obtained by limit in the kernel foliation leaf of surfaces in $\PrymD(2,2)^{\rm odd}$
with a three tori decomposition.  According to above discussion, the aim of this section is to show:
\begin{Proposition}
\label{prop:3tori:neighbor}
For any triple tori $\{(X_j,\omega_j,P_j), \, j=0,1,2\}$ in $\PrymD(0,0,0)$, there exist $\eps>0$ and a continuous 
locally injective map $\Psi: \mathring{D}(\eps) \rightarrow \PrymD(2,2)^{\rm odd}$ satisfying:
\begin{enumerate}
\item $\forall z \in \mathring{D}(\eps)$, the surface $(X,\omega)=\Psi(z)$ has a triple of homologous 
saddle connections $\{\s_0,\s_1,\s_2\}$ with distinct endpoints and $\omega(\s_j)=z$,
\item The map $\Psi$ is two to one and it induces an embedding of $\mathring{D}(\eps)/(z \sim -z)$ into $\PrymD(2,2)^{\rm odd}$, 
\item Up to action $\GL^+(2,\R)$, the set $\Psi(\mathring{D}(\eps))$ can be viewed as the neighborhood of 
$\{(X_j,\omega_j), \, j=0,1,2\}$ in $\PrymD(2,2)^{\rm odd}$.
\end{enumerate}
\end{Proposition}

We postpone the proof of Proposition~\ref{prop:3tori:neighbor} and first provide a description of the space 
of triples $\PrymD(0,0,0)$ (compare with~\cite[Theorem~8.3]{Mc07}).

\begin{Proposition}
\label{prop:3tori:decomp}
Let $\{(X_j,\omega_j,P_j), \, j=0,1,2\}$ be a triple tori  in $\PrymD(0,0,0)$
(where $X_1,X_2$ are exchanged by the Prym involution $\inv$). Then 
there exist $(e,d)\in\Z^2$, with $d>0$, and a covering  $p: X_1\rightarrow X_0$ of degree $d$ such that 
\begin{itemize}
\item $D=e^2+8d$,
\item $\gcd(e,p_{11},p_{12},p_{21},p_{22})=1$, where $(p_{ij})$  is the matrix of $p$ in some symplectic bases of $H_1(X_0,\Z)$ and $H_1(X_1,\Z)$.
\item $p^*\omega_0=\cfrac{\lbd}{2}\omega_1$, where $\lbd$ satisfies $\lbd^2=e\lbd +2d$.
\end{itemize}
\end{Proposition}  

\begin{proof}
Recall that the Prym involution preserves $X_0$ and exchanges $X_1,X_2$. Let $(a_j,b_j)$ be a symplectic basis of $H_1(X_j,\Z)$, where 
$a_2=-\inv(a_1), b_2=-\inv(b_1)$, and set $\hat{a}=a_1+a_2, \, \hat{b}=b_1+b_2$. Then 
$(a_0,b_0,\hat{a},\hat{b})$ is a symplectic basis of $H_1(X,\Z)^-$ ($X$ is the surface obtained by identifying $P_0\sim P_1 \sim P_2$). 
There exists a unique generator $T$ of $\Ord_D$ such that the matrix of $T$ in the basis $(a_0,b_0,\hat{a},\hat{b})$ is of the form 
$T=\left(\begin{smallmatrix} e\Id_2 & 2B \\ B^* & 0  \end{smallmatrix} \right)$, where $B\in {\bf M}_2(\Z)$, 
$B^*=J\cdot B \cdot J^{-1}$, and $T^*\omega=\lbd\omega$, with $\lbd>0$. 

Observe that $B$ can be regarded as a map from $H_1(X_1,\Z)$ to $H_1(X_0,\Z)$. Set 
$L_0=\Z\omega_0(a_0) + \Z\omega_0(b_0), \, L_1=\Z\omega_1(a_1)+\Z\omega_1(b_1)$. We can identify 
$(X_0,\omega_0)$ and $(X_1,\omega_1)$ with $(\C/L_0, dz)$ and $(\C/L_1, dz)$ respectively. 
The condition $T^*\omega=\lbd\omega$ reads
$$
\omega_0(2B(a_1))  = \lbd\cdot \omega_1(a_1) \qquad \mathrm{and} \qquad 
\omega_0(2B(b_1)) = \lbd\cdot \omega_1(b_1).
$$
Hence $\frac{\lbd}{2} L_1$ is a sublattice of $L_0$. It follows that there exists a covering map $p: \C/L_1 \rightarrow \C/L_0$ such that $p^*dz=\lbd/2 dz$. The degree of $p$ is given by $d=\det(B)>0$. Note that $T$ satisfies
$$
T^2=eT+2\det(B).
$$
Since $T$ is a generator of $\Ord_D$, we have $D=e^2+ 8\det(B)$. As $\lbd$ is an eigenvalue of $T$, $\lbd$ satisfies the same equation. 
\end{proof}

\begin{proof}[Proof of Proposition~\ref{prop:3tori:neighbor}]
Let $\eps>0$ be small enough so that  the set $D(P_j,\eps)=\{x \in X_j, \, \dist(x,P_j)<\eps\}$ is an embedded disc in 
$X_j, \, j=0,1,2$. The map $\Psi$ is defined as follows: for any $z\in \mathring{D}(\eps)$, let $\s_j$ be the geodesic segment 
in $X_j$ whose midpoint is $P_j$ such that $\omega(\s_j)=z$ (since $|z|< \eps$, $\s_j$ is an embedded segment). By slitting 
$X_j$ along $\s_j$, and  gluing $X_0,X_1,X_2$ along the slits in a cyclic order, we get a surface $(X,\omega)$ in $\H(2,2)$. 
It is easy to check that $(X,\omega)\in \PrymD(2,2)^{\rm odd}$.  We define $(X,\omega)=\Psi(z)$. Since we cannot distinguish 
the two zeros of $\omega$, one has $\Psi(z)=\Psi(-z)$. This ends the proof of Proposition~\ref{prop:3tori:neighbor}.
\end{proof}


\subsection{Collapsing surfaces to $\PrymD(4)$}
\label{sec:Collapsing:prym4}
\smallskip

This surgery already appears in~\cite{Kontsevich2003} (``Breaking up a zero''). As in the previous 
section, our aim is to show:
\begin{Proposition}
\label{prop:Prym4:neighbor}
For any $(X_0,\omega_0)\in \Omega E_D(4)$, there exist $\eps>0$ and a continuous locally injective map 
$\Psi :\mathring{D}(\eps) \rightarrow \Omega E_D(2,2)^{\rm odd}$ satisfying:
\begin{enumerate}
\item $\forall z\in \mathring{D}(\eps)$, the surface  $(X,\omega)=\Psi(z)$ has the same absolute periods 
as $(X_0,\omega_0)$,
\item There exists a saddle connection $\s$ in $X$ joining the zeros of $\omega$ such that $\omega(\s)=z^5$,
\item The map $\Psi$ induces an embedding $\mathring{D}(\eps)/(z\sim -z) \rightarrow \Omega E_D(2,2)^{\rm odd}$,
\item Up to the action of $\GL^+(2,\R)$, a neighborhood of $(X_0,\omega_0)\in \Omega E_D(4)$ in $\Omega E_D(2,2)^{\rm odd}$ 
is identified with $\Psi(\mathring{D}(\eps))$.
\end{enumerate}
\end{Proposition}
The constructive proof we will give is on the level of Abelian differentials {\em i.e.} in $\Prym(2,2)$ and $\Prym(4)$.
One can interpret this construction on the level of quadratic differentials {\em i.e.} $\QQQ(-1^4,4)$ and 
$\QQQ(-1^3,3)$, respectively. This last approach is related to the surgery ``breaking up a singularity'' 
in~\cite{Kontsevich2003} (breaking up the zero of degree $3$ of the quadratic differential into a pole and a zero of degree $4$).
\begin{proof}[Proof of Proposition~\ref{prop:Prym4:neighbor}]
Let $(X_0,\omega_0) \in \Omega E_D(4)$ and let $P_0$ be the unique zero of $\omega_0$.
We consider $0<\eps<1$ small enough so that the euclidian disc $D(P_0,\eps)=\{x\in X_0, \, \dist(x,P_0)\leq \eps\}$ is embedded 
into $X_0$. Since the conical angle of the zero is $10\pi$ the neighborhood of $P_0$ can be identified 
with a polydisc, that is the union of the $10$ half-discs.

Let $v\in \R^2\setminus\{0\}$ be a vector such that $|v|<\eps/2$. It determines a 
collection of (oriented) geodesic rays emanating from $P_0$ in the direction of $\pm v$. These rays intersect the boundary 
$\partial D(P_0,\eps)$ at 10 points denoted by $a_1,\dots,a_{10}$ following the orientation of $\partial D(P_0,\eps)$, where $a_{2k-1}$  and $a_{2k}$ are respectively the intersections of $\partial D(P_0,\eps)$ with  rays indirection $v$  and  rays in direction $-v$. We denote by $u_i$ the segment from $P_0$ to $a_i$. The union of $u_i$ and $u_{i+1}$ is the diameter of an euclidian half-disc which will be denoted by $D_i$ (here we use the convention $i \sim i-10$ if $i>10$).

To get a surface $(X,\omega)$ in $\Omega E_D(2,2)^{\rm odd}$ with a saddle connection $\s$ such that $\omega(\s)=v$, we replace $D(P_0,\eps)\subset X_0$ 
by a domain $\tilde{D}(\eps)$ constructed from $D_1,\dots,D_{10}$ by gluing them in such a way that there are two singular points, with angle $6\pi$, which are joined by a segment contained in the diameter of two half-discs $D_{k}$ and $D_{k+5}$ (see Figure~\ref{fig:split:zero:H4} for $k=3$).

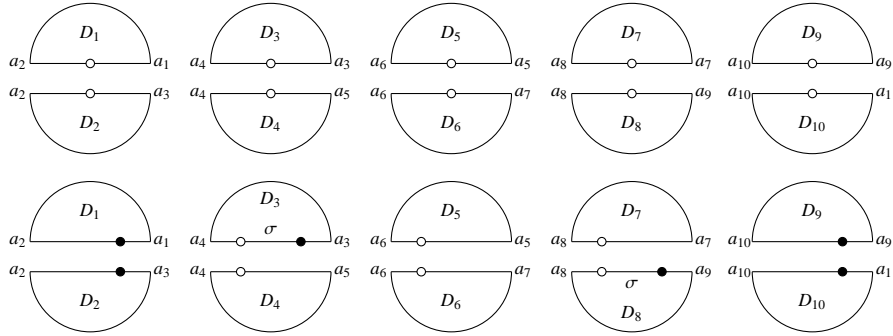
\begin{figure}[htb]
\begin{minipage}[t]{0.8\linewidth}
\begin{tikzpicture}[scale=0.8]
 \foreach \x in {(-5,0.5), (-2,0.5), (1,0.5), (4,0.5), (7,0.5)} {\draw \x arc (0:180:1); \draw \x -- +(-2,0);}
 
 \foreach \x in {(-7,0), (-4,0), (-1,0), (2,0), (5,0)} {\draw \x arc (180:360:1); \draw \x -- +(2,0);}
 
 \foreach \x in {(-6,0.5), (-6,0), (-3,0.5), (-3,0), (0,0.5), (0,0), (3,0.5), (3,0), (6,0.5), (6,0)} \filldraw[fill=white] \x circle (2pt);
 
 \draw (-4.8,0.5) node {\tiny $a_1$} (7.2,0) node {\tiny $a_1$};
 \draw (-7.2,0.5) node {\tiny $a_2$} (-7.2,0) node {\tiny $a_2$};
 \draw (-4.8,0) node {\tiny $a_3$} (-1.8,0.5) node {\tiny $a_3$};
 \draw (-4.2,0.5) node {\tiny $a_4$} (-4.2,0) node {\tiny $a_4$};
 \draw (-1.8,0) node {\tiny $a_5$} (1.2,0.5) node {\tiny $a_5$};
 \draw (-1.2,0.5) node {\tiny $a_6$} (-1.2,0) node {\tiny $a_6$};
 \draw (1.2,0) node {\tiny $a_7$} (4.2,0.5) node {\tiny $a_7$};
 \draw (1.8,0.5) node {\tiny $a_8$} (1.8,0) node {\tiny $a_8$};
 \draw (4.2,0) node {\tiny $a_9$} (7.2,0.5) node {\tiny $a_9$};
 \draw (4.8,0.5) node {\tiny $a_{10}$} (4.8,0) node {\tiny $a_{10}$};
 
 \draw (-6,1) node {\tiny $D_1$} (-6,-0.5) node {\tiny $D_2$};
 \draw (-3,1) node {\tiny $D_3$} (-3,-0.5) node {\tiny $D_4$};
 \draw (0,1) node {\tiny $D_5$} (0,-0.5) node {\tiny $D_6$};
 \draw (3,1) node {\tiny $D_7$} (3,-0.5) node {\tiny $D_8$};
 \draw (6,1) node {\tiny $D_9$} (6,-0.5) node {\tiny $D_{10}$};
\end{tikzpicture}
\end{minipage}

\bigskip

\begin{minipage}[t]{0.8\linewidth}
\begin{tikzpicture}[scale=0.8]
 \foreach \x in {(-5,0.5), (-2,0.5), (1,0.5), (4,0.5), (7,0.5)} {\draw \x arc (0:180:1); \draw \x -- +(-2,0);}
 
 \foreach \x in {(-7,0), (-4,0), (-1,0), (2,0), (5,0)} {\draw \x arc (180:360:1); \draw \x -- +(2,0);}
 
 \foreach \x in {(-3.5,0.5), (-3.5,0), (-0.5,0.5), (-0.5,0), (2.5,0.5), (2.5,0)} \filldraw[fill=white] \x circle (2pt);
 \foreach \x in {(-5.5,0.5), (-5.5,0), (-2.5,0.5), (3.5,0), (6.5,0.5), (6.5,0)} \filldraw[fill=black] \x circle (2pt);
 
 \draw (-4.8,0.5) node {\tiny $a_1$} (7.2,0) node {\tiny $a_1$};
 \draw (-7.2,0.5) node {\tiny $a_2$} (-7.2,0) node {\tiny $a_2$};
 \draw (-4.8,0) node {\tiny $a_3$} (-1.8,0.5) node {\tiny $a_3$};
 \draw (-4.2,0.5) node {\tiny $a_4$} (-4.2,0) node {\tiny $a_4$};
 \draw (-1.8,0) node {\tiny $a_5$} (1.2,0.5) node {\tiny $a_5$};
 \draw (-1.2,0.5) node {\tiny $a_6$} (-1.2,0) node {\tiny $a_6$};
 \draw (1.2,0) node {\tiny $a_7$} (4.2,0.5) node {\tiny $a_7$};
 \draw (1.8,0.5) node {\tiny $a_8$} (1.8,0) node {\tiny $a_8$};
 \draw (4.2,0) node {\tiny $a_9$} (7.2,0.5) node {\tiny $a_9$};
 \draw (4.8,0.5) node {\tiny $a_{10}$} (4.8,0) node {\tiny $a_{10}$};

 \draw (-6,1) node {\tiny $D_1$} (-6,-0.5) node {\tiny $D_2$};
 \draw (-3,1.2) node {\tiny $D_3$} (-3,-0.5) node {\tiny $D_4$};
 \draw (0,1) node {\tiny $D_5$} (0,-0.5) node {\tiny $D_6$};
 \draw (3,1) node {\tiny $D_7$} (3,-0.7) node {\tiny $D_8$};
 \draw (6,1) node {\tiny $D_9$} (6,-0.5) node {\tiny $D_{10}$};
 
 \draw (-3,0.7) node {\tiny $\s$} (3,-0.2) node {\tiny $\s$};

\end{tikzpicture}
\end{minipage}

 \caption{Splitting a zero of order $4$ into two zeros of order $2$.}
 \label{fig:split:zero:H4} 
\end{figure}

Note that we have a  Prym involution $\inv_0$ on $X_0$ which fixes $P_0$ and sends $D_k$ to $D_{k+5}$. By construction, there exists an involution on $\tilde{D}(\eps)$ which sends $D_{k}$ to $D_{k+5}$. In particular, this involution agrees with the restriction of $\inv_0$ on $\partial \tilde{D}(\eps)=\partial D(P_0,\eps)$. Therefore, we also have an involution $\inv$ on $X$ that exchanges the two zeros of $\omega$. It is easy to check that $(X,\omega)\in \Prym(2,2)$.

Since we have 5 choices for the pair of half-discs which contain $\s$ in their boundary, we see that there are five surfaces 
$(X,\omega)$ in $\Prym(2,2)$ close to $(X_0,\omega_0)$ satisfying the following  conditions:
\begin{itemize}
 \item[$\bullet$] The absolute periods of $\omega$ and $\omega_0$  coincide,
 \item[$\bullet$] There exists a saddle connection $\s$ in $X$, invariant by the Prym involution, joining the two zeros of $\omega$ such that $\omega(\s)=v$.  
\end{itemize}
Since the absolute periods of $\omega$ and $\omega_0$  coincide,  the new surface 
actually belongs to the real multiplication locus {\em i.e.} to $\Omega E_D(2,2)^{\rm odd}$.
This defines the desired map $\Psi :\mathring{D}(\eps) \rightarrow \Omega E_D(2,2)^{\rm odd}$ where $\Psi(z) = (X,\omega)$.
Observe that since we cannot distinguish the zeros of $\omega$, the surfaces corresponding to 
$\pm z$ are the same (with different choices for the orientation of $\s$).
\end{proof}

\begin{Remark}
The ``breaking up a zero'' surgery is clearly invertible: we can collapse the two zeros of $(X,\omega)$ along $\s$ to get the surface $(X_0,\omega_0)\in \PrymD(4)$. 
More generally, let $P,Q$ denote the zeros of $\omega$, where $(X,\omega)\in\PrymD(2,2)^\mathrm{odd}$, and let $\s$ be a saddle 
connection, that we assume to be horizontal, joining $P$ to $Q$ that is invariant by the involution $\inv$ 
(such a saddle connection always exists, for instance the union of a path of minimal length joining a fixed point of 
$\inv$ to $P$ or $Q$, and its image by $\inv$). If for any other horizontal saddle connection $\s'$ we have $|\s'|>2|\s|$ then one can collapse 
the zeros of $\omega$ along $\s$ by using the kernel foliation (see Section~\ref{sec:degenerating:surfaces}). The 
resulting surface $(X_0,\omega_0)$ belongs to $\PrymD(4)$. However if $\s$ has twins, that is another saddle connection $\s'$ such that $\omega(\s')=\omega(\s)$,  then the 
limit surface is no longer in $\PrymD(4)$ as we will see in the sequel.
\end{Remark}

\subsection{Collapsing surfaces to $\PrymD(2)^\ast$}
\label{sec:collapsing:prym2}
\smallskip
In this  section, we investigate degenerations by shrinking a pair of saddle connections that are exchanged by the Prym involution.
Let $\Omega E_{D'}(2)^*$ be the space of triples $(X,\omega,W)$, where $(X,\omega) \in \Omega E_{D'}(2)$, and 
$W$ is a Weierstrass point of $X$ which is not the zero of $\omega$. We will prove
\begin{Proposition}
\label{prop:H2:neighbor}
For any $(X_0,\omega_0,W_0) \in \Omega E_{D'}(2)^*$ there exist $0<\eps<1$, $D \in \{D',4D'\}$, 
and a continuous locally injective map $\Psi :\mathring{D}(\eps) \rightarrow \PrymD(2,2)^{\rm odd}$ with the 
following properties:
\begin{enumerate}
\item $\forall z \in \mathring{D}(\eps)$ the surface  $(X,\omega)=\Psi(z)$ has the same absolute periods as 
$(X_0,\omega_0,W_0)$,
\item there exists a pair of saddle connections $(\s_1,\s_2)$ on $X$ that are exchanged by the 
Prym involution and satisfy $\omega(\s_1)=\omega(\s_2)=z^3$.
\item The map $\Psi$ induces an embedding $\Psi: \mathring{D}(\eps)/(z\sim -z) \rightarrow \PrymD(2,2)^{\rm odd}$,
\item Up to action of $\GL^+(2,\R)$, $\Psi(\mathring{D}(\eps))$ is a neighborhood of $(X_0,\omega_0,W_0)$ in $\PrymD(2,2)^{\rm odd}$.
\end{enumerate}
\end{Proposition}
As for above surgeries, we will describe how one can degenerate some $(X,\omega) \in \PrymD(2,2)^{\rm odd}$ 
to the boundary of the stratum {\em i.e.} to $(X_0,\omega_0,W_0) \in \Omega E_{D'}(2)^*$, by using the kernel foliation.
The inverse procedure will give the map $\Psi$ of Proposition~\ref{prop:H2:neighbor}. Hence let us show:

\begin{Theorem}
\label{thm:collapse:Prym2}
Let $(\s_1,\s_2)$ be a pair of non-homologous saddle connections in $X$ that are exchanged by the Prym involution $\inv$. 
Suppose that for any other saddle connection $\s'$ joining $P$ to $Q$ in the same direction as $\s_1$, we have 
$|\s'| >|\s_1|$. Then as the length of $\s_1$ tends to zero (in the leaf  of the kernel foliation), $(X,\omega)$ tends to a point 
in the boundary of $\PrymD(2,2)^{\rm odd}$ which is represented by  a triple 
$(X_0,\omega_0,W_0)\in \Omega E_{D'}(2)^\ast$ for some $D' \in \{D,D/4\}$.
\end{Theorem}

Observe that we consider $\thetaup$ and $-\thetaup$ ($\thetaup \in \mathbb{S}^1$) as two distinct directions.
As usual, we choose the orientation for any saddle connection joining $P$ and $Q$ to be {\em from $P$ to $Q$}.
For the remaining of this section, we fix a pair of saddle connections $(\s_1,\s_2)$ satisfying assumption of 
Theorem~\ref{thm:collapse:Prym2}. We will need of the following:

\begin{Lemma}
\label{lm:collapse:g2:eigen}
Let us construct the translation surface $(X',\omega')$ by first cutting $(X,\omega)$ along $c=\s_1*(-\s_2)$ 
and then gluing the resulting pair of geodesic segments in each boundary component. Then
$$
(X',\omega') \in \Omega E_{D'}(1,1) \qquad \mathrm{for\ some} \qquad D' \in \{D,D/4\}.
$$
(the involution $\inv$ of $X$ descends to the hyperelliptic involution of $X'$).
\end{Lemma}

\begin{proof}[Proof of Lemma~\ref{lm:collapse:g2:eigen}]
We first show that $(X',\omega') \in \H(1,1)$. For that, we remark that the pair of angles specified by these two rays at 
the zeros $P$ and $Q$ are $(2\pi,4\pi)$. Since $\inv$ sends $\s_1$ to $-\s_2$ and preserves the orientation of 
$X$, necessarily the angle $2\pi$ at $P$ and the angle $2\pi$ at $Q$ belong to the same side of $c$ which prove 
the first fact.

The surface $(X',\omega')$ has two marked segments $c_1,c_2$, where $c_1$ is a saddle connection, and $c_2$ is 
simply a geodesic segment which has the same length and the same direction as $c_1$. We denote the endpoints of 
$c_1$ (respectively, $c_2$) by $P_1,Q_1$ (respectively, $P_2,Q_2$). Hence $P_1,P_2$ correspond to $P$ and 
$Q_1,Q_2$ correspond to $Q$. Note that $P_1,Q_1$ are the zeros of $\omega'$. We choose the orientation of 
$c_1$ (respectively, $c_2$) to be from $P_1$ to $Q_1$ (respectively, from $P_2$ to $Q_2$).

With these notations, $\inv$ induces an involution $\inv'$ on $X'$ such that $\inv'(c_1)=-c_1$ and $\inv'(c_2)=-c_2$.
It turns out that $\inv'$ has six fixed points on $X'$: these are the four fixed points of $\inv$ (none of them are contained 
in $c$)  and two additional fixed points in $c_1$ and $c_2$. By uniqueness $\inv'$ is therefore the hyperelliptic 
involution.

To conclude the proof, one needs to show that $(X',\omega')$ is an eigenform. For that we first need to choose 
a symplectic basis of $H_1(X',\Z)$. We proceed as follows (see Figure~\ref{fig:collapse:H2:basis}). 
Let $\alpha_{1,1},\alpha_{1,2},\alpha_2, \beta_2$ be the simple closed curves, and $\beta_{1,1}$ and $\beta_{1,2}$ 
be simple arcs in $X'$ as shown in Figure~\ref{fig:collapse:H2:basis}, where $\alpha_{1,2}=-\inv'(\alpha_{1,1})$ and 
$\beta_{1,2}=-\inv'(\beta_{1,1})$. Let $\beta'_1$ denote the simple closed curve which is the concatenation 
$c_1\cup\beta_{1,1}\cup c_2\cup\beta_{1,2}$. Set $\alpha'_1=\alpha_{1,1}$ 
(the orientations are chosen so that  $(\alpha'_1,\beta'_1, \alpha_2,\beta_2)$ is a symplectic basis of $H_1(X',\Z)$). 
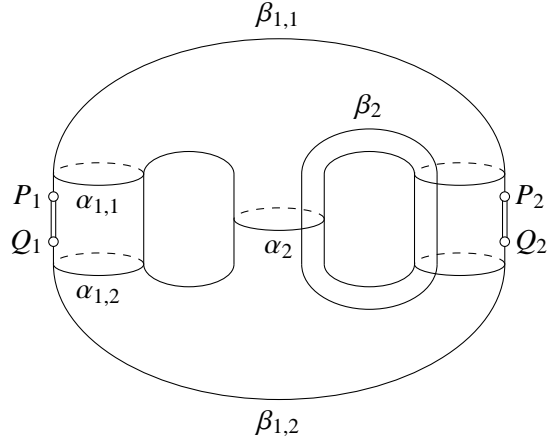
\begin{figure}[htbp]
\begin{center}
\begin{tikzpicture}[scale=0.6]
\draw (0,1) arc  (180:0: 5 and 3);
\draw (0,-1) arc (180:360: 5 and 3);
\draw (0.05,-0.5) -- (0.05,0.5) (-0.05,-0.5) -- (-0.05,0.5) (9.95,-0.5) -- (9.95,0.5) (10.05,-0.5)--(10.05,0.5);
\draw (0,-1) -- (0,-0.5) (0,0.5) -- (0,1) (10,-1) -- (10,-0.5) (10,0.5) -- (10,1);
\draw (2,-1) -- (2,1) (4,-1)-- (4,1)  (6,-1) -- (6,1) (8,-1) -- (8,1);
\draw (4,1) arc (0:180:1 and 0.5) (2,-1) arc (180:360: 1 and 0.5) (8,1) arc (0:180: 1 and 0.5) (6,-1) arc (180:360: 1 and 0.5);
\draw  (0,1) arc (180:360: 1 and 0.25) (0,-1) arc (180:360:1 and 0.25) (4,0) arc (180:360: 1 and 0.25) (8,1) arc (180:360: 1 and 0.25) (8,-1) arc (180:360: 1 and 0.25);
\foreach \x in {(2,1), (2,-1), (6,0), (10,1), (10,-1)} \draw[dashed] \x arc(0:180:1 and 0.25);
\draw (5.5,-1) -- (5.5,1) (8.5,-1) -- (8.5,1); 
\draw (8.5,1) arc (0:180: 1.5 and 1) (5.5,-1) arc (180:360: 1.5 and 1);
\draw (0,0.5) node[left] {$P_1$} (0,-0.5) node[left] {$Q_1$} (10,0.5) node[right] {$P_2$} (10,-0.5) node[right] {$Q_2$};
\draw (5,4) node[above] {$\beta_{1,1}$} (5,-4) node[below] {$\beta_{1,2}$} (1,0.75) node[below] {$\alpha_{1,1}$} (1,-1.25) node[below] {$\alpha_{1,2}$} (5,-0.25) node[below] {$\alpha_2$} (7,2) node[above] {$\beta_2$};
\foreach \x in {(0,0.5), (0,-0.5), (10,0.5), (10,-0.5)} \filldraw[fill=white] \x circle (3pt);
\end{tikzpicture}
\end{center}
\caption{Surface in $\H(1,1)$ obtained by cutting and gluing along a pair of saddle connections exchanged by the Prym involution.
The hyperelliptic involution $\inv'$ exchanges the upper and the lower halves of $X'$.}
\label{fig:collapse:H2:basis}
\end{figure}

Observe that $\beta_{1,1}$, $\beta_{1,2}$ correspond to two simple closed curves in $X$, and that 
$\alpha_{1,1}$, $\alpha_{1,2}$ are not homologous in $H_1(X,\Z)$. In other words  
$(\alpha_1,\beta_1,\alpha_2,\beta_2)$ is a symplectic basis of $H_1(X,\Z)^-$, where 
$\alpha_1=\alpha_{1,1}+\alpha_{1,2}, \beta_1=\beta_{1,1}+\beta_{1,2}$, and the intersection form is given by 
the matrix
$\left(\begin{smallmatrix}
2J & 0\\
0 & J\\
\end{smallmatrix}\right)$.

Since $(X,\omega) \in \PrymD(2,2)^{\rm odd}$, by definition there exists a unique generator $T$ of $\Ord_D$
that can be expressed (in the basis $(\alpha_1,\beta_1,\alpha_2,\beta_2)$ of $H_1(X,\Z)^-$) by the matrix
$$
T=\left( \begin{smallmatrix}
            e & 0 & a & b\\
        0 & e & c & d \\
        2d & -2b & 0 & 0\\
            -2c & 2a & 0 & 0\\
\end{smallmatrix}\right),
$$
where $D=e^2+8(ad-bc)$, $\gcd(a,b,c,d,e)=1$ and $T^\ast \omega = \lbd \cdot \omega$,  with $\lbd>0$. In the 
symplectic basis $(\alpha'_1,\beta'_1,\alpha_2,\beta_2)$  of $H_1(X',\Z)$ we define the endomorphism:
$$
T'=\left(\begin{smallmatrix}
e & 0 & 2a & 2b\\
0 & e & c  & d\\
d & -2b & 0 & 0\\
-c & 2a & 0 & 0\\
\end{smallmatrix}\right).
$$
Obviously $T'$ is self-adjoint with respect to the symplectic form $\left(\begin{smallmatrix} J & 0 \\ 0 & J\\ \end{smallmatrix}\right)$
and ${T'}^2=eT'+2(ad-bc)\mathrm{Id}$. 
Let us show that $\omega'$ is an eigenform for $T'$, namely  $(T')^\ast \omega' = \lbd' \cdot \omega'$,  with $\lbd'>0$.
This last equation reads (in the symplectic basis $(\alpha'_1,\beta'_1,\alpha_2,\beta_2)$):
\begin{equation}
\label{eq:symplectic}
(x,y,z,t)\cdot T'=\lbd'(x,y,z,t),
\end{equation}
where $(x,y,z,t) = (\omega'(\alpha'_1),\omega'(\beta'_1),\omega'(\alpha_2),\omega'(\beta_2))\in \C^4$. But
$$
\begin{array}{l}
\omega'(\alpha'_1)=\omega(\alpha_{1,1})=\frac{1}{2}\omega(\alpha_1), \\
\omega'(\beta'_1)=-\omega'(c_1)+\omega'(\beta_{1,1}) +\omega'(c_2) +\omega'(\beta_{1,2})= 
\omega(\beta_{1,1})+\omega(\beta_{1,2}) =\omega(\beta_1), \\
\omega'(\alpha_2)=\omega(\alpha_2), \\
\omega'(\beta_2)=\omega(\beta_2).
\end{array}
$$
Consequently in the basis $(\alpha_1,\beta_1,\alpha_2,\beta_2)$ the $1$-form $\omega$ is represented by the row vector 
$(2x,y,z,t)$. Now by assumption $T^\ast \omega = \lbd \cdot \omega$ or equivalently $(2x,y,z,t)\cdot T=\lbd(2x,y,z,t)$.
We easily check this implies the desired Equation~\eqref{eq:symplectic} with $\lbd'=\lbd$. 

Hence $T'$ generates a subring isomorphic to $\Ord_D$ in ${\rm End}({\bf Jac}(X'))$ for which $\omega'$ is an eigenform.
In other words $(X',\omega') \in \Omega E_{D'}(1,1)$ for some $D'$ dividing $D$. The proper subring isomorphic 
to $\Ord_{D'}$ is generated by the matrix $T'/k \in \rm{End}({\bf Jac}(X'))$ where $k=\gcd(2a,2b,c,d,e)$. 
By assumption $\gcd(a,b,c,d,e)=1$, therefore $k\in\{1,2\}$. Since $D=k^2D'$, the lemma follows. 
\end{proof}

We can now proceed to the proof of our results.
\begin{proof}[Proof of Theorem~\ref{thm:collapse:Prym2}]
We keep the notations of Lemma~\ref{lm:collapse:g2:eigen}. By construction, there is no obstruction to  
collapse the two zeros of $\omega'$ along $c_1$ along the kernel foliation through $(X',\omega')$:
the resulting surface belongs to $\Omega E_{D'}(2)$. Note that when $c_1$ is shrunken to a point, so is $c_2$. 
Since $c_2$ is invariant by the hyperelliptic involution of $X'$, in the limit $c_2$ becomes a marked Weierstrass point.
\end{proof}

\begin{proof}[Proof of Proposition~\ref{prop:H2:neighbor}]
The surgery ``collapse a pair of saddle connections exchanged by $\inv$'', as described above, 
is invertible: this is the map $\Psi$ of the proposition. Let us give a more precise definition of this map.

We fix a point $(X_0,\omega_0,W_0) \in \Omega E_{D'}(2)^*$, and choose $\eps >0$ small enough so that 
the sets $D(P_0,\eps)=\{x\in X_0, \, \dist(x,P_0) < \eps\}$ and $D(W_0,\eps)=\{x\in X_0, \, \dist(x, W_0) < \eps\}$, 
are two embedded (disjoint) discs ($P_0$ is the zero of $\omega_0$).  

Given any vector $v\in \R$, with  $|v|<\eps$,  we construct a Prym form in $\Prym(2,2)$ as follows. 
We break up the zero $P_0$ into two zeros in order to get a surface $(X',\omega') \in \H(1,1)$ (having the same 
absolute periods as $\omega$) with a marked  saddle 
connection, say $\s_1$, that is invariant by the hyperelliptic involution and such that $\omega'(\s_1)=v$.
Note that by assumption  $\s_1$ is disjoint from $D(W_0,\eps)$. Let $\s_2$ be a geodesic segment in $D(W_0,\eps)$ 
such that $\omega'(\s_2)=v$, and $W_0$ is the midpoint of $\s_2$.  
Cutting $X'$ along $\s_1$ and $\s_2$, then regluing the resulting boundary components, we get a new surface 
$(X,\omega) \in \H(2,2)$ together with an involution $\inv: X \rightarrow X$ (induced by the hyperelliptic involution of
$X'$). Since by construction $\inv^*\omega =-\omega$ one has $(X,\omega) \in \Prym(2,2)$.

The arguments of the proof of Lemma~\ref{lm:collapse:g2:eigen} actually show that  $(X,\omega)\in \Omega E_D(2,2)$
for some $D \in\{D',4D'\}$. We then define $\Psi(z) = (X,\omega)$, where $z$ is a complex number such that  $v=z^3$ (this is related to the fact that we have three choices for the segment $\s_1$). It is now straightforward to check the properties of the map $\Psi$. The proposition is proved.
\end{proof}

\section{Degenerating surfaces of $\PrymD(2,2)^{\rm odd}$}
\label{sec:degenerating:surfaces}


In this section, we show that the surgeries described in Section~\ref{sec:surgeries}
are sufficient to describe the all the degenerations (along the kernel foliation) 
of Prym eigenforms in $\PrymD(2,2)^{\rm odd}$ having an unstable cylinder when $D$ is not a square
(compare with~\cite{Lanneau:Manh:composantes}).

\begin{Theorem}
\label{thm:unstable:dec:collapse}
Assume $D$ is not a square, and let $(X,\omega)\in \PrymD(2,2)^{\rm odd}$ with 
an unstable cylinder decomposition in the horizontal direction. Then there exists a finite interval $[s_\mathrm{min},s_\mathrm{max}]$ 
such that for any $x \in ]s_\mathrm{min},s_\mathrm{max}[$, the surface $(X,\omega)+ (x,0)$ is well-defined 
and belongs to $\PrymD(2,2)^{\rm odd}$. Moreover when $x$ tends to $\partial [s_\mathrm{min},s_\mathrm{max}]$, 
$(X,\omega)+(x,\omega)$ converges to a surface $(Y,\eta)$ which belongs to
$$
\PrymD(0,0,0), \ \PrymD(4)\ \mathrm{or} \ \Omega E_{D'}(2)^* \ \text{ with } \ D' \in \{D, D/4\}.
$$
\end{Theorem}

We will use the following elementary lemma.
\begin{Lemma}
\label{lm:twins:sc:Dsq}
Let $(X,\omega) \in \PrymD(2,2)^{\rm odd}$. Assume that one of the following occurs:
\begin{enumerate}
\item There exists a non trivial homology class $c\in H_1(X,\Z)^-$ such that $\omega(c)=0$.
\item There exists two twins saddle connections in $X$ joining the two zeros of $\omega$, which are both invariant by the Prym involution.
\item There exists a triple of twins saddle connections $(\s_0,\s_1,\s_2)$ where $\s_0$ is invariant and $(\s_1,\s_2)$ are 
exchanged by the Prym involution, such that $c_0=\s_1*(-\s_2)$ is non-separating.
\end{enumerate}
Then $D$ is a square.
\end{Lemma}

\begin{proof}[Proof of Lemma~\ref{lm:twins:sc:Dsq}]
For the first condition, we set $K=\Q(\sqrt{D})$. If $D$ is not a square then $K$ is a real quadratic field over $\Q$ 
and, up to a rescaling by $\GL^+(2,\R)$, the map $H_1(X,\Q)^- \ni c \mapsto \omega(c)\in K(i)$ is an 
isomorphism of $\Q$-vector spaces. Thus $\omega(c)=0$ implies $c=0$ in $H_1(X,\Z)^-$. \medskip

For the second condition, let $\s_1,\s_2$ be a pair of twin saddle connections which are both invariant by  the Prym 
involution $\inv$. If $c=\s_1*(-\s_2) \in H_1(X,\Z)^-$ is separating then by cutting $X$ along $\s_1,\s_2$ and 
regluing the segments of the boundary of the two components, we get a pair of translation surfaces, each of which 
having a unique singularity with cone angle $4\pi$ (they thus belong to the stratum $\H(1)$). Since this stratum is empty
we get a contradiction and $c$ is non-separating {\em i.e.} $c \neq 0 \in H_1(X,\Z)^-$. One has 
$\omega(c)=\omega(\s_1) -\omega(\s_2)=0$ hence the first condition applies and $D$ is a square. \medskip

For the last condition, we set $c_j=\s_0*(-\s_j), \, j=1,2$. Remark that we have $\inv(c_1)=-c_2$ and 
$c_0=c_2-c_1$ in $H_1(X,\Z)$. 
Since $c_0$ is non-separating by assumption, it is a primitive element of $H_1(X,\Z)$. 
Observe that if one of the curves $c_1$ or $c_2$ is separating then the other is 
also separating (as $\inv(c_1)=-c_2$) and in this case $c_0=c_1-c_2= 0 \in H_1(X,\Z)$ contradicting the 
assumption. Hence both $c_1,c_2$ are non-separating. Let $c=c_1+c_2$. We have  $\inv(c)=-c$ so that 
$c\in H_1(X,\Z)^-$. If $c=0 \in H_1(X,\Z)$ then $c_2=-c_1$ {\em i.e.} $c_0=c_1-c_2=2c_1$: contradiction with 
the primitivity of $c_0\in H_1(X,\Z)$. Thus $c\neq 0 \in H_1(X,\Z)^-$. Since $\s_0,\s_1,\s_2$ are twin saddle connections, 
we have 
$$
\omega(c)=\omega(c_1)+\omega(c_2)=2\omega(\s_0)-\omega(\s_1)-\omega(\s_2)=0.
$$
Again the first condition applies and $D$ is a square.
\end{proof}

\begin{proof}[Proof of Theorem~\ref{thm:unstable:dec:collapse}]
We denote by $\{\s_i, \, i\in I\}$ the set of horizontal saddle connections on $(X,\omega)$ whose endpoints are the 
two distinct zeros of $\omega$ denoted by $P$ and $Q$. Recall that we always define the orientation of such a saddle connection to be from $P$ to $Q$, it is said to be {\em positively oriented} if the orientation is from the left to the right, otherwise it is said to be {\em negatively oriented}.  The corresponding holonomy vectors are $\{(s_i,0)=\omega(\s_i)\in \R^2,\, i\in I\}$.
For every $i\in I$, $\s_i$ is contained on the lower boundary of a unique cylinder. If $\s_i$ is positively oriented 
(namely $s_i>0$) then there exists $\s_j$ in the same lower boundary component as $\s_i$ which is 
negatively oriented. In particular, all the numbers $\{s_i\}$ cannot have the same sign.

Let us define 
$$
s_\mathrm{min}=\max\{-s_i, \, s_i>0\} \text{ and } s_\mathrm{max}=\min\{-s_i,\, s_i<0\}.
$$ 

\noindent If $(Y,\eta)=(X,\omega)+(x,0)$ then by construction $\eta(\s_i)=(s_i+x,0)$ and the surface $(Y,\eta)$ can be constructed 
from  the same cylinders as $(X,\omega)$. For all $x\in ]s_\mathrm{min},s_\mathrm{max}[$, $(X,\omega)+(x,0)$ is a well-defined surface in 
$\PrymD(2,2)^{\rm odd}$ since $s_i+x \neq 0, \, \forall i\in I$,  proving the first statement. We now prove the second assertion.

Let us analyze the case when $x$ tends to $s_\mathrm{min}$ (the case $x$ tends to $s_\mathrm{max}$ being similar).
Letting $\mathcal C_\mathrm{min}=\{\s_i, \, s_i=-s_\mathrm{min}\}$ and 
$\mathcal C_\mathrm{max}=\{\s_i, \, s_i=-s_\mathrm{max} \}$ (necessarily $|\mathcal C_\mathrm{min}| \leq 3$, and $|\mathcal C_\mathrm{max}|\leq 3$).  When 
$x\rightarrow s_\mathrm{min}$, only the saddle connections of $\mathcal C_\mathrm{min}$ can 
collapse to a point. We thus have three cases, parameterized by the number of elements of $\mathcal C_\mathrm{min}$.
\begin{enumerate}
\item $\mathcal C_\mathrm{min}=\{\s_{i_0}\}$: the unique saddle connection $\s_{i_0}$ is invariant by $\inv$
and $(X,\omega)+(x,0)$ converges to a surface in $\PrymD(4)$. \medskip

\item $\mathcal C_\mathrm{min}=\{\s_{i_1},\s_{i_2}\}$: $\s_{i_1}$ and $\s_{i_2}$ are exchanged by $\inv$
(otherwise the closed curve $c=\s_{i_1}*(-\s_{i_2}) \in H_1(X,\Z)^-$ represents a non zero element and, since
$\omega(c)=0$, Lemma~\ref{lm:twins:sc:Dsq} implies that $D$ is a square). 
By Theorem~\ref{thm:collapse:Prym2}, $(X,\omega)+(x,0)$ converges to a surface in $\Omega E_{D'}(2)^*$, for 
some $D'\in \{D, D/4\}$.\medskip

\item \label{case3} $\mathcal C_\mathrm{min}=\{i_0,i_1,i_2\}$: if there are two saddle connections in $\{\s_{i_0},\s_{i_1},\s_{i_2}\}$ that 
are invariant by $\inv$ then $D$ must be square (see Lemma~\ref{lm:twins:sc:Dsq}).
Hence one can assume that $\inv$ preserves $\s_{i_0}$ while it exchanges $\s_{i_1}$ and $\s_{i_2}$. 
If the closed curve $c_0=\s_{i_1}*(-\s_{i_2})$ is non-separating then $D$ must be a square (again by 
Lemma~\ref{lm:twins:sc:Dsq}). Thus $c_0$ is separating and $\{\s_{i_0},\s_{i_1},\s_{i_2}\}$ are homologous saddle 
connections. We only need to show that $X$ decomposes into three tori. Indeed, as $x$ tends to $s_\mathrm{min}$ 
the length of these saddle connections tends to  zero, and the limit surface is an element of $\PrymD(0,0,0)$. 
\end{enumerate}
Hence, in view of the above discussion, in order to finish the proof of the theorem, we need to show 
that, in case~\eqref{case3}, the complement of $\s_{i_0}\cup \s_{i_1}\cup \s_{i_2}$ has three connected components, each of which is a 
one-holed torus.

We begin by observing that $\s_{i_1},\s_{i_2}$ determine a pair of angle $(2\pi,4\pi)$ at $P$ and $Q$. 
Since $\inv$ exchanges $P$ and $Q$ and preserves the orientation of $X$, a careful look at the geodesic rays emanating from $P$ and $Q$ shows that the angles $2\pi$ at $P$ and the angle $2\pi$ at $Q$ belong to the same side of $c_0$. Cut $X$ along $c_0$, then glue the two segments in each boundary components together, we then obtain two closed translation surfaces, one of which has no singularities, hence must be a flat torus that will be denoted by $(X',\omega')$, the other one is then a surface $(X'',\omega'')$ in $\H(1,1)$. 

We have in $X'$ a marked geodesic segment $\s'$ which is the identification of $\s_1$ and $\s_2$, we denote the endpoints of this segment by $P'$ and $Q'$  such that $P'$ (resp. $Q'$) corresponds to $P$ (resp. to $Q$). For $(X'',\omega'')$, we denote the zeros of $\omega''$ by $P''$ and $Q''$ such that $P''$ (resp. $Q''$) corresponds to $P$ (resp. to $Q$). In $X''$ we have a pair of twin saddle connections $\s_0$ and $\s''$, where $\s''$ is the identification of $\s_1$ and $\s_2$. 

The involution $\inv$ induces an involution $\inv'$ on $X'$ and an involution $\inv''$ on $X''$. We can consider $\inv'$ and $\inv''$ as the restrictions of $\inv$ in $X'$ and $X''$ respectively. Note that $\inv'$ exchanges $P'$ and $Q'$ and $\inv'(\omega')=-\omega'$. Since $X'$ is an elliptic curve, there exists one such involution. We deduce in particular that $\inv'$ has four fixed points in $X'$, one of which is the midpoint of $\s'$, the other three are the fixed points of $\inv$. 

Recall that $\inv$ has four fixed points in $X$. Therefore, $\inv''$ has exactly two fixed points, one of which is the midpoint of $\s_0$ by assumption (recall that $\s_0$ is invariant by $\inv$), and the other one is the midpoint of $\s''$. Let $\iotaup$ denote the hyperelliptic involution of $X''$. Remark that $\iotaup$ has six fixed points. From the observations above, we can conclude that $\inv''\neq \iotaup$. 

We now claim that $\iotaup(\s_0)=-\s''$. Indeed, since $\iotaup$ is in the center of the group ${\rm Aut}(X'')$, we have $\iotaup\circ\inv''=\inv''\circ\iotaup$. Therefore $\iotaup$ preserves the set of fixed points of $\inv''$. If $\iotaup$ fixes the midpoint of $\s_0$, then it follows that $\iotaup\circ\inv''=\Id$, since both $\iotaup$ and $\inv''$ are involutions. Hence $\inv''=\iotaup$, and we have a contradiction. Therefore, $\iotaup$ must send the midpoint of $\s_0$ to the midpoint of $\s''$. Remark that $\iotaup^*\omega''=-\omega''$, which means that $\iotaup$ is an isometry of $(X'',\omega'')$. Thus $\iotaup$ maps $\s_0$ to another saddle connection such that $\omega''(\iotaup(\s_0))=-\omega''(\s_0)$. Since $\iotaup$ exchanges the zeros of $\omega''$, we conclude that $\iotaup(\s_0)=-\s''$.

Now, the element in $H_1(X'',\Z)$ represented by the closed curve $\s_0\cup \s''$ is preserved by $\iotaup$, which implies that this curve is separating. Cut $X''$ along $\s_0\cup \s''$, then glue the segments in the boundary of each component together, we then get two flat tori $(X''_1,\omega''_1)$ and $(X''_2,\omega''_2)$ which are exchanged by $\inv''$. This finishes the proof of Theorem~\ref{thm:unstable:dec:collapse}.
\end{proof}

\section{Cylinder decomposition of surfaces near $\PrymD(4)$ and $\PrymD(2)^\ast$}

Let $(X_0,\omega_0)$ be a surface in $\PrymD(4)$, and $\Psi:  \mathring{D}(\eps) \rightarrow \PrymD(2,2)^{\rm odd}$ be the map in Proposition~\ref{prop:Prym4:neighbor}.

\begin{Proposition}\label{prop:dec:near:Prym4}
Assume that the horizontal direction is completely periodic for $(X_0,\omega_0)$. Then there exists $0<\eps_1 < \eps$ such that for every $(X,\omega)\in \Psi(\mathring{D}(\eps_1))$, the horizontal direction is also completely periodic. Set $R_{(k,5)}(\eps_1)=\{\varrhoup e^{k \imath\frac{\pi}{5}}, \, 0 <  \varrhoup < \eps_1\}$, for $k=0,\dots,9$, and $\mathring{D}_{(k,5)}(\eps_1)=\{\varrhoup e^{\imath\thetaup}, \, 0< \varrhoup < \eps_1, \, (k-1)\pi/5 < \thetaup < k\pi/5\}$, for $k=1,\dots,10$. Then

\begin{enumerate}
  \item The cylinder decompositions in the horizontal direction of all surfaces in $\Psi(R_{(k,5)}(\eps_1))$ are unstable and have the same combinatorial data.
  \item The cylinder decompositions in the horizontal direction of all surfaces in $\Psi(\mathring{D}_{(k,5)}(\eps_1))$ are stable and have the same combinatorial data.
\end{enumerate}

\end{Proposition}

\begin{proof}
Let $\CCC_i, \, i=1,\dots,n$, denote the horizontal cylinders of $X_0$, and $\gamma_i$ denote the simple closed geodesic in $\CCC_i$ whose distances to the two boundary components of $\CCC_i$ are equal. Pick an $0<\eps_1<\eps$ small enough so that $D(P_0,\eps_1)=\{x\in X_0, \, \dist(x,P_0) < \eps_1\}$ is an embedded disc disjoint from the curves $\gamma_i$. 

By the choice of $\eps_1$, we  see that the map $\Psi$ is defined on the disc $\mathring{D}(\eps_1)$.  By definition, the surface $\Psi(\varrhoup e^{\imath\thetaup})$ has a small saddle connection (of length $\varrhoup^5$) in direction $5\thetaup$. It follows immediately that the horizontal direction is periodic for the surfaces in $\Psi(R_{(k,5)}(\eps_1))$. Since we have a horizontal saddle connection with distinct endpoints, the corresponding cylinder decomposition is unstable. Clearly, the combinatorial data of the decomposition of $\Psi(z)$ does not change as $z$ varies in $R_{(k,5)}(\eps_1)$.

Let us consider a surface $(X,\omega)=\Psi(z)$, where $z\in \mathring{D}_{(k,5)}(\eps_1)$. To simplify the proof, we will assume in addition that $z^5=(0,2h)$ with $0<h<\eps_1$, the general case can be proved by the same arguments. Recall that the cone angle at $P_0$ is $10\pi$, hence $D(P_0,\eps_1)$ is the union of $10$ half-discs $D^+_j=\{z\in \C, \, |z| < \eps_1, \, {\rm Re}(z) \geq 0\}, D^-_j=\{z \in \C, \, |z|< \eps_1, {\rm Re}(z)\leq 0\}, j=1,\dots,5$, which are glued together with the following rules (see Figure~\ref{fig:H4:splitting:H22})
\begin{itemize}
\item[$\bullet$] $D^+_j$ is glued to $D^-_j$ along the segment $\{{\rm Re}(z)=0, 0\leq {\rm Im}(z) < \eps_1\}$,
\item[$\bullet$] $D^-_j$ is glued to $D^+_{j+1}$ along the segment $\{{\rm Re}(z)=0, -\eps_1 < {\rm Im}(z) \leq  0\}$,
\end{itemize}

Set 
\begin{itemize}
\item[$\bullet$] $a^+_j=\{ z \in D^+_j, \, {\rm Im}(z)=0\}, \, a^-_j=\{z\in D^-_j, \, {\rm Im}(z)=0\}$,
\item[$\bullet$] $b^+_j=\{ z \in D^+_j, \, {\rm Im}(z)=h\}, \, b^-_j=\{z\in D^-_j, \, {\rm Im}(z)=h\}$,
\item[$\bullet$] $c^+_j=\{ z \in D^+_j, \, {\rm Im}(z)=-h\}, \, c^-_j=\{z\in D^-_j, \, {\rm Im}(z)=-h\}$,
\end{itemize}

Since the horizontal direction is periodic for $(X_0,\omega_0)$, we have a permutation $\piup$ of the set $\{1,\dots,5\}$ such that $a^-_{\piup(j)}$ and $a^+_j$ belong to the same saddle connection, which implies that $b^-_{\piup(j)}$ and $c^-_{\piup(j)}$ belong to the same geodesic rays which contain $b^+_j$ and $c^+_j$ respectively. 

Now the surface $(X,\omega)=\Psi(z)$ can be obtained from $(X_0,\omega_0)$ by replacing the disc $D(P_0,\eps_1)$ by another disc $\tilde{D}(\eps_1)$ constructed from the same half-discs $D^\pm_j$ with the following gluings (see Figure~\ref{fig:H4:splitting:H22} for the case $k=2$), here we use the convention $j\sim (j-5)$ if $j>5$, 

\begin{itemize}
\item[$\bullet$] $D^+_j$ is glued to $D^-_j$ along the segment $\{{\rm Re}(z)=0, \, h\leq {\rm Im}(z)<\eps_1\}$ for $ j \in\{k,k+1,k+2\}$.
\item[$\bullet$] $D^+_j$ is glued to $D^-_j$ along the segment $\{{\rm Re}(z)=0, \, -h\leq {\rm Im}(z)<\eps_1\}$ for $j \not\in\{k,k+1,k+2\}$.
\item[$\bullet$] $D^-_j$ is glued to $D^+_{j+1}$ along the segment $\{{\rm Re}(z)=0, \, -\eps_1 < {\rm Im}(z)\leq h\}$ for $j\in\{k,k+1\}$.
\item[$\bullet$] $D^-_j$ is glued to $D^+_{j+1}$ along the segment $\{{\rm Re}(z)=0, \, -\eps_1 < {\rm Im}(z)\leq -h\}$ for $j \not\in\{k,k+1\}$
\item[$\bullet$] $D^+_k$ is glued to $D^-_{k+2}$ along the segment $\{{\rm Re}(z)=0, -h\leq {\rm Im}(z) \leq h\}$.
\end{itemize}

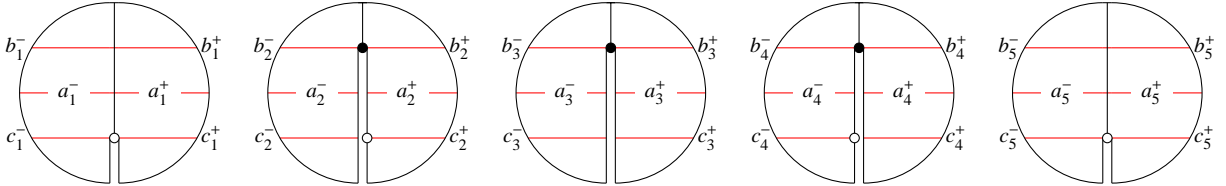
\begin{figure}[htb]
\centering
\begin{tikzpicture}[scale=0.6]
\foreach \x in {(-11,0),(-5.5,0),(0,0),(5.5,0), (11,0)} {\draw[red] \x +(-0.1,0) -- +(-2.1,0) \x +(0.1,0) -- +(2.1,0); \draw[red] \x ++(-0.1,0) +(0,1) -- +(150:2) +(0,-1) -- +(210:2); \draw[red] \x ++(0.1,0) +(0,1) -- +(30:2) +(0,-1) -- +(-30:2); }

\foreach \x in {(-11,0), (11,0)} {\draw[red] \x +(-0.1,1) -- +(0.1,1) +(-0.1,0) -- +(0.1,0);}

\draw (-12,0) node[fill=white] {$\scriptstyle a^-_1$} (-10,0) node[fill=white] {$\scriptstyle a^+_1$} (-12.7,1) node[left] {$\scriptstyle b^-_1$}  (-12.7,-1) node[left] {$\scriptstyle c^-_1$} (-9.3,1) node[right] {$\scriptstyle b^+_1$} (-9.3,-1) node[right] {$\scriptstyle c^+_1$};

\draw (-6.5,0) node[fill=white] {$\scriptstyle a^-_2$} (-4.5,0)  node[fill=white] {$\scriptstyle a^+_2$} (-7.2,1) node[left] {$\scriptstyle b^-_2$} (-7.2,-1) node[left] {$\scriptstyle c^-_2$} (-3.8,1) node[right] {$\scriptstyle b^+_2$} (-3.8,-1) node[right] {$\scriptstyle c^+_2$};

\draw (-1,0) node[fill=white] {$\scriptstyle a^-_3$} (1,0) node[fill=white] {$\scriptstyle a^+_3$} (-1.7,1) node[left] {$\scriptstyle b^-_3$}  (-1.7,-1) node[left] {$\scriptstyle c^-_3$} (1.7,1) node[right] {$\scriptstyle b^+_3$} (1.7,-1) node[right] {$\scriptstyle c^+_3$};

\draw (4.5,0) node[fill=white] {$\scriptstyle a^-_4$} (6.5,0) node[fill=white] {$\scriptstyle a^+_4$} (3.8,1) node[left] {$\scriptstyle b^-_4$}  (3.8,-1) node[left] {$\scriptstyle c^-_4$} (7.2,1) node[right] {$\scriptstyle b^+_4$} (7.2,-1) node[right] {$\scriptstyle c^+_4$};

\draw (10,0) node[fill=white] {$\scriptstyle a^-_5$} (12,0) node[fill=white] {$\scriptstyle a^+_5$} (9.3,1) node[left] {$\scriptstyle b^-_5$}  (9.3,-1) node[left] {$\scriptstyle c^-_5$} (12.7,1) node[right] {$\scriptstyle b^+_5$} (12.7,-1) node[right] {$\scriptstyle c^+_5$};

\foreach \x in {(0,0), (-5.5,0), (-11,0), (5.5,0), (11,0)} {\draw \x +(0.1,-2) arc (-90:90:2) \x +(-0.1,2) arc (90:270:2); \draw \x +(-0.1,2) -- +(0.1,2);}
\foreach \x in {(0,0), (-5.5,0), (5.5,0)} {\draw \x +(-0.1,2) -- +(0.1,2) \x +(0,1) -- +(0,2) \x +(-0.1,1) -- +(-0.1,-2) \x +(0.1,1) -- +(0.1,-2); \filldraw[fill=black] \x +(0,1) circle (3pt);  }
\foreach \x in {(11,0),(-11,0)} {\draw \x +(0,-1) -- +(0,2) \x +(-0.1,-1) -- +(-0.1,-2) \x +(0.1,-1) -- +(0.1,-2); \filldraw[fill=white] \x +(0,-1) circle (3pt);}

\filldraw[fill=white] (-5.5,0) +(0.1,-1) circle (3pt);
\filldraw[fill=white] (5.5,0) +(-0.1,-1) circle (3pt);

\end{tikzpicture}
\caption{Splitting a zero of order $4$ to two zeros of order $2$ ($k=2$).}
\label{fig:H4:splitting:H22}
\end{figure}

Let $P$ (resp. $Q$) denote the zero of $\omega$ corresponding to the point $(0,-h) \in D^+_k$ (resp. $(0,h) \in D^+_k$). It is clear from the gluing rules that any horizontal geodesic ray emanating from $P$ (reps. $Q$) ends up at $P$ (resp. $Q$). Thus $(X,\omega)$ admit as stable cylinder decomposition in the horizontal direction. Remark that the combinatorial data of the cylinder decomposition are encoded in the permutation $\piup$. Namely, $(X,\omega)$ has $n$ cylinders associated to the geodesics $\gamma_i, \, i=1,\dots,n,$ and $m$ additional cylinders, each of which corresponds to a  cycle of the permutation $(k,k+1,k+2)\circ\piup$. The core curves of the new cylinders contain the segments $a^{\pm}_j$. It is easy to check that the set of saddle connections contained in the upper and lower boundary components of a cylinder is completely determined by $\piup$ and $k$. The proposition is then proved.
\end{proof}

\begin{Remark}
 In general, the topological model of the decomposition of $(X,\omega)$ changes if we change the sector $\mathring{D}_{(k,5)}(\eps_1)$.
\end{Remark}

\medskip

By a saddle connection on $(X_0,\omega_0,W_0)\in \Omega E_{D'}(2)^*$, we refer to a geodesic segment whose endpoints are in the set $\{P_0,W_0\}$. We consider, by convention, a cylinder in $(X_0,\omega_0,W_0)$ as the union of all simple closed geodesics in the same  free homotopy  class  in $X_0\setminus\{P_0,W_0\}$.  Obviously, a direction $\thetaup$ is periodic for $(X_0,\omega_0,W_0)$ if and only if it is periodic for $(X_0,\omega_0)$, but the associated cylinder decomposition of $(X_0,\omega_0,W_0)$ may have one more cylinder than the one of $(X_0,\omega_0)$, since a simple closed geodesic passing through $W_0$ will cut the corresponding cylinder in $(X_0,\omega_0)$ into two cylinders in $(X_0,\omega_0,W_0)$. The following proposition follows from completely similar arguments as Proposition~\ref{prop:dec:near:Prym4}.

\begin{Proposition}\label{prop:cyl:dec:near:H2}
 Let $(X_0,\omega_0,W_0)$ be a surface in $\Omega E_{D'}(2)^*$. Assume that the horizontal direction is periodic for $(X_0,\omega_0,W_0)$. Let $\Psi :\mathring{D}(\eps) \rightarrow \PrymD(2,2)^{\rm odd}$ be the map defined in Proposition~\ref{prop:H2:neighbor}. Then there exists $0<\eps_1 <\eps$ such that for all $(X,\omega) \in \Psi(\mathring{D}(\eps_1))$, the horizontal direction is also periodic. Set $R_{(k,3)}(\eps_1)=\{\varrhoup e^{k\imath\frac{\pi}{3}}, 0 < \varrhoup < \eps_1\}, \, k=0,\dots,5$, and $\mathring{D}_{(k,3)}(\eps_1)=\{\varrhoup e^{\imath\thetaup}, \, 0 < \varrhoup < \eps_1,\, (k-1)\pi/3 < \thetaup < k\pi/3 \}, k=1,\dots,6.$ We have

\begin{enumerate}
 \item The associated cylinder decomposition of surfaces in $\Psi(R_{(k,3)}(\eps_1))$ are unstable and have the same combinatorial data.
 \item The associated cylinder decomposition of surfaces in $\Psi(\mathring{D}_{(k,3)}(\eps_1))$ are  stable and have the same combinatorial data.
\end{enumerate}
\end{Proposition}

\section{The set of Veech surfaces is not dense}
In this section we will prove the following theorem:
\begin{Theorem}
\label{thm:exist:open:noVeech}
If $D$ is not a square then for any connected component $\mathscr C$ of  $\Omega E_D(2,2)^{\rm odd}$,
there exists an open subset $\mathcal{U} \subset \mathscr C$ which contains no Veech surfaces.
\end{Theorem}

\subsection{Cylinder decomposition and prototypes}

We first prove the following lemma, which says that if we have a three tori decomposition such that the direction of the slits is periodic, then up to $\GL^+(2,\R)$, the surface belongs to the real kernel foliation leaf of some ``prototypical surface'' in a finite family.

\begin{Lemma}
\label{lm:3cyl:prototype}
Let $(X,\omega)\in\Omega E_D(2,2)^{\rm odd}$ be an eigenform with 
a triple of homologous saddle connections $\{\s_0,\s_1,\s_2\}$ so that
$(X,\omega)$ admits a three tori decomposition into tori $(X_j,\omega_j), j=0,1,2$. Assume that $(X,\omega)$ is periodic in the direction of $\s_0$. 
Let $(\widetilde{a}_j,\widetilde{b}_j)$ be a basis of $H_1(X_j,\Z)$ with $\widetilde{a}_j$ parallel to $\s_j$, and $\inv(\widetilde{a}_1)=-\widetilde{a}_2, \ \inv(\widetilde{b}_1)=-\widetilde{b}_2$, where $\inv$ is the Prym involution. Then there exists a tuple $(w,h,t,e)\in \Z^4$ satisfying
$$
(\mathcal{P}_D(0,0,0))\ 
\left\{\begin{array}{l}
w>0, h>0, 0\leq t<\gcd(w,h), \gcd(w,h,t,e)=1,\\
D=e^2+8wh
\end{array}\right.
$$

\noindent such that up to the action of $\GL^+(2,\R)$ and Dehn twists, we have
$$
\begin{array}{lll}
\omega(\Z\widetilde{a_0}\oplus \Z \widetilde{b_0})&=&\lbd\cdot \Z^2,\\
\omega(\Z\widetilde{a_j}\oplus \Z \widetilde{b_j}) &=& \Z(w,0) \oplus \Z(t,h) \qquad \text{ for }\ j=1,2,
\end{array}
$$
where  $\lbd \in \Q(\sqrt{D})$ is the unique positive root of the equation $\lbd^2-e\lbd -2wh=0$.
\end{Lemma}

\begin{proof}
We include a sketch of this result (compare with~\cite[Proposition 4.2]{Lanneau:Manh:H4}). 
Set  $\widetilde{a} = \widetilde{a_1} + \widetilde{a_2}$ and 
$\widetilde{b} = \widetilde{b_1} + \widetilde{b_2}$. We have  
$(\widetilde{a_0},\widetilde{b_0},\widetilde{a},\widetilde{b})$ is a symplectic basis 
of $H_1(X,\Z)^-$.  The restriction of the intersection form is given by the matrix
$\left(\begin{smallmatrix}
J & 0\\
0 & 2J\\
\end{smallmatrix}\right)$.

Since $(X,\omega)\in\Omega E_D(2,2)^{\rm odd}$, let us denote by $T$ a generator of the order $\Ord_D$.
In the above coordinates, since $T$ is self-adjoint, $T$ has the following form (up to replacing $T$ by $T-f\cdot \mathrm{Id}$)
$$
T=\left(\begin{smallmatrix}
e & 0 & 2w & 2t\\
0 & e & 2c  & 2h\\
h & -t & 0 & 0\\
-c & w & 0 & 0\\
\end{smallmatrix}\right),
$$
for some $(w,h,t,e,c)\in \Z^5$. Since $\omega$ is an eigenform, we have $T^\ast \omega = \lbd\cdot \omega$ 
for some $\lbd$ (that can be chosen to be positive by changing $T$ to $-T$). Now up to the 
action of $\GL^+(2,\R)$, one can always assume that $\omega(\Z\widetilde{a_0}\oplus \Z \widetilde{b_0})=
\lbd\cdot \Z^2$. Now in our coordinates, $\mathrm{Re}(\omega)=(\lbd,0,x,y)$ and 
$\mathrm{Im}(\omega)=(0,\lbd,0,z)$, for some $x,y,z>0$. Reporting into the equation $T^\ast \omega = \lbd\cdot \omega$,
we draw $x=2w$, $y=2t$, $z=2h$ and $c=0$. Since $T$ satisfies the quadratic equation $T^2-eT -2wh\mathrm{Id}=0$, we get
$D=e^2+8wh$. We can renormalize further using Dehn twists so that 
$0\leq t<\gcd(w,h)$. Finally properness of $\Ord_D$ implies $\gcd(w,h,t,e)=1$.
All the conditions of $\mathcal{P}_D(0,0,0)$ are now fulfilled and the lemma is proved.
\end{proof}

\begin{Definition}\label{def:prototype:3:tori}
For each $D$, let $\mathcal{P}_D(0,0,0)$ denote the set $\{ (w,h,t,e) \in \Z^4, \ (w,h,t,e) \text{ satisfies } ({\mathcal P}_D(0,0,0))\}$. 
We call an element of $\mathcal{P}_D(0,0,0)$ a {\em prototype}. The set of prototypes is clearly finite. 
\end{Definition}

\subsection{Switching decompositions} 
Let $(X,\omega)$ be a surface  in $\Omega E_D(2,2)^{\rm odd}$ which admits a three-tori decomposition by a triple of saddle connections  $\{\s_0,\s_1,\s_2\}$. We also assume that the direction of $\s_j$ is periodic. Let $(X_j,\omega_j)$ and $(\widetilde{a}_j,\widetilde{b}_j)$ be as in Lemma~\ref{lm:3cyl:prototype}.  We wish now to investigate the situation where $X$ admits other three-tori decompositions.

By Proposition~\ref{prop:3tori:decomp}, for any primitive element $b_0\in H_0(X_0,\Z)$, 
there  exists a unique primitive element $b_j\in H_1(X_j,\Z), j=1,2$ such that 
$$
\omega(b_j)=\frac{2\beta_j}{\lbd}\omega(b_0)
$$
with $\beta_j\in \N$. This is because $L(X_j,\omega_j)$ is a sublattice of $\frac{2}{\lbd}L(X_0,\omega_0)$ (see Proposition~\ref{prop:3tori:decomp}), hence it contains a vector parallel to $2/\lbd\omega_0(b_0)$ ($L(X_j,\omega_j)$ is the lattice associated to $(X_j,\omega_j)$). We call $b_j$ the {\em shadow} of $b_0$ in $X_j$.

The following lemma provides us with a sufficient condition of the existence of many other three-tori decompositions. Its proof is inspired from~\cite[Theorem~5.3]{Mc5}.

\begin{Lemma}\label{lm:switch:small:slit}
Let $b_0$ be a primitive element of $H_1(X_0,\Z)\setminus\{\pm \widetilde{a}_0\}$ and let $b_j$ be the shadows of $b_0$ in $X_j, j=1,2$. Set $c= b_0+b_1+b_2$. Then there exists $s_0>0$ such that if the ratio  $s=|\s_0|/|\widetilde{a}_0|$ is smaller than $s_0$, then the surface $(X,\omega)$ admits a three-tori decomposition by a triple of saddle connections $\{\delta_0,\delta_1,\delta_2\}$ such that $\delta_j*(-\s_j)=c$. 
\end{Lemma}

\begin{proof}
For $v_1=(x_1,y_1), v_2=(x_2,y_2)$ in $\R^2$, let us define $v_1\wedge v_2=\det\left(\begin{smallmatrix} x_1 & x_2 \\ y_1 & y_2 \end{smallmatrix}\right)$.
By assumption, we have $b_0\not\in \Z\widetilde{a}_0$, hence $|\omega(b_0)\wedge\omega(\widetilde{a}_0)|>0$. Since $\omega(b_j)$ is parallel to $\omega(b_0)$, and $\omega(\widetilde{a}_j)$ is parallel to $\omega(\widetilde{a}_0)$, we also have $|\omega(b_j)\wedge\omega(\widetilde{a}_j)|>0$. 

Choose $s_0$ small enough so that if $0<s< s_0$, then $0< s|\omega(b_j)\wedge\omega(\widetilde{a}_j)| < \mathbf{Area}(X_j)$. Assume that $|\s_j|<s_0|\widetilde{a}_j|$ for  $j=0,1,2$. Note that $|\s_0|=|\s_1|=|\s_2|$, and $|\widetilde{a}_1|=|\widetilde{a}_2|=w/\lbd |\widetilde{a}_0|$.

Let $\hat{\s}_j$ be the marked geodesic segment corresponding to $\{\s_0,\s_1,\s_2\}$ in the torus $X_j$, and let $\gamma_j$ be a simple closed geodesic representing the homology class $b_j\in H_1(X_j,\Z)$. By assumption, we have $0<|\omega(\gamma_j)\wedge\omega(\hat{\s}_j)|<\mathbf{Area}(X_j)$, hence $\gamma_j$ intersects $\hat{\s}_j$ at at most one point. Thus the union of all the geodesics representing $b_j$ which intersect $\hat{\s}_j$ is an embedded cylinder $\hat{\CCC}_j$ in $X_j$.

Recall that $(X,\omega)$ is obtained from $X_0,X_1,X_2$ by slitting and regluing along $\hat{\s}_j$. As a consequence, we see that the union of the cylinders $\hat{\CCC}_j, \, j=0,1,2$, is an embedded cylinder $\CCC$ whose waist curves represent the homology class $c=b_0+b_1+b_2$.  Let $\delta_j $ be the image of $\s_j$ under a Dehn twist in $\CCC$, then $\{\delta_j, \, j=0,1,2\}$ is also a triple of homologous saddle connections which decompose $X$ into three tori (see Figure~\ref{fig:switch:small:slit}). By definition, we have $\delta_j*(-\s_j)=c$,  and the lemma follows. Remark that the direction of $b_0$ is periodic.
\begin{figure}[htb]

\centering
\begin{tikzpicture}[scale=0.8]
\fill[red!20!yellow!40] (0,4) -- (0,2) -- (8,2) -- (8,4) -- cycle;

\draw (0,4) -- (0,0) -- (2,0) -- (2,2) -- (6,2) -- (6,0) -- (8,0) -- (8,4) -- (6,4) -- (6,7) -- (2,7) -- (2,4) -- cycle;

\draw (2,4) -- (2,2) (6,4) -- (6,2) (2,4) -- (6,4) (0,2) --(2,2) (6,2) -- (8,2);

\draw[red] (0,4) -- (8,2)  (2,4) -- (8,2.5) (6,4) -- (8,3.5) (0,3.5) -- (6,2) (0,2.5) -- (2,2);

\foreach \x in {(2,7), (6,7), (0,2), (2,2), (6,2), (8,2)} \filldraw[fill=white] \x circle (2pt);

\foreach \x in {(0,4), (2,4), (6,4), (8,4), (0,0), (2,0), (6,0), (8,0)} \filldraw[fill=black] \x circle (2pt);

\draw (1,-0.5) node {\small $X_1$} (4,1.5) node {\small $X_0$} (7,-0.5) node {\small $X_2$};
\draw (1,2) node[below] {\tiny $b_1$} (4,4) node[above] {\tiny $b_0$} (7,2) node[below] {\tiny $b_2$};
\draw[red] (4,3.2) node {\tiny $\delta_0$} (5,3.5) node {\tiny $\delta_1$} (3.2,2.5) node {\tiny $\delta_2$} ;
\draw (-0.2,3) node {\tiny $\s_0$} (8.2,3) node {\tiny $\s_0$}  (1.8,2.6) node {\tiny $\s_1$} (6.2,3.5) node {\tiny $\s_2$};

\end{tikzpicture}

\caption{ Switching three-tori decomposition.}
\label{fig:switch:small:slit}
\end{figure}
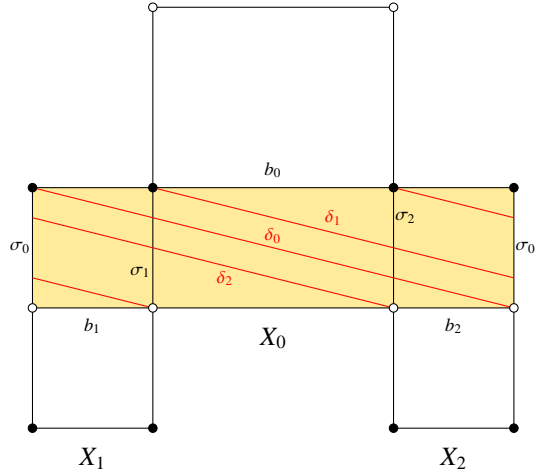
\end{proof}

Using the same notations as in Lemma~\ref{lm:switch:small:slit}. Let $(X'_j,\omega'_j), \; j=0,1,2$, denote the tori in the decomposition specified by $\{\delta_0,\delta_1,\delta_2\}$ ($X'_0$ is the torus which is fixed by $\inv$). We regard $X_j$ and $X'_j$ as  subsurfaces of $X$. The following elementary lemma provides us with an explicit basis of $H_1(X'_0,\Z)$, its proof is left to the reader.

\begin{Lemma}\label{lm:switch:decomp:basis}
Let $a_0$ be a primitive  element of $H_1(X_0)$ such that $(a_0,b_0)$ is a basis of $H_1(X_0,\Z)$. Then we have $H_1(X'_0,\Z)=\Z\cdot(a_0+c)+\Z\cdot b_0$.
\end{Lemma}

Next, we have

\begin{Lemma}\label{lm:switch:cond:periodic}
Let $(X,\omega)$ be a surface in $\Omega E_D(2,2)^{\rm odd}$ satisfying the hypothesis of Lemma~\ref{lm:switch:small:slit}. Let $a_0$ be a primitive element of $H_1(X_0,\Z)$ such that $(a_0,b_0)$ is a basis of $H_1(X_0,\Z)$, then we can write $\widetilde{a}_0=pa_0+qb_0$ with $(p,q) \in \Z^2$. Set $\beta=2\beta_1+2\beta_2=4\beta_1 \in \Z$, where $\omega(b_j)=(2\beta_j/\lbd)\omega(b_0)$.
Assume that the direction of $\delta_0$ is completely periodic, then we have 

\begin{equation}\label{eq:switch:dec:per:cond}
s=\frac{\lbd+\beta}{(rp+p-q)\lbd +p\beta}
\end{equation}

\noindent with $r \in \Q$.
\end{Lemma}
\begin{proof}
We know that the saddle connections $\{\delta_0,\delta_1,\delta_2\}$ decompose $X$ into three tori $X'_0,X'_1,X'_2$, where $X'_0$ is preserved by $\inv$. By Lemma~\ref{lm:switch:decomp:basis} we  have $H_1(X'_0,\Z)=\Z\cdot (a_0+b_0+b_1+b_2)+\Z\cdot b_0$. Set $A=\omega(a_0+b_0+b_1+b_2), B=\omega(b_0)$, then we have  $L(X'_0)=\Z A+ \Z B$, where $L(X'_0)$ is the lattice associated to $X'_0$. Set $v=\omega(\s_0)$, $w=\omega(\delta_0)$. We have 
$$A=\omega(a_0)+\omega(b_0) +\frac{\beta}{\lbd}\omega(b_0)=\omega(a_0)+(1+\frac{\beta}{\lbd})B.$$
Thus
$$ \omega(a_0)=A-(1+\frac{\beta}{\lbd})B.$$
Since $ \widetilde{a}_0=pa_0+qb_0$, we have
$$v=s\omega(\widetilde{a}_0)=s(p\omega(a_0)+q\omega(b_0))=s(p(A-(1+\frac{\beta}{\lbd})B) +qB)=s(pA+(q-p(1+\frac{\beta}{\lbd}))B).$$
Now
\begin{eqnarray*}
w & = & v+\omega(b_0+b_1+b_2)\\
  & = & spA+s(q-p(1+\frac{\beta}{\lbd}))B +(1+\frac{\beta}{\lbd})B\\
  & = & spA +(sq+(1-sp)(1+\frac{\beta}{\lbd}))B.
\end{eqnarray*}
The direction of $\delta_0$  is periodic if and only if $w$ is parallel to a vector in the lattice $\Z A+\Z B$, which is equivalent to
$$ r=\frac{sq +(1-sp)(1+\frac{\beta}{\lbd})}{sp} =\frac{sq\lbd+(1-sp)(\lbd+\beta)}{sp\lbd} \in \Q.$$
\noindent It follows
$$ srp\lbd=sq\lbd+(\lbd+\beta)-sp(\lbd+\beta),$$
\noindent or equivalently
$$s=\frac{\lbd+\beta}{rp\lbd-q\lbd+p(\lbd+\beta)}=\frac{\lbd+\beta}{(rp+p-q)\lbd+p\beta}.$$
\end{proof}
 
We can now prove 
 
\begin{Proposition}\label{prop:finite:sol}
Let $(X,\omega)$ be a surface in $\Omega E_D(2,2)^{\rm odd}$, where $D$ is not a square. Assume that there exists a triple of homologous saddle connections $\{\s_0,\s_1,\s_2\}$ which decompose $(X,\omega)$ into three tori, and the direction of $\s_j$ is periodic. Set $\DS{s=\frac{|\s_0|}{|\widetilde{a}_0|}}$, where $\widetilde{a}_0$ is a simple closed geodesic parallel to $\s_0$ in the torus which is preserved by the involution. Then there exists a constant $s_0>0$ depending only on $D$ such that if $s<s_0$ then $(X,\omega)$ is not a Veech surface.
\end{Proposition}

\begin{proof}
Let $(\widetilde{a}_j,\widetilde{b}_j), \, j=0,1,2,$ be as in Lemma~\ref{lm:3cyl:prototype}. Let $(e,w,h,t)$ be the prototype in ${\mathcal P}_D(0,0,0)$ which is associated to the cylinder decomposition in the direction of $\s_0$. Set $(a_0,b_0)=(\widetilde{a}_0,\widetilde{b}_0)$, and $(a'_0,b'_0)=(\widetilde{a}_0+\widetilde{b}_0,\widetilde{a}_0+2\widetilde{b}_0)$. Let $b_j$ (resp. $b'_j$) be the shadow of $b_0$ (resp. $b'_0$) in $X_j, \, j=1,2$. We have 
$$\omega(b_1+b_2)=\frac{\beta}{\lbd}\omega(b_0), \, \omega(b'_1+b'_2)=\frac{\beta'}{\lbd}\omega(b'_0),$$

\noindent where $\beta,\beta'\in \N$ are determined by the prototype $(e,w,h,t)$. From Lemma~\ref{lm:switch:small:slit}, there exists $s_1>0$ such that if $s<s_1$, then $(X,\omega)$ admits three-tori decompositions by the triples of saddle connections $\{\delta_j, \, j=0,1,2\}$ and $\{\delta'_j, \,  j=0,1,2\}$, where $\delta_0$ and $\delta'_0$ satisfy
$$\delta_0*(-\s_0)=b_0+b_1+b_2 \in H_1(X,\Z), \, \text{ and } \delta'_0*(-\s_0)=b'_0+b'_1+b'_2 \in H_1(X,\Z).$$

By definition, we have $\DS{\widetilde{a}_0=a_0=2a'_0-b'_0}$. Assume that $(X,\omega)$ is a Veech surface, then the directions of $\delta$ and $\delta'$ must be periodic, hence, from Lemma~\ref{lm:switch:cond:periodic}, we have
\begin{equation}\label{eqn:Vch:per:cond:1}
s=\frac{\lbd+\beta}{(r+1)\lbd+\beta}=\frac{\lbd +\beta'}{(2r'+3)\lbd+2\beta'}
\end{equation}

\noindent with $r,r'\in \Q$. Set $R=r+1, R'=2r'+3$, we see that the equation (\ref{eqn:Vch:per:cond:1}) is equivalent to
$$R'\lbd^2+(R'\beta+2\beta')\lbd +2\beta\beta'= R\lbd^2+(R\beta'+\beta)\lbd +\beta\beta'$$

\noindent Using $\lbd^2=e\lbd+2wh$, we get
\begin{eqnarray*}
R'(e\lbd+2wh)+(R'\beta+2\beta')\lbd + 2\beta\beta' & = & R(e\lbd+2wh) +(\beta+R\beta')\lbd +\beta\beta'\\
\Leftrightarrow (R'e+R'\beta+2\beta')\lbd +(2whR'+2\beta\beta') & = & (Re+\beta+R\beta')\lbd + (2whR+\beta\beta')
\end{eqnarray*}

\noindent It follows
$$\left\{ \begin{array}{l}
R'(e+\beta)+2\beta'=R(e+\beta')+\beta\\
2whR'+2\beta\beta'=2whR+\beta\beta'
\end{array}\right.$$

\noindent or
\begin{equation}\label{eqn:Vch:per:cond:2}
\left\{\begin{array}{l}
R(e+\beta')-R'(e+\beta)=2\beta'-\beta\\
\DS{R-R'=\frac{\beta\beta'}{2wh}}.
\end{array}\right.
\end{equation}

We first remark that $\beta\neq\beta'$, otherwise Equation(\ref{eqn:Vch:per:cond:1}) would imply that $(R-R')\lbd=\beta$, and hence $R-R'\not\in \Q$ since $\beta\neq 0$. It follows that the linear system (\ref{eqn:Vch:per:cond:2}) has a unique solution. Let $s_2$ be the value of $s$ corresponding to this solution which given by Equation (\ref{eqn:Vch:per:cond:1}). It follows that if $s<\min\{s_1,s_2\}$ then the directions of $\delta_0$ and $\delta'_0$ cannot be both periodic, hence $(X,\omega)$ cannot be a Veech surface. Since the set ${\mathcal P}_D(0,0,0)$ is finite, the proposition follows.
\end{proof}

The next proposition  is a direct consequence of Proposition~\ref{prop:finite:sol}. 

\begin{Proposition}\label{prop:no:Vee:near:3tori}
Let $\{(X_j,\omega_j,P_j), \, j=0,1,2\}$ be an element of $\PrymD(0,0,0)$, and $\Psi$ be the map in Proposition~\ref{prop:3tori:neighbor}. Then  there exists $0<\delta<\eps$ such that if $(X,\omega)\in \Psi(\mathring{D}(\delta))$, then $(X,\omega)$ is not a Veech surface.
\end{Proposition}

\begin{proof}
Let $\ell_0$ be the length of the shortest simple closed geodesic in $(X_0,\omega_0)$, and $s_0$ be the constant in Proposition~\ref{prop:finite:sol}.  Pick $\delta< \min\{\eps, s_0\ell_0\}$. By definition, if $(X,\omega)=\Psi(z)$, then we have a triple of homologous saddle connections $\{\s_0,\s_1,\s_2\}$ which decompose $X$ into three tori such that $\omega(\s_j)=z$. Assume that $z\in \mathring{D}(\delta)$. We have two cases
\begin{itemize}
 \item $z$ is not parallel to any vector in $L(X_0)$, the lattice associated to $X_0$. In this case, the direction of $\s_j$ is not periodic, hence $(X,\omega)$ is not a Veech surface.
 \item $z$ is parallel to some vector in $L(X_0,\omega_0)$. Let $v$ be the primitive vector in $L(X_0,\omega_0)$ in the same direction as $z$, then $(X,\omega)$ admits a decomposition into three cylinder in the direction of $z$, and the width of the cylinder invariant by the Prym involution is $|v|$. By assumption, we have 

$$\frac{|\s_0|}{|v|} \leq \frac{|\s_0|}{\ell_0} < s_0.$$

\noindent Therefore, $(X,\omega)$ cannot be a Veech surface by Proposition~\ref{prop:finite:sol}
\end{itemize}

The proposition is then proved. 
\end{proof}

Using Proposition~\ref{prop:no:Vee:near:3tori}, we can now prove the theorem announced at the beginning of the
section.

\begin{proof}[Proof of Theorem~\ref{thm:exist:open:noVeech}]
Fix a connected component $\mathscr C$ of  $\Omega E_D(2,2)^{\rm odd}$. By the main result of \cite{Lanneau:Manh:composantes}, we know that there exists a surface $(X,\omega)\in \mathscr{C}$ which admits a three-tori decomposition by a triple of homologous saddle connections $\{\s_0,\s_1,\s_2\}$. 

We can assume that the direction of $\s_j$ is periodic. By Lemma~\ref{lm:3cyl:prototype}, we get a
prototype $(w,h,t,e)$ in $\mathcal{P}_D(0,0,0)$. Set $L_0=\Z(\lbd,0)+\Z(0,\lbd)$, $L_1=L_2=\Z(w,0)+\Z(t,h)$, and $(X_j,\omega_j)=\C/L_j, \, j=0,1,2$. The triple $\{(X_j,\omega_j), _, j=0,1,2\}$ belongs to $\PrymD(0,0,0)$. Let $\Psi: \mathring{D}(\eps) \rightarrow \Omega E_D(2,2)^{\rm odd}$ be the map in Proposition~\ref{prop:3tori:neighbor}. It is  easy to see that $\Psi(\mathring{D}(\eps)) \subset \mathscr{C}$. From Proposition~\ref{prop:no:Vee:near:3tori},  we know that there exists $0<\delta < \eps$ such that the set $\mathcal{V}=\Psi(\mathring{D}(\delta))$ does not contain any Veech surface. As a consequence the set $\mathcal{U}=\GL^+(2,\R)\cdot \mathcal{V}$ does not contain any 
Veech surface either. It is easy to see that $\mathcal{U}$ is an open subset of $\mathscr C$. The theorem is then proved.
\end{proof}

\section{Finiteness of closed orbits}
\label{sec:proof:finiteness}

In this section we will prove our main second main result, namely:

\begin{Theorem}
\label{thm:fin:close:orb}
If $D$ is not a square then the number of closed $\GL^+(2,\R)$-orbits in $\PrymD(2,2)^{\rm odd}$ is finite.
\end{Theorem}

We first show a useful finiteness result up to the kernel foliation for surfaces in  $\PrymD(2,2)^{\rm odd}$.
Recall that $(X,\omega)$ admits an unstable cylinder decomposition in the horizontal direction if and only if this 
direction is periodic, and there exists (at least) one horizontal saddle connection whose endpoints are distinct zeros of 
$\omega$.

\begin{Theorem}
\label{thm:prot:unstable:dec}
If $D$ is not a square then there exists a finite family ${\mathcal P}_D$ of surfaces in $ \PrymD(2,2)^{\rm odd}$ 
such that for any $(X,\omega)\in \PrymD(2,2)^{\rm odd}$ with an unstable cylinder decomposition one has, 
up to rescaling by $\GL^+(2,\R)$:
$$
(X,\omega) = (X_i,\omega_i)+ (x,0) \qquad \textrm{for some } (X_i,\omega_i)\in \mathcal P_D.
$$
\end{Theorem}
If we label the zeros of $\omega$ by $P$ and $Q$, we always choose the orientation for any saddle connection joining 
$P$ and $Q$ to be {\em from $P$ to $Q$}: this defines in a unique way the surface $(X,\omega)+ (x,0)$.

\begin{proof}[Proof of Theorem~\ref{thm:prot:unstable:dec}]
By~\cite{Mc4}, for any $D' \equiv 0,1 \mod 4, D'>0$, the set $\Omega E_{D'}(2)^*$ is a finite union of 
Teichm\"uller curves. More precisely there exists a finite family ${\mathbb P}_{D'}(2)$ of 
surfaces ({\em prototypical splittings}) such that any $(X,\omega)\in \Omega E_{D'}(2)^*$
that is horizontally periodic belongs to the $\mathrm P$-orbit (here ${\mathrm P}=\{\left(\begin{smallmatrix} * & * \\ 0 & *  \end{smallmatrix} \right) \subset \GL^+(2,\R)\}$) of some surface in ${\mathbb P}_{D'}(2)$.

In~\cite{Lanneau:Manh:H4}, we have proved the same result for the stratum $\PrymD(4)$:
there exists a finite family ${\mathbb P}_D(4)$ of surfaces such that any horizontally periodic surface 
$(X,\omega)\in \Omega E_{D}(4)$ belongs to the $\mathrm P$-orbit of a surface in ${\mathbb P}_D(4)$.
The related statement for the stratum $\PrymD(0,0,0)$ corresponds to Lemma~\ref{lm:3cyl:prototype}: 
let ${\mathbb P}_D(0,0,0)$ be the  set of corresponding surfaces in $\PrymD(0,0,0)$. We will call the surfaces in the families ${\mathbb P}_{D'}(2)$, ${\mathbb P}_D(4)$, $\mathbb{P}_D(0,0,0)$  {\em prototypical surfaces}.

Given a discriminant $D>0$, for each prototypical surface $X_\infty$ in these finite families ${\mathbb P}_D(0,0,0)$, ${\mathbb P}_D(4)$ and ${\mathbb P}_{D'}(2)$, where $D'\in \{D, D/4\}$, we apply, respectively, Propositions~\ref{prop:3tori:neighbor}, \ref{prop:Prym4:neighbor} and \ref{prop:H2:neighbor}. This furnishes a map $\Psi: \mathring{D}(\eps) \rightarrow \PrymD(2,2)^{\rm odd}$ where $\eps>0$.  

By construction, surfaces in $\Omega E_D(2,2)^{\rm odd}$ whose horizontal kernel foliation leaf contains $X_\infty$, \ie $X_\infty$ is a limit of the real kernel foliation leaf through such surfaces, and close enough to $X_\infty$ are contained in the set $\Psi(R_{(k,n)}(\eps))$, where  $n\in\{1,3,5\}, \, k \in \{0,\dots,2n-1\}$, depending on the space to which $X_\infty$ belongs. For each prototypical surface, and each admissible pair $(k,n)$, we pick a surface in $\Psi(R_{(k,n)}(\eps))$. Let $\mathbb{P}_D$ denote this (finite) family. 
Note that for all the surfaces in this family, the cylinder decomposition in the horizontal direction is unstable. 
Now, thanks to Theorem~\ref{thm:unstable:dec:collapse}, if $(X,\omega) \in \PrymD(2,2)^{\rm odd}$ admits an unstable cylinder  decomposition, 
then up to action of $\GL^+(2,\R)$, the horizontal kernel foliation leaf through $(X,\omega)$ contains  some prototypical 
surface. Therefore $(X,\omega)$ belongs to the same horizontal leaf of a surface in the family $\mathbb{P}_D$, 
and the theorem follows.
\end{proof}

We have now all necessary tools to prove our main result.
\begin{proof}[Proof of Theorem~\ref{thm:fin:close:orb}]
Let $\{(X_i,\omega_i), \, i \in I\}$ be a family of Veech surfaces that generates an infinite family of closed $\GL^+(2,\R)$-orbits 
in $\PrymD(2,2)^{\rm odd}$. We will show that the set
$$
\Orb =\bigcup_{i\in I} \GL^+(2,\R)\cdot(X_i,\omega_i)
$$
is dense in a component of $\PrymD(2,2)^{\rm odd}$ contradicting Theorem~\ref{thm:exist:open:noVeech}.

Since the direction of any saddle connection on a Veech surface is periodic, each surface in the family 
$\{(X_i,\omega_i), \, i \in I\}$ admits infinitely many unstable cylinder decompositions. Therefore, we can assume that 
each of the surfaces $(X_i,\omega_i)$ belongs to the horizontal kernel foliation leaf of one of the surfaces in the family 
$\mathcal{P}_D$ of Theorem~\ref{thm:prot:unstable:dec}. Since the set $\mathcal{P}_D$ is finite, there exists a surface 
$(X,\omega)\in\mathcal{P}_D$ and an infinite subfamily $I_0\subset I$ such that $(X_i,\omega_i)=(X,\omega)+(x_i,0)$ for 
any $i \in I_0$. By Theorem~\ref{thm:unstable:dec:collapse}, $x_i \in ]a,b[$, where $a,b$ does not depend on $i$.

Compactness of the interval $[a,b]$ implies the existence of a subsequence $\{i_k\}_{k \in \N} \subset I_0$ 
such that $\{x_{i_k}\}$ converges to some $x\in [a,b]$. The sequence $(X_{i_k},\omega_{i_k})=(X,\omega)+(x_{i_k},0)$ thus 
converges to $(Y,\eta):=(X,\omega)+(x,0)$. If $x \in ]a,b[$ then $(Y,\eta)$ belongs to $\PrymD(2,2)^{\rm odd}$. However 
if $x\in \{a,b\}$ then $(Y,\eta)$ belongs to the boundary of the stratum $\PrymD(2,2)^{\rm odd}$, namely 
$\PrymD(4), \Omega E_{D'}(2)^*$ with $D'\in \{D,D/4\}$, or $\PrymD(0,0,0)$. We distinguish separately the four cases below.

\noindent {\bf Case $(Y,\eta) \in \PrymD(2,2)^{\rm odd}$.}\\
Let $v$ be a periodic direction on $(Y,\eta)$ that is different from $(1,0)$.
By Propositions~\ref{prop:stable:dec} and~\ref{prop:cyl:dec:unstable}, for $k$ large enough, $(X,\omega)+(x_{i_k},0)$ 
admits a stable cylinder decomposition in this direction. Moreover, we can assume that the decompositions of $(X_{i_k},\omega_{i_k})$ in direction 
$v$ share the same combinatorial data, and the same widths of cylinders. Finally, since $(X_{i_k},\omega_{i_k})$
are Veech surfaces, the direction $v$ is parabolic. The assumptions of Theorem~\ref{thm:byproduct}
are therefore fulfilled and there exists $\eps>0$ such that $(Y,\eta) + xv \in \overline{\mathcal O}$ 
for all  $x\in (-\eps,\eps)$. By Corollary~\ref{cor:1} there exists $\eps' >0$ so that 
$(Y,\eta) +w \in \overline{\mathcal O}$ for any $w\in \B(\varepsilon')$ proving that $\mathcal O$ is dense in the 
corresponding component of $\PrymD(2,2)^{\rm odd}$. \medskip

\noindent {\bf Case $(Y,\eta) \in \PrymD(4)$.}\\
In this case $(Y,\eta)$ is a Veech surface. Choose a periodic direction $v$ for $(Y,\eta)$
that is different from $(1,0)$. Let $\Psi: \mathring{D}(\eps) \rightarrow \PrymD(2,2)^{\rm odd}$  be the map in 
Proposition~\ref{prop:dec:near:Prym4}. We can assume that $\Psi(R_{(k,5)}(\eps))$ consists of surfaces in 
$\PrymD(2,2)^{\rm odd}$ which have a small saddle connection in direction $v$. 
There exists a sector $\mathring{D}_{(k,5)}(\eps)$ such that $\Psi(\mathring{D}_{(k,5)}(\eps))$ contains infinitely 
many elements of the family $\{(X_{i_k},\omega_{i_k})\}$. Note that every surface in $\Psi(\mathring{D}_{(k,5)}(\eps))$ 
admits  a stable cylinder decomposition in direction $v$ with the same combinatorial data and the same widths of cylinders. 
A statement similar to Theorem~\ref{thm:byproduct} also holds for this case, showing that there exists $0<\delta < \eps$ such that 
$\Psi(R_{(k-1,5)}(\delta))$ is included in $\overline{\Orb}$. Hence $\overline{\Orb}$ is dense in the corresponding
component of $\PrymD(2,2)^{\rm odd}$. \medskip

\noindent {\bf Case $(Y,\eta) \in \Omega E_{D'}(2)^*$.}\\
In particular $(Y,\eta)$ is a Veech surface (viewed as a surface of 
$\Omega E_{D'}(2)$). The same arguments as above show that $\overline{\Orb}$ is dense in the corresponding
component of $\PrymD(2,2)^{\rm odd}$. \medskip

\noindent {\bf Case $(Y,\eta) \in \PrymD(0,0,0)$.}\\
In this case $(X,\omega)$ has a triple of horizontal saddle connections $\{\s_0,\s_1,\s_2\}$ that decompose the 
surface into a connected sum of three tori, and $(Y,\eta)$ can be viewed as the limit when the length of $\s_j$ 
goes to zero. By Proposition~\ref{prop:no:Vee:near:3tori}, there is no Veech surface in the neighborhood of 
$(Y,\eta)$. This is a contradiction. \medskip

From above discussion, we draw that $\overline{\Orb}$ is dense in a component of 
$\PrymD(2,2)^{\rm odd}$: this is a contradiction with Theorem~\ref{thm:exist:open:noVeech}.
The proof of Theorem~\ref{thm:fin:close:orb} is now complete.
\end{proof}

\bigskip

\appendix

\section{Existence of Veech surfaces in infinitely many Prym eigenform loci}

It follows from the work of McMullen~\cite{Mc6bis} that there exists only finitely many $\GL^+(2,\R)$ closed orbits in the union $\underset{D \text{ not a square}}{\bigcup} \Omega E_D(1,1)$ (see \cite{Lanneau:Moeller} for a similar result in $\PrymD(1,1,2)$). However the situation is different in $\PrymD(2,2)^{\rm odd}$. We will show that for infinitely  many discriminants $D$ that are not squares, the locus $\PrymD(2,2)^{\rm odd}$ contains at least one $\GL^+(2,\R)$ closed orbit (the fact that $\Omega E_{D_1}(2,2)^{\rm odd}$ and $\Omega E_{D_2}(2,2)^{\rm odd}$ are disjoint if $D_1\neq D_2$ will be proved in \cite{Lanneau:Manh:composantes}). Remark that the corresponding Veech surfaces we found are not primitive, they are double coverings of surfaces in $\PrymD(2)$. It is unknown to the authors if there exists any primitive Veech surface in $\underset{D \text{ not a square}}{\bigcup}\Omega E_D(2,2)^{\rm odd}$.

Following~\cite{Mc4} we say that a quadruple of integers $(w,h,t,e)$ is a 
{\em splitting prototype} of discriminant $D$ if the conditions below are fulfilled:
$$
\left\{\begin{array}{l}        w>0,h>0,\;        0\leq
  t<\gcd(w,h),\\    \gcd   (w,h,t,e)    =1,\\    D=e^2+4wh,\\    0<
  \lbd:=\frac{e+\sqrt{D}}{2}<w.\\
\end{array}
\right.
$$
To each splitting prototype one can associate a Veech surface $(X,\omega)\in \PrymD(2)$ as follows
(see Figure~\ref{fig:H2}).
\begin{figure}[htbp]
\begin{center}
\begin{tikzpicture}[scale=0.8]
%
\draw (0,0)  -- (0,3)  -- (5,3)  -- (5,2) --  (2,2) --  (2,0)-- cycle; 

\draw[->,>= angle 45, thick,  dashed] (0,1) -- (2,1); 
\draw[->,>= angle 45, thick,  dashed] (0,2.5) -- (5,2.5); 
\draw[->,>= angle 45, thick,  dashed] (3.5,2) -- (3.5,3); 
\draw[thick,  dashed, ->,  >=angle 45] (0,0)  .. controls (-1.5,0.5)  and (-1.5,  1.5) ..   (0,2);

\filldraw[fill=white,  draw=black]  (0,0)  circle (2pt)  (0,2)  circle
(2pt) (0,3) circle (2pt)  (2,3) circle (2pt) (5,3) circle (2pt) (5,2) 
circle  (2pt) (2,2)  circle (2pt) (2,0)  circle (2pt);

\draw  (1,1) node[above]  {$\scriptstyle a_1$}  
(-1.5,1) node {$\scriptstyle b_1$}  
(2.5, 2.5) node[above] {$\scriptstyle a_2$} 
(5,2.5) node[right] {$\scriptstyle b_2$};
\end{tikzpicture}
\end{center}
\caption{Prototypical splitting of type $(w,h,0,e)$ where $\omega(a_1)=(\lbd,0)$,
$\omega(b_1)=(0,\lbd)$, $\omega(a_2)=(w,0)$ and $\omega(b_2)=(0,h)$. 
Parallel edges are identified to obtain a surface $(X,\omega) \in \PrymD(2)$
}
\label{fig:H2}
\end{figure}
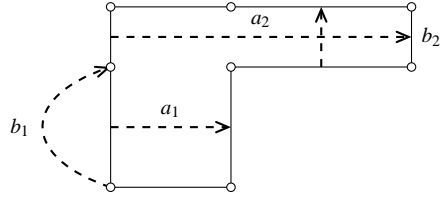

Define a pair of lattices in $\C$ by $\Lambda_1=\Z(\lbd,0)\oplus\Z(0,\lbd)$ and
$\Lambda_2=\Z(w,0)\oplus\Z(t,h)$ (recall that $\lbd:=\frac{e+\sqrt{D}}{2}>0$). 
We construct the corresponding tori $(E_i,\omega_i) = (\C/\Lambda_i,dz)$
and the genus two surface $(X,\omega)$ where $X=E_1\#E_2$ and $\omega=\omega_1+\omega_2$.

Geometrically, the surface $(X,\omega)$ is made of two horizontal cylinders whose core curves are denoted by 
$a_1$ and $a_2$ (see~\cite{Mc4} and Figure~\ref{fig:H2} for details). 

Let $\{a_1,b_1,a_2,b_2\}$ be the symplectic basis of $H_1(X,\Z)$ such that $\omega(a_1)=(\lbd,0)$, $\omega(b_1)=(0,\lbd)$, 
$\omega(a_2)=(w,0)$ and $\omega(b_2)=(t,h)$. A generator of the order $\Ord_D$ is given (in the above 
basis) by the following matrix
$$
T=\left( \begin{smallmatrix}
            e & 0 & w & t\\
        0 & e & 0 & h \\
        h & -t & 0 & 0\\
            0 & w & 0 & 0\\
\end{smallmatrix}\right).
$$

It is straightforward to check that $T$ is a self-adjoint with respect to the intersection form of $H_1(X,\Z)$, $T^2=eT+wh\Id$, and $T$ satisfies $T^*\omega=\lbd\omega$. It follows that $T$ generates a proper subring in ${\rm End}(\mathbf{Jac}(X))$ for which $\omega$ is an eigen vector. Thus $(X,\omega)\in \Omega E_D(2)$, and therefore $(X,\omega)$ is a Veech surface (see \cite{Mc6} for  more details).

\begin{Theorem}\label{thm:inf:Veech}
Let $(w,h,t,e)$ be a splitting prototype for a discriminant $D$, and $(X,\omega)$ be the associated Veech surface in $\Omega E_D(2)$. Let $(Y_1,\eta_1)$ and $(Y_2,\eta_2)$ be two surfaces in $\H(2,2)$ constructed from $(w,h,t,e)$ as shown in Figure~\ref{fig:double:cov:model1:2}. Then both $(Y_1,\eta_1)$ and $(Y_2,\eta_2)$ are Veech surfaces in  some Prym eigenform loci in $\H(2,2)^{\rm odd}$. More specifically,  we have 

\begin{itemize}
 \item[(i)] $(Y_1,\omega_1)\in \Omega E_{4D}(2,2)^{\rm odd}$ if $h$ is odd, otherwise $(Y_1,\eta_1)\in \Omega E_D(2,2)^{\rm odd}$,
 \item[(ii)] $(Y_2,\omega_2) \in \Omega E_{4D}(2,2)^{\rm odd}$ if $w$ is odd, otherwise $(Y_2,\eta_2)\in \Omega E_D(2,2)^{\rm odd}$.
\end{itemize}

\end{Theorem}

\begin{figure}[htb]
\begin{minipage}[t]{0.45\linewidth}
\centering
\begin{tikzpicture}
\fill[yellow!50!red!20] (0,2) -- (0,1) -- (1,1) -- (1,2) -- cycle;
\fill[yellow!50!red!20] (0,0) -- (0,-1) -- (1,-1) -- (1,0) -- cycle;

\draw (0,3) -- (0,-1) -- (1,-1) -- (1,0) -- (3,0) -- (3,1) -- (1,1) -- (1,2) -- (3,2) -- (3,3) -- cycle;

\foreach \x in {(0,2), (0,1), (0,0)} \draw \x -- +(1,0);

\foreach \x in {(0,1.5), (0,-0.5)} \draw[dashed, ->, >=angle 45] \x -- +(1,0);

\foreach \x in {(0,2.5), (0,0.5)} \draw[dashed, ->, >=angle 45] \x -- +(3,0);

\foreach \x in {(1,-1), (3,0), (1,1), (3,2)} \draw[dashed, ->, >=angle 45] \x .. controls +(0.5,0.25) and +(0.5,-0.25) .. +(0,1) ;

\foreach \x in {(0,3), (0,2), (0,-1), (1,3), (1,2), (1,-1),( 3,3), (3,2) } \filldraw[fill=black] \x circle (2pt);
\foreach \x in {(0,1), (0,0), (1,1), (1,0), (3,1), (3,0) } \filldraw[fill=white] \x circle (2pt);

\draw (0.5,1.5) node[above] {$\sst a_{11}$} (0.5,-0.5) node[above] {$\sst a_{12}$} (2,2.5) node[above] {$\sst a_{21}$} (2,0.5) node[above] {$\sst a_{22}$};

\draw (1.3,1.5) node[right] {$\sst b_{11}$} (1.3,-0.5) node[right] {$\sst b_{12}$} (3.3,2.5) node[right] {$\sst b_{21}$} (3.3,0.5) node[right] {$\sst b_{22}$};

\draw (1.5,-1.5) node {$(Y_1,\eta_1)$};
\end{tikzpicture}
\end{minipage}
\begin{minipage}[t]{0.45\linewidth}
\centering
\begin{tikzpicture}
\fill[yellow!50!red!20] (0,3) -- (0,2) -- (3,2) -- (3,3) -- cycle;
\fill[yellow!50!red!20] (0,1) -- (0,0) -- (3,0) -- (3,1) -- cycle;

\draw (0,3) -- (0,-1) -- (1,-1) -- (1,0) -- (3,0) -- (3,1) -- (1,1) -- (1,2) -- (3,2) -- (3,3) -- cycle;

\foreach \x in {(0,2), (0,1), (0,0)} \draw \x -- +(1,0);

\foreach \x in {(0,1.5), (0,-0.5)} \draw[dashed, ->, >=angle 45] \x -- +(1,0);

\foreach \x in {(0,2.5), (0,0.5)} \draw[dashed, ->, >=angle 45] \x -- +(3,0);

\foreach \x in {(1,-1), (3,0), (1,1), (3,2)} \draw[dashed, ->, >=angle 45] \x .. controls +(0.5,0.25) and +(0.5,-0.25) .. +(0,1) ;

\foreach \x in {(0,3), (0,0), (0,-1), (1,3), (1,0), (1,-1),( 3,3), (3,0) } \filldraw[fill=black] \x circle (2pt);
\foreach \x in {(0,2), (0,1), (1,2), (1,1), (3,2), (3,1) } \filldraw[fill=white] \x circle (2pt);

\draw (0.5,1.5) node[above] {$\sst a_{11}$} (0.5,-0.5) node[above] {$\sst a_{12}$} (2,2.5) node[above] {$\sst a_{21}$} (2,0.5) node[above] {$\sst a_{22}$};

\draw (1.3,1.5) node[right] {$\sst b_{11}$} (1.3,-0.5) node[right] {$\sst b_{12}$} (3.3,2.5) node[right] {$\sst b_{21}$} (3.3,0.5) node[right] {$\sst b_{22}$};

\draw (1.5,-1.5) node {$(Y_2,\eta_2)$};
\end{tikzpicture} 
\end{minipage}

\caption{Double coverings of a surface in $\Omega E_D(2)$: $\eta_i(a_{11})=\eta_i(a_{12})=\lbd, \eta_i(b_{11})=\eta_i(b_{12})=\imath\lbd$, $\eta_i(a_{21})=\eta_i(a_{22})=w, \eta_i(b_{21})=\eta_i(b_{22})=t+\imath h, \, i=1,2$.  The cylinders fixed by the Prym involution are colored.}
\label{fig:double:cov:model1:2}
\end{figure}
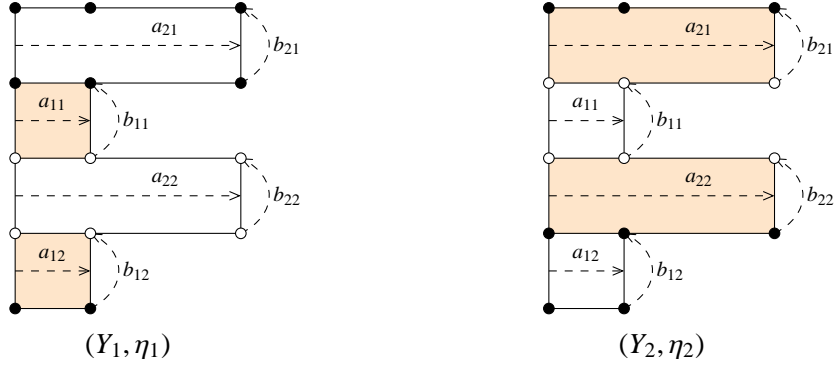

\medskip

\begin{Remark}\hfill
\begin{itemize}
 \item[$\bullet$] In general, the Teichm\"uller discs generated by $(Y_1,\omega_1)$ and by $(Y_2,\omega_2)$ are different, for instance when $h$ is odd, and $w$ is even.

 \item[$\bullet$] If $D \equiv 5 \mod 8$, then it is easy to see that $e,w,h$ are all odd. Therefore, in both construction $(Y_i,\eta_i)$ belongs to $\Omega E_{4D}(2,2)^{\rm odd}$.  
\end{itemize} 
\end{Remark}

\begin{proof}
It is easy to see that both $(Y_1,\eta_1)$ and $(Y_2,\eta_2)$ are double coverings of $(X,\omega)$, the deck transformation sends $a_{ij}$ to $a_{i j+1}$ and $b_{ij}$ to $b_{ij+1}$ (here we use the convention $(i3)\sim (i1)$). Since $(X,\omega)$ is a Veech surface both $(Y_1,\omega_1) $ and $(Y_2,\omega_2)$ are Veech surfaces (see \cite{GutJud00} and \cite{Masur2002}).

Remark that $Y_i$ has an involution $\inv_i$ that exchanges the zeros of $\eta_i$ such that $\inv_i^*\eta_i=-\eta_i$, in Figure~\ref{fig:double:cov:model1:2} the cylinders fixed by $\inv_i$ are colored. It follows that $(Y_i,\eta_i)$ belongs to the Prym locus $\Prym(2,2) \subset \H(2,2)^{\rm odd}$ ($\Prym(2,2)$ consists of double coverings of quadratic differentials in $\QQQ(-1^4,4)$). By some standard arguments (see \cite{Lanneau:Manh:H4} and \cite{Mc6}), we can conclude that $(Y_i,\eta_i)$ is a Prym eigenform, thus $(Y_i,\eta_i)$ is contained in some locus $\Omega E_{\widetilde{D}}(2,2)^{\rm odd}$. It remains to determine the discriminant $\widetilde{D}$.

Set $H_1(Y_i,\Z)^-=\{\alpha \in H_1(Y_i,\Z) \, | \, \inv_i(\alpha)=-\alpha\}$. Since $(Y_i,\eta_i)\in \Prym(2,2)$, we have $H_1(Y_i,\Z)^- \simeq \Z^4$.  We choose a basis of $H_1(Y_i,\Z)^-$ as follows: 

\begin{itemize}
\item[$\bullet$] for $(Y_1,\eta_1)$, set $\alpha_1 = a_{11} = a_{12}$ and $\alpha_2 = a_{21}+a_{22}$, we choose 
$\beta_1=b_{11}+b_{12}$ and $\beta_2=b_{21}+b_{22}$. In particular the restriction of the 
symplectic form has the following matrix $\left(\begin{smallmatrix} J & 0 \\ 0 & 2J\\ \end{smallmatrix}\right)$. 

\item[$\bullet$] for $(Y_2,\eta_2)$, set $\alpha_1=a_{11}+a_{12}, \alpha_2=a_{21}=a_{22}, \beta_1=b_{11}+b_{12}, \, \beta_2=b_{21}+b_{22}$. In this basis, the restriction of the intersection form to $H_1(Y_2,\Z)^-$ is given by $\left(\begin{smallmatrix} 2J & 0 \\ 0 & J\\ \end{smallmatrix}\right)$. 

\end{itemize}

In the above bases, the coordinates of $\eta_i$ are the following:
$$
\mathrm{Re}(\eta_1)= (\lbd,0,2w,2t) \qquad \mathrm{and} \qquad \mathrm{Im}(\eta_1)=(0,2\lbd,0,2h).
$$

$$
\mathrm{Re}(\eta_2)= (2\lbd,0,w,2t) \qquad \mathrm{and} \qquad \mathrm{Im}(\eta_2)=(0,2\lbd,0,2h).
$$


Let $\widetilde{T}_1$ be the following self-adjoint endomorphism of $H_1(Y_1,\Z)^-$ (given in the basis $\{\alpha_1,\beta_1,\alpha_2,\beta_2\}$):
$$
\widetilde{T}_1=\left( \begin{smallmatrix}
            2e & 0 & 4w & 4t\\
        0 & 2e & 0 & 2h \\
        h & -2t & 0 & 0\\
            0 & 2w & 0 & 0\\
\end{smallmatrix}\right).
$$

Similarly, let $\widetilde{T}_2$ be the self-adjoint  endomorphism of $H_1(Y_2,\Z)^-$ (given in the basis $\{\alpha_1,\beta_1,\alpha_2,\beta_2\}$) by the following matrix

$$
\widetilde{T}_2: = \left(\begin{smallmatrix}
 2e & 0 & w & 2t \\
   0 & 2e & 0 & 2h \\
   4h & -4t & 0 & 0 \\
   0  & 2w & 0 & 0 \\   
   \end{smallmatrix}\right)
$$

It is straightforward to check that $\widetilde{T}_i^\ast \eta_i=(2\lbd)\cdot \eta_i$ thus $\eta_i$ is an eigenform of $\widetilde{T}_i$. Remark that both $\widetilde{T}_i$ satisfy $\widetilde{T}_i^2-2e\widetilde{T}_i-4wh\mathrm{Id}=0$, which implies  that  $\widetilde{T}_i$ generates a self-adjoint  subring of ${\rm End}(\Prym(Y_i))$ isomorphic to $\Ord_{D'}$, where $D'= (2e)^2+16wh=4(e^2+4wh)=4D$.  

There exists a unique proper subring of  $\mathrm{End}(\Prym(Y_i))$ for which $\eta_i$ is an eigenform, this proper subring is isomorphic to a quadratic order $\Ord_{\widetilde{D}_i}$. Clearly, this subring must contain $\widetilde{T}_i$, hence it is generated by $\widetilde{T}_i/k_i$, where $k_1=\gcd(2e,4w,2h,2w,h,4t,2t)=\gcd(2e,2w,h,2t)$, and $k_2=\gcd(2e,w,2h,2t)$. Since 
$\gcd(w,h,t,e)=1$ we have $k_i\in\{1,2\}$. Note that $4D=k_i^2\widetilde{D}_i$, therefore  $\widetilde{D}_i=4D$ if $k_i=1$, and $\widetilde{D}_i=D$ if $k_i=2$. We can now conclude by noticing that $k_1=1$ if and only if $h$ is odd, and $k_2=1$ if and only if $w$ is odd.
\end{proof}



\begin{thebibliography}{EMM2}


\providecommand{\bysame}{\leavevmode ---\ }
\providecommand{\og}{``}
\providecommand{\fg}{''}

\bibitem[Arn81]{Ar1} 
{\scshape P.~Arnoux} -- {\og Un invariant pour les \'echanges 
d'intervalles et les flots sur les surfaces (French)\fg}, 
\emph{Th\`ese, Universit{\'e} de Reims} (1981).

\bibitem[AR08]{Avila2008} {\scshape A.~Avila, {\normalfont and} M.~Resende } --{\og Exponential mixing for the Teichm\"uller flow on the moduli space of quadratic differentials \fg}, 
\emph{Comment. Math. Helv.} {\bf 87} (2012), no.3, pp.~589-638.

\bibitem[Ba07]{Bai:07} {\scshape M.~Bainbridge} --  {\og Euler characteristic of Techm\"uller curves in genus two \fg}, {\em Geom. Topol.} {\bf 11} (2007), pp.~1887-2073.

\bibitem[Ba10]{Bai:10} {\scshape M.~Bainbridge} -- {\og Billiards in L-shaped tables with barriers \fg} , {\em Geom. Funct. Anal.} {\bf 20} (2010), no. 2, pp.~299-356. 

\bibitem[BaM\"ol12]{Bainbridge:Moeller12} {\scshape M.~Bainbridge \normalfont{ and } M.~M\"oller} -- {\og Deligne-Mumford compactification of the real multiplication locus and Teich\"uller curves in genus three \fg}, {\em Acta Math.} {\bf 208} (2012), pp.~1-92. 

\bibitem[BL08]{Lanneau:rauzy}
{\scshape C.~Boissy {\normalfont and} E.~Lanneau} -- {\og Dynamics and geometry of the Rauzy-Veech induction for quadratic differentials \fg}, 
\emph{Erg. Th. Dyn. Sys.} {\bf 29} (2009), pp.~767-816.

\bibitem[Bos88]{Boshernitzan1988}
{\scshape M.~ Boshernitzan } -- {\og Rank two interval exchange transformations \fg}, 
\emph{Erg. Th. Dyn. Sys.} {\bf 8} (1988), pp.~379-394.

\bibitem[C04]{Calta04} {\scshape K.~Calta } -- {\og Veech surfaces
  and      complete       periodicity      in      genus      two\fg},
  \emph{J. Amer. Math. Soc.} {\bf 17} (2004), no.~4, pp.~871--908.

\bibitem[CS07]{CaSm} {\scshape K.~Calta {\normalfont and} J.~Smillie} -- {\og Algebraically periodic translation surfaces\fg}, \emph{J. Mod. Dyn.} {\bf 2} (2007), no.2, pp.~209-248.

\bibitem[ChM06]{CM2006}
{\scshape Y.~Cheung {\normalfont and} H.~Masur} -- 
{\og Minimal nonergodic directions on genus 2 translation surfaces\fg}, 
\emph{Erg. Th. Dyn. Sys.} {\bf 26} (2006) pp.~341--351.

\bibitem[DN88]{Danthony1988}
{\scshape C.~Danthony {\normalfont and} A.~Nogueira} -- {\og 
Involutions lin\'eaires et feuilletages mesur\'es (French)\fg},
\emph{C. R. Acad. Sci. Paris S\'er. I Math.} {\bf 307} (1988), no. 8, pp.~409--412. 

\bibitem[EMi13]{Eskin:preprint}  {\scshape  A.~Eskin {\normalfont  and}
  M.~ Mirzakhani} -- {\og On invariant and stationary measures for the
  $\textrm{SL}(2,\mathbb   R)$   action    on   moduli   space   \fg},
  {\em arXiv:1302.3320} (2013).
  
\bibitem[EMiMo13]{EskMirMoh12} {\scshape A.~Eskin, M.~Mirzakhani, {\normalfont and } A.~Mohammadi} --{\og Isolation, Equidistribution, and Orbit Closures for the $\SL(2,\R)$ action on Moduli space\fg}, {\em arXiv:1305.3015} (2013).

\bibitem[EMZ03]{EMZ}  {\scshape  A.~Eskin, H.~Masur {\normalfont  and}
  A.~Zorich} -- {\og The Principal Boundary, Counting Problems and the Siegel--Veech Constants   \fg},
  \emph{Publ. Math. Inst. Hautes \'Etudes Sci.} {\bf 97} (2003), pp.~61--179.

\bibitem[GHSch03]{GutHubSch03} {\scshape E.~Gutkin, P.~Hubert, {\normalfont and } T.~Schmidt} -- {\og Affine diffeomorphisms of translation surfaces: periodic points, Fuchsian groups, and arithmeticity \fg}, {\em Ann. Sci. \'Ecole Norm. Sup. (4)} {\bf 36} (2003), no.6, pp.~847-866.

\bibitem[GJ00]{GutJud00} {\scshape E.~Gutkin {\normalfont and } C.~Judge} --{\og Affine mappings of translation surfaces: geometry and arithmetic \fg}, {\em Duke Math. J.} {\bf 103} (2000) no.2, pp.~191--213. 

\bibitem[HS04]{HubSch04} {\scshape P.~Hubert {\normalfont and } T.~Schmidt} -- {\og Infinitely generated Veech groups \fg}, {\em Duke Math. J.} {\bf 123} (2004), no.1, pp.~49-69.

\bibitem[Kat92]{Katok}  {\scshape S.~Katok } -- {\og Fuchsian Groups \fg}, 
\emph{Chicago Lectures in Math., Univ. of Chicago Press, Chicago} 1992.

\bibitem[KenS00]{KS}
{\scshape R.~Kenyon {\normalfont and} J.~Smillie } -- {\og Billiards
  in rational-angled triangles\fg}, \emph{Comment.\ Math.\ Helv.}  {\bf
  75} (2000), pp.~65--108.

\bibitem[KZ03]{Kontsevich2003}
{\scshape M.~Kontsevich, {\normalfont and} A.~Zorich } --
{\og Connected components of the moduli spaces of Abelian differentials with 
prescribed singularities\fg}, \emph{Invent. Math.} {\bf 153} (2003), no. 3, pp.631--678.

\bibitem[L04]{Lanneau04} {\scshape E.~Lanneau } -- {\og Hyperelliptic components 
of the moduli spaces of quadratic differentials with prescribed singularities\fg},
  \emph{Comment. Math. Helv.} {\bf 79} (2004), pp.~471--501.
  

\bibitem[LN13]{Lanneau:Manh:H4}  {\scshape E.~Lanneau {\normalfont
    and}    D.-M.~Nguyen    },    {\og    Teichm\"uller
curves generated by Weierstrass Prym eigenforms in genus three and genus four
   \fg},  \emph{Journal of Topology} (2013).

\bibitem[LN13a]{Lanneau:Manh:cp}  \bysame   {\og    Complete periodicity of Prym eigenforms \fg},  
\emph{arXiv:1301.0783} (2013).

\bibitem[LN13c]{Lanneau:Manh:composantes}  \bysame    {\og  Components of Prym eigenform loci in genus three
   \fg},  {\em in preparation}.


\bibitem[LM\"ol13]{Lanneau:Moeller}  {\scshape E.~Lanneau {\normalfont
    and}    M.~M\"oller   },   {\og    Finiteness of the number of primitive Teichm\"uller curves
in $\PrymD(1,1,2)$ \fg},  \emph{preprint} (2013).



\bibitem[MMY05]{Marmi:Moussa:Yoccoz}
{\scshape S.~Marmi, P.~Moussa {\normalfont and} J.-C.~Yoccoz } --
{\og The cohomological equation for Roth type interval exchange
  transformations\fg}, \emph{J. \ Amer. \ Math. \ Soc.}
{\bf 18} (2005), pp.~823--872.

\bibitem[M86]{Ma86} 
{\scshape H.~Masur} -- {\og Closed trajectories for quadratic differentials with an 
application to billiards\fg}, 
\emph{Duke Math. J.} {\bf 53} (1986), no.~2, pp.~307--314. 

\bibitem[MZ08]{Masur:Zorich} {\scshape H.~Masur {\normalfont and} A.~Zorich   } -- {\og Multiple Saddle Connections on flat Surfaces and Principal Boundary of the Moduli Spaces of Quadratic Differentials \fg}, {\em Geom. Funct. Anal.} {\bf 18} (2008), no.~3, 
pp.~919-987.

\bibitem[MT02]{Masur2002}
{\sc H. Masur, S. Tabachnikov}--{\og Rational billiards and flat structures\fg}, 
{\em Handbook of dynamical systems}, \textbf{1A}, \emph{North-Holland, Amsterdam} (2002), pp.~1015--1089.

\bibitem[MaWri13]{Matheus:Wright13} {\scshape C.~Matheus \normalfont{ and } A.~Wright} -- {\og Hodge-Teichm\"uller planes and finiteness results for Teich\"uller curves \fg}, {\em arXiv:1308.0832} (2013).


\bibitem[McM03a]{Mc1}  {\scshape  C.~McMullen}   --  {\og  Billiards  and   {Teichm\"uller}    curves   on    Hilbert    modular   surfaces\fg},
  \emph{J.\ Amer.\ Math.\ Soc.} {\bf 16} (2003), no.~4, pp.~857--885.

\bibitem[McM03b]{Mc2}  \bysame   ,  {\og  {Teichm\"uller}   geodesics  of
  infinite complexity\fg}, \emph{Acta Math.}  {\bf 191} (2003), no.~2,
  pp.~191--223.
  
\bibitem[McM05a]{Mc4} \bysame , {\og {Teichm\"uller} curves in genus two:
  Discriminant  and  spin\fg}, \emph{Math.\  Ann.}  {\bf 333}  (2005),
  pp.~87--130.

 \bibitem[McM05b]{Mc5} \bysame , {\og {Teichm\"uller} curves in genus two:
  The decagon and beyond\fg}, \emph{J.\ reine angew.\ Math.} {\bf 582}
  (2005), pp.~173--200.

\bibitem[McM06a]{Mc6bis} \bysame , {\og {Teichm\"uller} curves in genus two:
  Torsion divisors and ratios of sines\fg}, \emph{Invent.\ Math.} {\bf
  165} (2006), pp.~651--672.


\bibitem[McM06]{Mc6} \bysame  , {\og Prym  varieties and {Teichm\"uller}  curves \fg}, \emph{Duke Math. J.} {\bf 133} (2006), pp.~569--590.

\bibitem[McM07]{Mc07}\bysame ,{\og Dynamics of $\SL_2(\R)$ over the moduli space in genus two\fg}, \emph{Ann.\ of Math. (2)} {\bf 165} (2007), no.2, pp.~397-456.

\bibitem[M\"ol06]{Moller2006}  {\scshape M.~M\"oller}, {\og  Variations  of Hodge   structure of {Teichm\"uller} curves \fg}, {\em J. \ Amer. \ Math. \ Soc.}  {\bf 19}  (2006), no.~2, pp.~327--344.

\bibitem[M\"ol08]{Moller08}  {\scshape M.~M\"oller}, {\og Finiteness results for Teichm\"uller curves \fg}, {\em Ann. Inst. Fourier (Grenoble)}, {\bf 58:1} (2008) pp.~63--83.  

\bibitem[Sah]{Sah} {\scshape C.-H. Sah}  -- {\og Scissors congruences of the interval\fg}, \emph{Preprint} (1981).

\bibitem[SW06]{Smillie2006}  {\scshape  J.  Smillie {\normalfont  and}
  B. Weiss  } -- {\og Finiteness  results for flat  surfaces: a survey
  and   problem  list,  (2006)   in  Partially   hyperbolic  dynamics,
  laminations,  and  Teichmueller  flow  \fg} {\em  Proceedings  of  a
  conference,  Fields Institute,  Toronto Jan  2006), G.  Forni (ed.)}
  (2006).

\bibitem[SW07]{Smillie2007} {\scshape  J.  Smillie {\normalfont  and}   B. Weiss  } -- {\og Veech's dichotomy and the lattice property \fg}, {\em Erg. Th. Dyn. Sys.}, {\bf 28} (2008), no.6, pp.~1959-1972.

\bibitem[SW10]{Smillie2010} {\scshape  J.  Smillie {\normalfont  and}   B. Weiss  }--{\og Characterizations of lattice surfaces \fg}, {\em Invent. Math.} {\bf 180} (2010), no.~3, pp.~535-557.

\bibitem[Thu88]{Thurston1988} {\scshape  W.~Thurston } --  {\og On the
  geometry   and   dynamics   of   homeomorphisms   of   surfaces\fg},
  \emph{Bull.\ A.M.S.} {\bf 19} (1988), pp.~417--431.

\bibitem[Vee82]{Veech1982}
{\scshape W.~Veech } -- {\og Gauss measures for transformations on the space of interval exchange maps\fg}, 
\emph{Ann. of Math. (2)} {\bf 115} (1982), no.~1, pp.~201--242.

\bibitem[Vee89]{Veech1989}  \bysame,  {\og {Teichm\"uller} curves  in
  modular space,  Eisenstein series, and an  application to triangular
  billiards \fg}, {\em Invent. Math.} {\bf 97} (1989), no.~3,  pp.~553--583.

\bibitem[Wri13]{Wright2013}   {\scshape  A.~Wright } -- {\og Cylinder deformations in orbit closures
of translation surfaces \fg}, {\em arXiv:1302.4108} (2013).



\bibitem[Zor06]{Zorich:survey}  {\scshape  A.~Zorich  } --  {\og  Flat
  surfaces\fg},  \emph{Frontiers   in  number  theory,   physics,  and
  geometry} 437--583, Springer, Berlin, 2006.

\end{thebibliography}
\end{document}